\documentclass[12pt]{amsart}
\usepackage[dvips]{color}
\usepackage{amsmath}
\usepackage{amsxtra}
\usepackage{amscd}
\usepackage{amsthm}
\usepackage{amsfonts}
\usepackage{amssymb}
\usepackage{eucal}
\usepackage{epsfig}
\usepackage{graphics}
\usepackage{graphicx}
\usepackage{accents}
\usepackage{dynkin-diagrams}
\usepackage{comment}
\usepackage{enumitem}

\usepackage{mathtools}
\usepackage{tikz-cd}
\usepackage{mathtools}
\usepackage{pgf}
\usetikzlibrary{arrows, matrix}

\usepackage{here} 
\usepackage{simplewick}
\numberwithin{equation}{section}
\newtheorem{thm}{Theorem}[section]
\newtheorem{prop}[thm]{Proposition}
\newtheorem{lem}[thm]{Lemma}

\newtheorem{conj}[thm]{Conjecture}

\newcommand{\nn}{\nonumber}

\theoremstyle{definition}

\newcommand{\C}{{\mathbb C}}
\newcommand{\Z}{{\mathbb Z}}

\newcommand{\cW}{\mathcal{W}}
\newcommand{\Sb}{\mathfrak{S}}

\newcommand{\bs}{\boldsymbol}

\newcommand{\gl}{\mathfrak{gl}}

\newcommand{\sln}{\mathfrak{sl}}

\newcommand{\sss}{\textsf{s}}

\newcommand{\ssy}{\textsf{y}}

\newcommand{\res}{{\rm res}}
\newcommand{\Res}{\mathop{\rm res}}

\newcommand{\on}{\operatorname}
\newcommand{\mc}{\mathcal}

\newcommand{\cont}[2]{\contraction[1ex]{}{#1}{}{#2} #1 #2}

 \newcommand{\Trn}{\on{Trn}}
 
\begin{document}

\begin{title}[$qq$-characters and vertex operators]
{Combinatorics of vertex operators and deformed $W$-algebra of type D$(2,1;\alpha)$}
\end{title}
\author{B. Feigin, M. Jimbo, and E. Mukhin}
\address{BF: National Research University Higher School of Economics,  101000, Myasnitskaya ul. 20, Moscow,  Russia, and Landau Institute for Theoretical Physics, 142432, pr. Akademika Semenova 1a, Chernogolovka, Russia
}
\email{bfeigin@gmail.com}
\address{MJ: Department of Mathematics,
Rikkyo University, Toshima-ku, Tokyo 171-8501, Japan}
\email{jimbomm@rikkyo.ac.jp}
\address{EM: Department of Mathematics,
Indiana University Purdue University Indianapolis,
402 N. Blackford St., LD 270, 
Indianapolis, IN 46202, USA}
\email{emukhin@iupui.edu}


\begin{abstract} 
We consider sets of screening operators with fermionic screening currents. We study sums of vertex operators which formally commute with the screening operators assuming that each vertex operator has rational contractions with all screening currents with only simple poles. We develop and use the method of $qq$-characters which are combinatorial objects described in terms of deformed Cartan matrix. We show that each $qq$-character gives rise to a sum of vertex operators commuting with screening operators and describe ways to understand the sum in the case it is infinite. 

We discuss combinatorics of the $qq$-characters and
their relation to the $q$-characters of representations of quantum groups.  

We provide a number of explicit examples of the $qq$-characters with the emphasis on the case of D$(2,1;\alpha)$. We describe a relationship of the examples to various integrals of motion.
\end{abstract}


\maketitle

\section{Introduction}
It is widely acknowledged that $W$-algebras form a fundamental class of conformal field theories. A lot of information has been collected but the picture is far from being complete. It has long been known that $W$-algebras have nontrivial deformations,  \cite{AKOS}, \cite{FR1}. Recently the interest to these deformations was revived due to discovered connections to gauge theories and integrable systems, \cite{FKSW}, \cite{N}, \cite{KP1}, \cite{KP2}.

We are interested in the deformed $W$-algebras since they possess families of commuting operators closely connected to affine XXZ models, see \cite{FJM2}, \cite{FJMV}. 

As in the undeformed case, the deformed $W$-algebras are generated by currents commuting with a set of screening operators. 
The screening operators are integrals of screening currents. In the undeformed case, screening currents can be quite complicated, which makes it difficult to understand the mechanism of commutation with the $W$-currents. In the deformed case one observes some simplifications as the screening currents are sums of several vertex operators and the combinatorics of  commutation with $W$-currents is often easy to track. This paper is an attempt to understand the $W$-currents combinatorially.

For simplicity, we restrict ourselves to the case when each screening current is a fermionic current written as a single vertex operator. Moreover, we assume that the $W$-currents are sums of vertex operators, such that the contractions of each term with all screening currents are rational functions with at most simple poles. Then the commutator of $W$-current with a screening current is a sum of delta functions multiplied by vertex operators which, after integrating the delta functions out, have to cancel in pairs. Such a cancellation pattern corresponds to a combinatorial object called $qq$-characters. In some cases, $qq$-characters were observed in \cite{N}, \cite{KP1}, \cite{KP2}, \cite{FJMV}. In this paper we give a general definition of the $qq$-characters and study them.

The $qq$-characters in flavor are similar to the $q$-characters of level zero representations of quantum affine algebras, \cite{FR2}, \cite{FM}. The $q$-character of a module $V$ is a Laurent polynomial with non-negative integer coefficients essentially given by the formal sums of collections of rational functions which are eigenfunctions of Cartan currents $K_i(z)$ in the module. If the matrix element of the generating current $F_i(z)$ between two eigenvectors is non-trivial, then it is always given by the delta function $\delta(u/z)$ (and, in general, its derivatives) multiplied by a constant, where $u$ is a pole of the eigenvalue of $K_i(z)$ on both vectors. Then the corresponding eigenvalues are related by a simple factor $A_{i,u}^{-1}$ called ``affine root" which allows us to construct and to study the $q$-characters combinatorially.

Similarly, the $qq$-characters are Laurent polynomials with non-negative integer coefficients essentially given by the formal sums of collections of rational functions which are contractions of vertex operators with screening currents $S_i(z)$. If the commutator of $S_i(z)$ with two vertex operators contains delta functions which cancel after summing and taking the integral, then the corresponding contractions should also be related by a simple factor. This allows us to define analogs of affine roots and study $qq$-characters in a way similar to $q$-characters.

However, there are important differences. First, even in the case of \cite{FR1}, which is directly related to quantum affine algebras, one has an extra, ``elliptic" parameter which participates in the contraction of screening currents among themselves. This leads to a much larger set of examples of families of screening, described by deformed Cartan matrices, see Appendix C in \cite{FJMV}. 

Second, the pairwise cancellations of terms do not necessarily correspond to multiplication by one affine root, but also to special products (strings) of affine roots. We call such cancellation patterns ``blocks". The  option of having blocks gives a much larger set of $qq$-characters related to a given deformed Cartan matrix compared to the $q$-characters.

Third, since quantum affine algebras are Hopf algebras and the main part of comultiplication of $K_i(z)$ is  $K_i(z)\otimes K_i(z)$, the
$q$-characters can be multiplied as Laurent polynomials. But the $qq$-characters do not have this structure. Instead, there is a combinatorial fusion which corresponds to the fusion of currents.

Thus $qq$-characters and $q$-characters are different in general, and it is not clear if there is a conceptual explanation of the similarity of the combinatorics. 
We do show that some special class of $qq$-characters corresponds to $q$-characters, see Theorem \ref{module thm}. Such $qq$-characters are one-parameter deformations of the $q$-characters which is the origin of the name of the $qq$-characters.

\medskip

Our principal result is that given a $qq$-character one can construct a current given as a sum of vertex operators which  formally commutes with the screening operators, see Theorem \ref{prop:qq-coeff}. We expect that if a sum of vertex operators whose contractions with screening currents are rational functions with at most simple poles, commutes with the screening operators then it comes from a $qq$-character.

\medskip

With the general knowledge of $qq$-characters and Theorem \ref{prop:qq-coeff}, one can generate a large number of interesting examples. Many cases, in particular,  all affine-type examples involve infinite $qq$-characters and, therefore, infinite sums of vertex operators. We observe that such sums often have a periodicity property, giving rise to well-defined integrals which are honest operators commuting with the screenings.

The most important examples for us are the cases of three screening currents of D$(2,1;\alpha)$ type and its affine analog given by four screening currents, since they are connected to the $(\hat{\mathfrak{sl}}_2\times \hat{\mathfrak{sl}}_2)/\hat{\mathfrak{sl}}_2$ coset theory, see \cite{BL}, \cite{FJM1}.
In this case the simplest $qq$-characters correspond to the 6  vector representations (and 12 of them in the affine case) which are infinite-dimensional.

In \cite{FJM1} we revealed  three copies of  quantum toroidal $\gl_2$ algebra whose transfer matrices produce three commuting families of integrals of motion. We show that the first integrals in each family come from the $qq$-characters corresponding to the affine vector representations. We expect that other integrals can also be obtained that way.

For special values of the parameter $\alpha$, the deformed Cartan matrix of type D$(2,1;\alpha)$ becomes that of types  $\mathfrak{osp}_{4,2}$ or $\gl_{2,2}$ which have finite-dimensional vector representations of dimensions $6$ and $4$, respectively. We show that there exist two series of such resonances of parameters depending on $k\in\Z_{>0}$, for the first one we have a finite $qq$-character with $4k+2$ terms and for the second one we have two finite $qq$-characters with $4k$ terms. It would be interesting to understand the conformal limit of the corresponding deformed $W$-algebras. 

\medskip

In this paper we consider only ``tame" $qq$-characters where all screening currents are fermionic type with the same elliptic parameter. The structure of the $qq$-characters in the presence of bosonic type screenings is similar, but combinatorics is more intricate. Roughly speaking, this happens because the representation theory of $U_q\hat{\gl}_2$ is more complicated than that of $U_q\hat{\gl}_{1,1}$. We plan to address this issue in the future publications. 

At the moment, there are many unanswered questions. We do not have a complete list of deformed Cartan matrices which admit non-trivial finite tame $qq$-characters. We have no classification of the tame $qq$-characters even in the simplest case of $\gl_{2,1}$. The structure of non-tame $qq$-characters (which correspond to poles of higher order in the contractions) also deserves an additional study.

\medskip

The paper is constructed as follows. We start in Section \ref{sec: qq} with the combinatorial part describing the definition and properties of tame $qq$-characters. We discuss the algorithm of construction of tame $qq$-characters and the fusion procedure. Sections \ref{sec:examples} and \ref{sec:18,66,130} are devoted to various examples of $qq$-characters. In Section \ref{sec q char} we discuss the connection of $qq$-characters to $q$-characters. Section \ref{sec:freefield} is the vertex operator part of the paper.

\section{The $qq$-characters.}\label{sec: qq}
In this section we describe the combinatorics of $qq$-characters. We restrict ourselves to the case of fermionic roots of the same kind. The $qq$-characters in the bosonic situation were introduced in \cite{N}, \cite{KP1}, more general $qq$-characters appeared in \cite{FJMV}.

\subsection{The terminology}\label{sec terminology}
Let $q$ be a variable. We call it an elliptic variable. We also prepare a finite number of other independent variables $q_1,q_2,...,$ and work over the ring $R$ of Laurent polynomials in all variables with integer coefficients. A monomial is an element of $R$ of the form $q^a\prod_iq_i^{a_i}$, where $a,a_i\in\Z$. Note that in our convention the coefficient of a monomial is one, e.g. $2q$ is not a monomial. Then the set of all monomials is a multiplicative group inside $R$. 

\medskip

We start with a general deformed Cartan matrix of fermionic type. Let $I$ be a set of integers of cardinality $r$. Elements of $I$ will be referred to as ``colors".
We call a symmetric, non-degenerate, $r\times r$ matrix $C=(c_{ij})_{i,j\in I}$ with entries in $R$ a deformed Cartan matrix of fermionic type if all entries of $C$ are of the form $c_{ij}=\sigma_{ij}-\sigma_{ij}^{-1}$ where $\sigma_{ij}$ are monomials, and all diagonal entries are $c_{ii}=q-q^{-1}$.

In particular, we have $\sigma_{ij}=\sigma_{ji}$, $\sigma_{ii}=q$,  and  if $c_{ij}=0$ then $\sigma_{ij}=1$.

In this text we consider only deformed Cartan matrices of fermionic type; for brevity we call them simply Cartan matrices.

While it is not clear what a most general reasonable definition of a deformed Cartan matrix is, there are several studied classes. In \cite{FR2},  deformed Cartan matrices are associated to Dynkin diagrams. In \cite{KP2}, deformed Cartan matrices are associated to quivers. In \cite{FJMV}, deformed Cartan matrices are associated to a class of representations of some quantum toroidal algebras. The majority of the explicit examples of Cartan matrices in this paper are given by the construction of \cite{FJMV}, the examples of D$(2,1;\alpha)$ and $\hat{\textrm D}(2,1;\alpha)$  are related to \cite{FJM1}.
 
\medskip 

Next, we prepare some formal rings and language to work with them. The $qq$-characters will be elements of such rings.

Let $\mathcal Y$ be a ring of Laurent polynomials with integer coefficients in commutative formal variables $Y_{i,\sigma}$  where $i\in I,$ and $\sigma\in R$ is a monomial.\footnote{We follow notation of \cite{FJMV} which is different from the usual $q$-character notation. Variable $Y_i$ should be compared to variable $X_i^{-1}$ in \cite{FJMM1}, \cite{FJMM2} and usual $Y_i$ variables are ratios of two $X_i$ variables.}

A monomial in $\mathcal Y$ is a finite product of generators $Y_{i,\sigma}^{\pm1}$. Clearly, the set of all monomials is a multiplicative group in $\mathcal Y$.
For a monomial  $m$ and $\chi\in\mc Y$ we write $m\in \chi$ if the coefficient of $m$ in $\chi$ is non-zero.
 
For each $i\in I$ define a $\Z$-grading $\deg_i$ of $\mathcal Y$ by setting $\deg_i Y_{j,\sigma}^{\pm1}=\pm\delta_{ij}$. We write $\deg(m)=(\deg_i(m))_{i\in I}$ and call it degree of $m\in\mathcal Y$.

For a monomial $\sigma_0\in R$, define the ring automorphism 
$$
\tau_{\sigma_0}:\  \mathcal Y \to \mathcal Y, \qquad Y_{i,\sigma}^{\pm 1} \mapsto Y_{i,\sigma_0\sigma}^{\pm1}.
$$
 We call the map $\tau_{\sigma_0}$ the shift by $\sigma_0$.

For a set $J\subset I$, we have a subring ${\mathcal Y}_J\subset \mathcal Y$ generated  by $Y_{j,\sigma}^{\pm1}$  $(j\in J)$. 
Define the surjective ring homomorphism 
$$
\rho_J:\ \mathcal Y \to {\mathcal Y}_J , \qquad Y_{i,\sigma}^{\pm1}\mapsto \begin{cases}
Y_{i,\sigma}^{\pm1} & (i\in J), \\
1 &(i\not\in J).
\end{cases} 
$$
We call $\rho_J$ the restriction map. 

\medskip
Some $qq$-characters will be infinite sums. We continue with the description of the corresponding extension of $\mathcal Y$.

Let $\tilde {\mathcal Y}$ be the space of formal sums of countably many monomials in $\mathcal Y$.  We have the inclusion of spaces $\mathcal Y\subset \tilde {\mathcal Y}$. Clearly, the space $\tilde {\mathcal Y}$ is a $\mathcal Y$-module. 

The subspace ${\tilde {\mathcal Y}}_J\subset \tilde{\mathcal Y}$ is the space of of formal sums of countably many monomials in ${\mathcal Y}_J$.

The ring maps $\tau_{\sigma_0}$ and $\rho_J$ are  extended to maps of vector spaces  
$\tau_{\sigma_0}:\ \tilde{\mathcal Y} \to \tilde{\mathcal Y}$ and  $\rho_J:\ \tilde{\mathcal Y} \to \tilde{\mathcal Y}_J$. Note that we use the same notation for these maps.

\medskip

We often deal with products of elements of $\mc Y$ or even $\tilde{\mc Y}$ where all participating generators are distinct and therefore there are no cancellation or combining of generators. We call such products generic. Here are the formal definitions.

We call a monomial $m$ in $\mathcal Y$ generic if $m$ is a product of distinct generators $Y_{i,\sigma}^{\pm1}$. In other words, in a generic monomial any generator $Y_{i,\sigma}$ can appear only in powers $-1,0$, or $1$. We call an element $\chi=\sum_s m_s\in\tilde{\mc Y}$ generic if all monomials $m_s$ are generic. 

The condition that we have no positive powers more than 1 is equivalent to the assumption of simple poles in the contractions, or the property of being tame, and it will be essential. The condition that there are no negative powers  smaller than minus one is added for convenience only and will not be used. In all non-trivial examples we consider such powers do not appear.

Two monomials $m_1,m_2$ are called mutually generic, if both $m_1,m_2$ are generic, $m_1m_2$, $m_1/m_2$ are generic. In other words, variables present in $m_1$ do not appear in $m_2$. In particular,  
multiplying $m_1$  by $m_2$, we encounter no cancellations.

Two series $\chi_1,\chi_2\in\tilde{\mc Y}$ are mutually generic if every monomial in $\chi_1$ is mutually generic with every monomial in $\chi_2$.

Note that if $\chi_1,\chi_2\in \tilde{\mc Y}$ are mutually generic then the product $\chi_1\chi_2\in \tilde{\mc Y}$ is well defined and generic.

\subsection{The definition of tame $qq$-characters}
Tame $qq$-characters associated to a deformed Cartan matrix are elements of $\tilde{\mc Y}$ with special combinatorial properties which we now describe.

Given a deformed Cartan matrix $C=(\sigma_{ij}-\sigma_{ij}^{-1})_{i\in I}$, define the affine roots $A_{j}$, $j\in I$, by the formula:
$$
A_j=\prod_{i\in I} Y_{i,\sigma_{ij}} Y_{i,\sigma_{ij}^{-1}}^{-1}.
$$
For a monomial $\sigma\in R$, we set $A_{j,\sigma}=\tau_\sigma A_j$.  The affine roots $A_{j,\sigma}$ are generic monomials in $\mc Y$ of degree zero.

We assumed that the deformed Cartan matrix $C$ is non-degenerate, it implies that  the affine roots $A_{j,\sigma}$ are all algebraically independent.

\medskip

We define elementary blocks. 
An elementary block of color $i$ and length $k+1$ is an element $B^{(k)}_{i}\in \mc Y$ which has the following properties
\begin{itemize}
    \item  the block $B^{(k)}_{i}$ is a sum of $k+1$ monomials, $B^{(k)}_{i}=m_0+\dots+m_{k}$;
    \item the monomial $m_j$ has the form 
    $$
    m_j=\bar m\bar m_j\prod_{0\leq s\leq k \atop s\neq j}Y_{i,q^{-k+2s}}, 
     $$
     where $\bar m_j$ is a generic monomial in variables $Y_{s,\sigma}^{\pm1}$, $s\neq i$, and $\bar m$ is a generic monomial in variables $Y_{i,\sigma}^{-1}$, $\sigma\neq q^a$, $a\in\{-k,-k+2,\dots,k\}$;
     \item the monomials $m_j$ are connected by the affine roots of color $i$:
      $m_{j+1}= m_j A_{i,q^{-k+2j+1}}^{-1}$.
\end{itemize}

We also define shifted elementary blocks $B^{(k)}_{i,\sigma}=\tau_\sigma B^{(k)}_{i} \in \mc Y$.

We note that elementary blocks look similar to $q$-characters of $U_q\hat{\mathfrak{sl}}_2$ irreducible evaluation modules. However, we work with a principally different case of fermionic roots. It is well known that all irreducible finite-dimensional $U_q\hat{\mathfrak{gl}}_{1,1}$ modules are tensor products of several two-dimensional evaluation modules and a one-dimensional module. Only blocks  $B^{(1)}_{i,\sigma}$ of length 2 and trivial blocks $B^{(0)}_{i,\sigma}$  have such form. We will discuss a connection to quantum group $q$-characters in Section \ref{sec q char}.

We call the top monomial $m_0$ in the block $B^{(k)}_{i}$ the $i$-dominant monomial  and the bottom monomial $m_k$ the $i$-anti-dominant monomial. We call a variable $Y_{i,\sigma}$ in a monomial $m$ in a block $B^{(k)}_{i}$ $i$-dominant (resp. $i$-anti-dominant) if  $m A_{i,\sigma q^{-1}}^{-1}$ (resp.  $m A_{i,\sigma q}$) is also in the same block $B^{(k)}_{i}$.

We consider products of mutually generic shifted elementary blocks. This is analogous to (but not the same as) taking tensor products of $U_q\hat{\mathfrak{sl}}_2$ evaluation modules which remain irreducible and tame.

\medskip

We are finally ready to define $qq$-characters.
A series $\chi\in\tilde{\mc Y}$ is called a tame $qq$-character if for all $i\in I$, the series $\chi$ is a sum of products of mutually generic shifted elementary blocks of color $i$. All $qq$-characters in this paper are tame, so we will simply call them $qq$-characters.

Clearly, a $qq$-character is always generic.  A shift $\tau_\sigma \chi$ of a $qq$-character $\chi$ is clearly a $qq$-character. If $J\subset I$ is such that the corresponding Cartan submatrix is non-degenerate, then the restriction $\rho_J\chi$ 
of a $qq$-character $\chi$ is clearly a $qq$-character.

A $qq$-character is called finite if it is a sum of finitely many monomials.

A sum of $qq$-characters is a  $qq$-character.
A  $qq$-character is called simple if it is not a sum of two non-zero tame $qq$-characters. 

The constants $\chi=n$ $(n\in\Z_{\geq 0})$ are trivially tame $qq$-characters.  More generally, any generic series $\chi\in\tilde{\mc Y}$ with non-negative coefficients containing only $Y_{i,\sigma}^{-1}$ is a tame $qq$-character. We call such $qq$-characters polynomial.

A product of two mutually generic $qq$-characters is a $qq$-character. A  degree zero $qq$-character $\chi$ is called prime if it is not a product of two degree zero $qq$-characters.

We call a  $qq$-character slim if it has degree zero and for each $i$ and for all occurring shifted elementary blocks $B_{i,\sigma}^{(k)}$ the length $k+1$ is at most $2$. Slim $qq$-characters are to be compared to $q$-characters, see Section \ref{sec q char}. Non-slim characters do not correspond to $q$-characters.

We call a monomial $m\in\mc Y$ $i$-linear if $\rho_{\{i\}}m= Y_{i,\sigma_1}Y_{i,\sigma_2}^{-1}$ for some (not necessarily distinct) monomials $\sigma_1,\sigma_2\in R$.  A monomial $m$ is linear if it is $i$-linear for all $i\in I$. We call a  $qq$-character $\chi$ linear if all monomials in $\chi$ are linear.
Linear $qq$-characters are slim.

\medskip

In general, a generic $\chi\in\tilde{\mc Y}$ can be written as a sum of products of mutually generic shifted elementary blocks of color $i$ in several ways. For example, it happens when the same monomial occurs several times. However, we expect that it does not occur for simple characters. We now prove the uniqueness under some technical assumption which is sufficient for our purposes.

We say monomial $m$ is $i$-connected to monomial $n$ if $n=mA_{i,\sigma}^{-1}$, $m$ contains $Y_{i,\sigma q}$, and $n$ contains $Y_{i,\sigma q^{-1}}$. 

Thus in a shifted elementary block $B_{i,\sigma}^{(k)}=m_0+\dots+m_k$, the monomial $m_j$ is connected to monomial $m_{j+1}$, $j=0,\dots, k-1$.

We call a set of monomials $m_j$, $j\in\Z$, an infinite chain of color $i$ if for all $j\in \Z$,
$m_j$ is $i$-connected to $m_{j+1}$.

\begin{lem}\label{unique}
Let $\chi$ be a $qq$-character. Assume that all monomials in $\chi$ are distinct. Assume further that $\chi$ has no infinite chains of color $i$. Then it can be written as a sum of  products of mutually generic shifted elementary blocks of color $i$ in the unique way.
\end{lem}
\begin{proof}
Because of the assumptions, there exists a monomial $m\in\chi$ such that either $m$ is not $i$-connected to any other monomial in $\chi$ or no other monomial in $\chi$ $i$-connected to $m$. In the first case $m$ must be a product of $i$-dominant monomials and in the second case $m$ must be a product of $i$-anti-dominant monomials. In both cases, there is a uniquely determined product of blocks of color $i$ which has to be present in $\chi$. Subtracting this product and continuing to find such monomial $m$ in the remaining sum of monomials, we obtain the lemma.
\end{proof}
Note that since $A_{i,\sigma}$ are algebraically independent, we cannot have loops: if monomial $m_j$ is connected to monomial $m_{j+1}$ for $j=1,\dots, k$, then $m_k$ is not connected to $m_1$. Thus, any finite $qq$-character with distinct monomials satisfies the assumptions of the lemma.

\medskip

The next definitions do depend on the way the $qq$-character is written as a sum of products of mutually generic shifted elementary blocks of color $i$.  

If $\chi$ is a $qq$-character, then we call a monomial in $\chi$ $i$-dominant (resp. $i$-anti-dominant) if it is a product of $i$-dominant  (resp. $i$-anti-dominant) monomials in the blocks of color $i$. We call a monomial dominant (resp. anti-dominant) if it is $i$-dominant for all $i\in I$. We call a variable $Y_{i,\sigma}$ in a monomial $m\in\chi$  $i$-dominant (resp. $i$-anti-dominant) if it is $i$-dominant (resp. $i$-anti-dominant) in at least one of the blocks of color $i$.

\medskip

The $qq$-characters are visualized via its graphs. The graph of a $qq$-character $\chi$ is a colored directed graph whose vertices are monomials $m\in\chi$. There is an edge of color $i$ from a monomial $m$ to monomial $m'$, if $m$ and $m'$ belong to the same product of blocks of color $i$, if $m'=mA_{i,\sigma}^{-1}$ and if $m$ contains dominant variable  $Y_{i,\sigma q}$ while $m'$ contains anti-dominant variable $Y_{i,\sigma q^{-1}}$.  We  denote this situation by $m{\xrightarrow {i,\sigma}} m'$. 

A monomial is $i$-dominant if and only if in the graph there are no incoming edges of color $i$. A monomial is $i$-anti-dominant if and only if in the graph there are no outgoing edges of color $i$.

A $qq$-character $\chi$ is linear if and only if in the graph, for each $i\in I$ and each monomial $m\in \chi$ there is at most one edge of color $i$ with vertex $m$. (This edge can be incoming or outgoing.) 

\medskip

Clearly, every connected component of a graph of a  $qq$-character is a graph of a $qq$-character. Thus,
a $qq$-character is simple if and only if all of its graphs are connected. 

Let us repeat that, in general, a $qq$-character can have several graphs associated to it. We expect that the graph of a simple $qq$-character is unique. Due to Lemma \ref{unique} this is the case in all examples we consider in this text.

\subsection{The algorithm of constructing tame $qq$-characters.}\label{sec algorithm}
Simple degree zero $qq$-characters are rigid objects and can often be reconstructed from just one monomial. The algorithm is similar to the one used for $q$-characters of quantum affine algebras, see \cite{FM}.
On one hand it is somewhat simpler, since we are in the tame situation. On the other hand it is complicated by the absence of a good concept of dominant monomials, since we are in the superalgebra situation. Namely, we can say which monomial is $i$-dominant after the $qq$-character is constructed, but not before (as it was in the non-super case).

Every finite  $qq$-character has a dominant and an anti-dominant monomial. Every finite $qq$-character with a unique dominant (or anti-dominant) monomial is simple.

Suppose we have a generic monomial $m_+$ and would like to find a simple $qq$-character $\chi\in{\tilde {\mc Y}}$ such  that $m_+\in\chi$ and such that $m_+$ is a unique dominant monomial. Our algorithm starts with $\chi=m_+$ where all occurring $Y_{i,\sigma}$ in $m_+$ (in positive power) are called unmarked.

If $m_+$ contains no $Y_{i,\sigma}$, $\chi=m_+$ is a simple tame polynomial $qq$-character. 
Otherwise, we choose a maximal string of (unmarked) variables in $m_+$ of the form $Y_{i,\sigma}Y_{i,q^{2}\sigma}\dots Y_{i,q^{2k-2}\sigma}$. Here the word maximal means we have no (unmarked) $Y_{i,q^{-2}\sigma},Y_{i,q^{2k}\sigma}$ entering $m_+$. 

If $Y_{i,q^{-2}\sigma}^{-1}$ is in $m_+$, the algorithm fails, meaning no such $qq$-character exists. Otherwise,
we set $m_0=m_+$, and add to $\chi$ monomials $m_1,\dots,m_k$ so that we obtain a block of color $i$ of length $k+1$. For example,  $m_1=A_{i,q^{-1}\sigma}^{-1}m_0$ contains $Y_{i,q^{-2}\sigma}Y_{i,q^2\sigma}\dots Y_{i,q^{2k-2}\sigma}$, $m_2=A_{i,q\sigma}^{-1}m_1$ contains $Y_{i,q^{-2}\sigma}Y_{i,\sigma}Y_{i,q^4\sigma}\dots Y_{i,q^{2k-2}\sigma}$, etc.

In the monomials $m_0,m_1,\dots,m_k$
we mark all positive powers $Y_{i,q^{2j}\sigma}$, $j=-1,\dots,k-1$, and call all other new positive powers $Y_{s,\sigma_{is}^{-1}\sigma q^{2j+1}}$,  $s\neq i,$ in $m_1,\dots,m_k$ unmarked.

We call this process the expansion of a string in the $i$-th direction. 

If any of the monomials $m_1,\dots, m_k$ is not generic, the algorithm fails. Otherwise,
we continue in the same way. Namely, we choose an unmarked maximal string in any of the monomials (note that the marked generators are ignored) and expand it. 

In the process we follow two rules. 
First rule is that if one of the monomials we add during expansion already exists in $\chi$ and the positive powers participating in the expansion are all unmarked, then we mark them and do not add this monomial for the second time. 

Second rule is that we expand in the order of depth. Note that for any monomial $m'\in\chi$, $m_+/m'$ is a product of shifted affine roots. Since the affine roots are algebraically independent, the way to write $m_+/m'$  as a product of shifted affine roots is unique. We say $m'$ has depth $k$ if $m_+/m'$ is a product of $k$ shifted affine roots. The dominant monomial $m_+$ has depth zero. We expand it first. Then we expand all generated monomials of depth $1$, then of depth $2$ and so on.

Following these two rules, 
we proceed with the expansions until no unmarked positive powers is left.

Then, it is clear, that the algorithm either fails or produces a simple  $qq$-character with a unique dominant monomial $m_+$. 

We also note that the affine roots have degree zero. Therefore for all monomials $m'$ in the result we have $\deg_i m_+=\deg_i m_+'$,  $i\in I$.

\medskip

One can use another version of the algorithm, declaring  the initial monomial $m$ to be anti-dominant and expanding the strings in the other direction. Moreover, one can do a mixture: declare that the initial monomial $m$ is $i$-dominant, $i\in J$, where  $J\subset I$  is a subset of colors, and $i$-anti-dominant for $i\not\in J$. In such a way, we will be able to obtain infinite $qq$-characters  which have neither dominant nor anti-dominant monomials. 

\subsection{Truncation of $qq$-characters.}\label{sec truncation}
We describe a procedure which, given a $qq$-character, allows to produce a $qq$-character with a smaller number of terms.  We will use this procedure in Section \ref{sec 21}.

Suppose we have a $qq$-character $\chi$ obtained by the algorithm from a dominant monomial $m_+$. Let $m\in \chi$ be a monomial which was obtained after several expansions and which had an unmarked positive power $Y_{i,\sigma}$ when first obtained. 

Consider now the dominant monomial $m_+ Y_{i,\sigma}^{-1}$ and apply the algorithm. Then it proceeds the same way as the algorithm applied to $m_+$. But when we arrive at the monomial $m$, the unmarked positive power is cancelled and we do not do that expansion anymore. Therefore, the new $qq$-character will have less terms compared to $\chi$, it can be obtained from $\chi Y_{i,\sigma}^{-1}$ by dropping the appropriate terms. We call this $qq$-character the truncation of $\chi Y_{i,\sigma}^{-1}$ and denote it by $\Trn(\chi Y_{i,\sigma}^{-1})$.

\medskip

The truncation procedure is an analog of the  construction of finite type modules which are obtained by multiplying known modules by polynomial modules and taking the irreducible submodule,  see \cite{FJMM1}, \cite{FJMM2}. The finite type modules have properties  similar to finite-dimensional ones, but  they are in general infinite-dimensional.

Similarly, the truncation produces valid $qq$-characters, but there is a price to pay: the truncation changes the degree of the $qq$-character and, in general, one gets a $qq$-character of a non-zero degree.

\subsection{Combinatorial fusion}
Multiplication of  $qq$-characters often produces the same truncation phenomenon. A non-generic product of $qq$-characters $\chi_1\chi_2$ is often not a $qq$-character, it may be not even a well-defined element of $\tilde {\mc Y}$. However, if there exist mutually generic monomials $m_1\in\chi_1$, $m_2\in\chi_2$,  then all monomials of the $qq$-character generated by the product $m_1m_2$ are in $\chi_1\chi_2$. 

For example, let $r=1$, $I=\{1\}$. Then $B^{(1)}_1=Y_{1,q}+Y_{1,q^{-1}}$ is a block of color one and of length 2. Then
$B^{(1)}_{1,q^{-1}} B^{(1)}_{1,q}$ is not a generic product but it contains the block of length three $B^{(2)}_1=Y_{1,q^2}Y_{1,1}+Y_{1,q^{2}}Y_{1,q^{-2}}+ Y_{1,1}Y_{1,q^{-2}}$: we have
$$
B^{(1)}_{1,q^{-1}}B^{(1)}_{1,q}=B^{(2)}_1+Y_{1,1}^2.
$$
Note that there is no complementary $qq$-character at all as $Y_{1,1}^2$ is not generic.

We call such products truncated.

We describe a combinatorial procedure which allows us to do the truncation of products of $qq$-characters without invoking the algorithm. This procedure originates in the study of the fusion of currents and, therefore, we call it combinatorial fusion.

For each $i\in I$, and a monomial $\sigma\in R$, define group homomorphisms $l_{i,\sigma}$, $r_{i,\sigma}$ sending monomials in $\mc Y$ to $R$ considered as an additive group by the rule:
\begin{align*}
l_{i,\sigma}(Y_{j,\tau}) \mapsto \delta_{ij}(q-q^{-1})\tau\sigma^{-1},  \\
r_{i,\sigma}(Y_{j,\tau}) \mapsto -\delta_{ij} (q-q^{-1})\tau^{-1}\sigma. 
\end{align*}

We call the homomorphisms $l_{i,\sigma}$, $r_{i,\sigma}$ the  combinatorial left and right contractions with the affine root $A_{i,\sigma}$. They are to be compared with \eqref{AY},  see Section \ref{sec:freefield} below.

We note, cf. \eqref{AA},
$$
l_{i,\tau_1} (A_{j,\tau_2})=r_{j,\tau_2}(A_{i,\tau_1})=(q-q^{-1})(\sigma_{ij}-\sigma_{ij}^{-1})\tau_2\tau_1^{-1}.
$$

Let $m,n\in \mc Y$ be two monomials. Assume that they have the form 
\begin{align}\label{mA}
m=m_0\prod_{j}A_{i_j,\sigma_j}^{a_j}, \qquad n=n_0\prod_{j}A_{i_j,\sigma_j}^{b_j},
\end{align}
where the product is over some finite set of indices and $a_j,b_j\in\Z$.  Define the relative combinatorial contraction:
$$
\frac{[m,n]}{[m_0,n_0]}=\sum_{j}( a_j l_{i_j,\sigma_j}(n)+b_jr_{i_j,\sigma_j}(m_0))=\sum_{j}( a_j l_{i_j,\sigma_j}(n_0)+b_jr_{i_j,\sigma_j}(m)).
$$
In particular, we have the following simple properties:
\begin{align}\label{prop of contr}
\frac{[m,n]}{[m,n]}=0, \qquad \frac{[m,n]}{[m',n']}+\frac{[m',n']}{[m,n]}=0, \qquad
\frac{[m,n]}{[m',n']}+\frac{[m',n']}{[m'',n'']}=\frac{[m,n]}{[m'',n'']}.
\end{align}
Then we define the relative pairing 
$$
(m,n)_{m_0,n_0}=\left(\frac{[m,n]}{[m_0,n_0]}\right)_0\in\Z,
$$
where, for a Laurent polynomial $p\in R$, we denote by $p_0$ its constant term. 

Let us now have two simple $qq$-characters $\chi_1,\chi_2\in\mc Y$. Choose $m_0\in\chi_1$, $n_0\in\chi_2$. Then any $m\in\chi_1$, $n\in\chi_2$ have form \eqref{mA}. 

We define the combinatorial fusion of $\chi_1$ and $\chi_2$ by multiplying the two and keeping only the terms with maximal pairing:
$$
\chi_1 * \chi_2= \sum_{\substack{m\in\chi_1, n\in\chi_2,\\ (m,n)_{m_0,n_0}=M}} mn, \qquad M=\max_{m\in\chi_1, \ n\in\chi_2}{\{(m,n)_{m_0,n_0}\}}.
$$

Clearly, the definition of $\chi_1 *\chi_2$ does not depend on the choice of $m_0,n_0$, see \eqref{prop of contr}. Moreover it is  commutative as we clearly have
$$
(m,n)_{m_0,n_0}=(n,m)_{n_0,m_0}.
$$
We use the combinatorial fusion to construct non-trivial $qq$-characters.

For example, consider the case $r=1$ when there is only one color. We have
\begin{align*}
& (Y_{1,1}+Y_{1,q^2})*(Y_{1,q^2}+Y_{1,q^4})= Y_{1,1}Y_{1,q^2}+Y_{1,1}Y_{1,q^4}+Y_{1,q^2}Y_{1,q^4}.
\end{align*}
It is easy to see that all blocks of length $k+1$ can be obtained as multiple combinatorial fusion of blocks of length $2$.

However, the combinatorial fusion should be used with care. 
For example, in the case $r=1$, we have $(Y_{1,1}+Y_{1,q^2})* (Y_{1,1}+Y_{1,q^2})=2Y_{1,1}Y_{1,q^2}$. 
In fact, the correct answer here should be  $(Y_{1,1}^2+(Y_{1,1}Y_{1,q^2})'+Y_{1,q^2}^2)$ as the main terms in two copies of $Y_{1,1}Y_{1,q^2}$ actually cancel and one has to consider ``the derivative" and bring back the other terms. Such an example is not tame though.

\begin{conj}
Let $\chi_1,\chi_2$ be simple $qq$-characters. Suppose the fusion product $\chi_1*\chi_2$ is generic and has all non-zero coefficients one. Then $\chi_1*\chi_2$ is a $qq$-character.
\end{conj}

If $\chi_1,\chi_2$ are mutually generic, then 
we have $(m,n)_{m_0,n_0}=0$, for any $m\in\chi_1$, $n\in\chi_2$. Indeed, if $A_{i,\tau}$ is present in the expression \eqref{mA} for $m$, then some monomial of $\chi_1$ contains $Y_{i,\tau q}$ and some monomial in $\chi_1$ contains $Y_{i,\tau q^{-1}}$. It follows that $\chi_2$ does not contain these monomials and, therefore, $A_{i,\tau}$  makes no contribution to pairing $(m,n)_{m_0,n_0}$.

In particular, if $\chi_1,\chi_2$ are mutually generic, then the fusion product coincides with the usual product:
$\chi_1*\chi_2=\chi_2*\chi_1=\chi_1\chi_2$.

\section{Examples of  $qq$-characters.}\label{sec:examples}
For a randomly chosen Cartan matrix, there are no finite degree zero $qq$-characters. It seems that every non-trivial example is interesting. 

We often use the notation
\begin{align*}
Y_{1,\sigma}={\bf 1}_\sigma, \quad Y_{2,\sigma}={\bf 2}_\sigma, \quad Y_{1,\sigma}^{-1}={\bf 1}^\sigma, \quad 
Y_{1,\sigma_1}Y_{1,\sigma_2}Y_{1,\sigma_3}^{-1}={\bf 1}^{\sigma_3}_{\sigma_1,\sigma_2}, \qquad Y_{i+1,\sigma_1}Y_{i+1,\sigma_2}^{-1}={(\bs{i + 1})}^{\sigma_2}_{\sigma_1},
\end{align*}
and so on.

\subsection{The case of $\mathfrak{gl}_{2,1}$}\label{sec 21}
We work with two independent variables $q$ and $q_1$. We set $q_2=q^{-1}q_1^{-1}$, and $p=q^2 q_1^2=q_2^{-2}$.
Let $I=\{1,2\}$ and 
$$
C=\begin{pmatrix}
q-q^{-1} & q_1-q_1^{-1} \\
q_1-q_1^{-1} & q-q^{-1}
\end{pmatrix}.
$$
We have
$$
A_1=\bs 1_q^{q^{-1}} \bs 2_{q_1}^{q_1^{-1}}, \qquad  A_2=\bs 1_{q_1}^{q_1^{-1}} \bs 2_{q}^{q^{-1}}.
$$
We call $qq$-characters corresponding to this Cartan matrix $qq$-characters of  $\mathfrak{gl}_{2,1}$ type.

Note that we have a natural symmetry exchanging colors: $\bs 1_\sigma \leftrightarrow \bs 2_\sigma$. Given a family  of $qq$-characters one can produce more $qq$-characters by shifting, taking generic products, and exchanging colors.

The case of $\mathfrak{gl}_{2,1}$ is fundamental for us, because for any deformed Cartan matrix of fermionic type, any restriction $\rho_J$ with $|J|=2$ gives either the trivial case of two non-interacting fermions or the case of $\gl_{2,1}$ with the appropriately chosen $q_1$. The case of $\mathfrak{gl}_{2,1}$  is also the simplest one and it is convenient to illustrate our methods with.

\medskip

We start with the dominant degree $(1,0)$ monomial $V_0=\bs 1_q$. Then we expand it and get monomial $V_1=A^{-1}_{1,1}V_0=\bs 1_{q^{-1}}{\bs 2}^{q_1}_{q_1^{-1}}$. Note that $\bs 1_{q^{-1}}$ is marked but 
$\bs 2_{q_1^{-1}}$ is not.
It creates the need to expand $V_1$ in color $2$ and we get a monomial of depth two,  $V_2= A^{-1}_{2,q_2}V_1= \bs 1_{qq_2^2}\bs 2_{q_1q_2^2}^{q_1}$.  We again get an unmarked $\bs 1_{qq_2^2}$ which we expand and get $V_3=\bs 1_{q^{-1}q_2^2}\bs 2_{q_1^{-1}q_2^2}^{q_1}$ which in its turn needs to be expanded in color 2. And so on. As the result we obtain an infinite linear $qq$-character of degree $(1,0)$ which we call $\chi_1^+$. 

We have
\begin{align}\label{chi+}
\chi_1^{+}=\sum_{i=0}^\infty (\bs 1_{qp^{-i}}\bs 2_{q_1p^{-i}}^{q_1} +\bs 1_{q^{-1}p^{-i}} \bs 2_{q_1^{-1}p^{-i}}^{q_1}).
\end{align}

We now start with the anti-dominant monomial $\bs 1_{q^{-1}}$. Expanding, we obtain an infinite 
linear $qq$-character of degree $(1,0)$ which we call $\chi_1^-$. It can be obtained by changing in  \eqref{chi+}, $q,q_1,p,q_2$ to $q^{-1},q_1^{-1},p^{-1},q_2^{-1}$. 

Similarly, starting from dominant monomial $\bs 2_q$ and anti-dominant monomial $\bs 2_{q^{-1}}$ we obtain the infinite linear $qq$-characters of degree $(0,1)$ which we call $\chi_2^+$ and $\chi_2^-$. The characters $\chi_2^\pm$ are obtained from $\chi_1^\pm$ by exchanging $\bs 1 \leftrightarrow \bs 2$.

We call all these $qq$-characters (and their shifts)  half-lines.

Next, we consider the monomial $\bs 1_q \bs 2_{q_1}$. We look at it as $1$-dominant and $2$-anti-dominant. Expanding, we obtain an infinite 
linear $qq$-character of degree $(1,1)$ which we call $\chi^{+,-}$. We have  
\begin{align}\label{chi+-}
\chi^{+,-}=\sum_{i\in\Z} (\bs 1_{qp^{-i}}\bs 2_{q_1p^{-i}} +\bs 1_{q^{-1}p^{-i}} \bs 2_{q_1^{-1}p^{-i}}).
\end{align}
Note that $\chi^{+,-}$ is periodic: $\tau_{p} \chi^{+,-}=\chi^{+,-}$. Also note that the change
$\bs 1 \leftrightarrow \bs 2$ in $\chi^{+,-}$ gives $\tau_{q_2} \chi^{+,-}$.
We call the $qq$-character $\chi^{+,-}$ and (the shifts of $\chi^{+,-}$ ) the line.

\medskip

Now we are ready to construct slim characters.

First, we have slim  linear prime $qq$-characters obtained by generic products $\chi_1^\pm\bs 1^{\sigma}$, $\chi_2^\pm\bs 2^{\sigma}$ and 
$\chi^{+,-}\bs 1^{\sigma_1}\bs 2^{\sigma_2}$. 
We call them degree zero half-lines and lines respectively. 

Next we use the truncation, see Section \ref{sec truncation}. Consider the $qq$-character $\chi_1^+ \bs 1^\sigma$. For general $\sigma$, the product is generic and therefore it is a slim infinite $qq$-character. However, for $\sigma=qp^{-n}$, where $n\in\Z_{\geq 0}$, we have a cancellation and a truncation. We obtain a finite linear $qq$-character with $2n+1$ terms which we denote by $\chi_1^{2n+1}=\Trn(\chi_1^+ \bs 1^{qp^{-n}})$ and call a degree zero segment.

\begin{align}\label{chi n}
\chi_1^{2n+1}=\sum_{i=0}^{n-1} (\bs 1_{qp^{-i}}^{qp^{-n}}\bs 2_{q_1p^{-i}}^{q_1} +\bs 1_{q^{-1}p^{-i}}^{qp^{-n}} \bs 2_{q_1^{-1}p^{-i}}^{q_1})+\bs 2_{q_1p^{-n}}^{q_1}.
\end{align}
We also have a linear $qq$-character $\chi_2^{2n+1}$ obtained either by truncation of $\chi_2^{+}$ or by replacing $\bs 1 \leftrightarrow \bs 2$ in $\chi_1^{2n+1}$.

To obtain a $qq$-character with an even number of terms we need to truncate $\chi_1^+\bs 2^{q_1^{-1}p^{-n}}$. This linear $qq$-character  has $2n+2$ terms and degree $(1,-1)$ which we call a segment. Similarly, we have a truncation of $\chi_2^+\bs 1^{q_1^{-1}p^{-m}}$ of degree $(-1,1)$. Making a shift by some $\kappa$ and multiplying we obtain a slim prime $qq$-character with $2n\times 2m$ terms. We denote this character by $\chi^{2n,2m}$ and call a prime rectangle. The dominant monomial of  $\chi^{2n,2m}$ is $m_+^{2n,2m}=\bs 1_{q}^{\kappa q_1^{-1}p^{1-m}}\bs 2_{\kappa q}^{ q_1^{-1}p^{1-n}}$.

The prime rectangles have the form $\chi_{12}^{2n,2m}=\sum_{a=0}^{2n-1}\sum_{b=0}^{2m-1} V_{a,b}$, where
\begin{align*}
&V_{2k,2\ell}=\bs 1_{qp^{-k},\, \kappa q_1p^{-\ell} }^{\sigma_1,\,\kappa q_1} \bs 2_{\kappa qp^{-\ell},\, q_1p^{-k} }^{\sigma_2,\,q_1}, \\
&V_{2k+1,2\ell+1}=\bs 1_{q^{-1}p^{-k},\, \kappa q_1^{-1}p^{-\ell} }^{\sigma_1,\,\kappa q_1}  \bs 2_{\kappa q^{-1}p^{-\ell},\, q_1^{-1}p^{-k} }^{\sigma_2,\,q_1}, \\
&V_{2k+1,2\ell}=\bs 1_{q^{-1}p^{-k},\, \kappa q_1p^{-\ell} }^{\sigma_1,\,\kappa q_1} \bs  2_{\kappa qp^{-\ell},\, q_1^{-1}qp^{-k} }^{\sigma_2,\,q_1}, \\
&V_{2k,2\ell+1}=\bs 1_{qp^{-k},\, \kappa q_1^{-1}p^{-\ell} }^{\sigma_1,\,\kappa q_1} \bs 2_{\kappa q^{-1}p^{-\ell},\, q_1p^{-k} }^{\sigma_2,\, q_1}.
\end{align*}
where $\sigma_1=\kappa q^{-1} p^{1-m}$, $\sigma_2= q^{-1}p^{1-n}$, and $\kappa$ is sufficiently general to avoid any cancellations.

The graph of a prime rectangle with highest monomial $m_+^{6,4}=\bs 1_q^{\kappa q^{-1}p^{-1}}\bs 2_{q\kappa}^{q^{-1}p^{-2}}$ is shown in Figure \ref{rectangle pic}. Note that this $qq$-character is slim but not linear.

\begin{figure} [H] 
\begin{center}
\begin{tikzpicture}[scale=0.53]
\node at  (-12,3.5) {\small $V_{0,0}$};
\node at  (-8,3.5)  {\small $V_{0,1}$};
\node at  (-4,3.5)  {\small $V_{0,2}$};
\node at  (0,3.5)  {\small $V_{0,3}$};
\node at  (4,3.5)  {\small $V_{0,4}$};
\node at  (8,3.5)  {\small $V_{0,5}$};

\node at  (-12,0) {\small $V_{1,0}$};
\node at  (-8,0)  {\small $V_{1,1}$};
\node at  (-4,0)  {\small $V_{1,2}$};
\node at  (0,0)  {\small $V_{1,3}$};
\node at  (4,0)  {\small $V_{1,4}$};
\node at  (8,0)  {\small $V_{1,5}$};

\node at  (-12,-3.5) {\small $V_{2,0}$};
\node at  (-8,-3.5)  {\small $V_{2,1}$};
\node at  (-4,-3.5)  {\small $V_{2,2}$};
\node at  (0,-3.5)  {\small $V_{2,3}$};
\node at  (4,-3.5)  {\small $V_{2,4}$};
\node at  (8,-3.5)  {\small $V_{2,5}$};

\node at  (-12,-7) {\small $V_{3,0}$};
\node at  (-8,-7)  {\small $V_{3,1}$};
\node at  (-4,-7)  {\small $V_{3,2}$};
\node at  (0,-7)  {\small $V_{3,3}$};
\node at  (4,-7)  {\small $V_{3,4}$};
\node at  (8,-7)  {\small $V_{3,5}$};

\draw[blue, ->] (-11,3.5)--node[below]{{\small $A_{1,1}^{-1}$}} (-8.7,3.5);
\draw[red, ->] (-7,3.5)-- node[below]{{\small $A_{2,q_2}^{-1}$}}(-4.7,3.5);
\draw[blue,->] (-3,3.5)-- node[below]{{\small $A_{1,q_2^2}^{-1}$}}(-0.7,3.5);
\draw[red, ->]  (1,3.5)--node[below]{{\small $A_{2,q_2^3}^{-1}$}}(3.3,3.5);
\draw[blue, ->] (5,3.5)--node[below]{{\small $A_{1,q_2^4}^{-1}$}}(7.3,3.5);

\draw[blue, ->] (-11,0)--node[below]{{\small $A_{1,1}^{-1}$}} (-8.7,0);
\draw[red, ->] (-7,0)-- node[below]{{\small $A_{2,q_2}^{-1}$}}(-4.7,0);
\draw[blue,->] (-3,0)-- node[below]{{\small $A_{1,q_2^2}^{-1}$}}(-0.7,0);
\draw[red, ->]  (1,0)--node[below]{{\small $A_{2,q_2^3}^{-1}$}}(3.3,0);
\draw[blue, ->] (5,0)--node[below]{{\small $A_{1,q_2^4}^{-1}$}}(7.3,0);

\draw[blue, ->] (-11,-3.5)--node[below]{{\small $A_{1,1}^{-1}$}} (-8.7,-3.5);
\draw[red, ->] (-7,-3.5)-- node[below]{{\small $A_{2,q_2}^{-1}$}}(-4.7,-3.5);
\draw[blue,->] (-3,-3.5)-- node[below]{{\small $A_{1,q_2^2}^{-1}$}}(-0.7,-3.5);
\draw[red, ->]  (1,-3.5)--node[below]{{\small $A_{2,q_2^3}^{-1}$}}(3.3,-3.5);
\draw[blue, ->] (5,-3.5)--node[below]{{\small $A_{1,q_2^4}^{-1}$}}(7.3,-3.5);

\draw[blue, ->] (-11,-7)--node[below]{{\small $A_{1,\kappa }^{-1}$}} (-8.7,-7);
\draw[red, ->] (-7,-7)-- node[below]{{\small $A_{2,\kappa q_2}^{-1}$}}(-4.7,-7);
\draw[blue,->] (-3,-7)-- node[below]{{\small $A_{1,\kappa q_2^2}^{-1}$}}(-0.7,-7);
\draw[red, ->]  (1,-7)--node[below]{{\small $A_{2,q_2^3}^{-1}$}}(3.3,-7);
\draw[blue, ->] (5,-7)--node[below]{{\small $A_{1,\kappa q_2^4}^{-1}$}}(7.3,-7);

\draw[red, ->] (-12,3.0)--node[right] {{\small $A_{2,\kappa }^{-1}$}} (-12,0.7);
\draw[blue, ->] (-12,-0.5)--node[right ]{{\small $A_{1,\kappa q_2}^{-1}$}} (-12,-2.8);
\draw[red, ->] (-12,-4.0)--node[right ]{{\small $A_{2,\kappa q_2^2}^{-1}$}} (-12,-6.3);

\draw[red, ->] (-8,3.0)--node[right ]{{\small $A_{2,\kappa }^{-1}$}} (-8,0.7);
\draw[blue, ->] (-8,-0.5)--node[right ]{{\small $A_{1,\kappa q_2}^{-1}$}} (-8,-2.8);
\draw[red, ->] (-8,-4.0)--node[right ]{{\small $A_{2,\kappa q_2^2}^{-1}$}} (-8,-6.3);

\draw[red, ->] (4,3.0)--node[right ]{{\small $A_{2,\kappa }^{-1}$}} (4,0.7);
\draw[blue, ->] (4,-0.5)--node[right ]{{\small $A_{1,\kappa q_2}^{-1}$}} (4,-2.8);
\draw[red, ->] (4,-4.0)--node[right ]{{\small $A_{2,\kappa q_2^2}^{-1}$}} (4,-6.3);

\draw[red, ->] (8,3.0)--node[right ]{{\small $A_{2,\kappa }^{-1}$}} (8,0.7);
\draw[blue, ->] (8,-0.5)--node[right ]{{\small $A_{1,\kappa q_2}^{-1}$}} (8,-2.8);
\draw[red, ->] (8,-4.0)--node[right ]{{\small $A_{2,\kappa q_2^2}^{-1}$}} (8,-6.3);

\draw[red, ->] (-4,3.0)--node[right ]{{\small $A_{2,\kappa }^{-1}$}} (-4,0.7);
\draw[blue, ->] (-4,-0.5)--node[right ]{{\small $A_{1,\kappa q_2}^{-1}$}} (-4,-2.8);
\draw[red, ->] (-4,-4.0)--node[right ]{{\small $A_{2,\kappa q_2^2}^{-1}$}} (-4,-6.3);

\draw[red, ->] (0,3.0)--node[right ]{{\small $A_{2,\kappa }^{-1}$}} (0,0.7);
\draw[blue, ->] (0,-0.5)--node[right ]{{\small $A_{1,\kappa q_2}^{-1}$}} (0,-2.8);
\draw[red, ->] (0,-4.0)--node[right ]{{\small $A_{2,\kappa q_2^2}^{-1}$}} (0,-6.3);

\end{tikzpicture}
\end{center}
\caption{The $4\times 6$ rectangle $\gl_{2,1}$ $qq$-character.}\label{rectangle pic}
\end{figure}
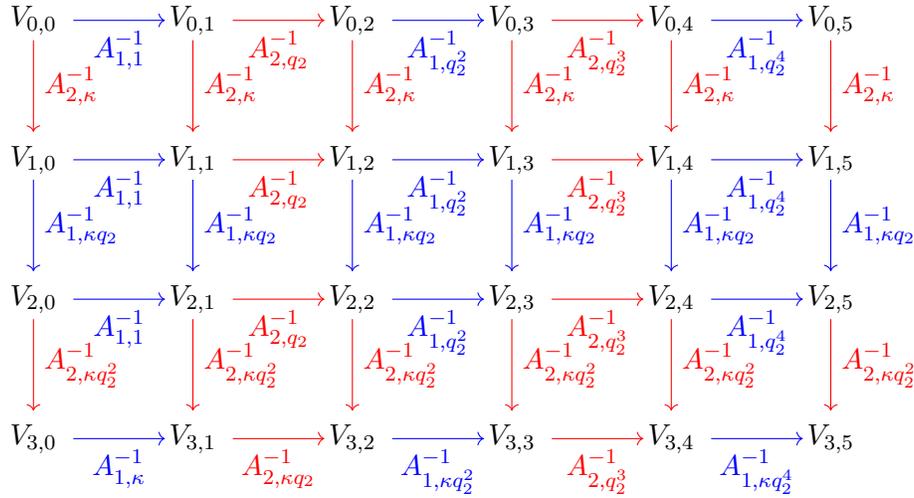

Finally, we can truncate products (with general enough $\kappa$):
\begin{align*}
   & \chi_{12}^{\pm,2n}=\Trn\left(\chi_1^\pm  \tau_\kappa (\chi_2^+) \, \bs 1^{\kappa q^{-1}p^{1-n}}\right), \qquad  &\chi_{21}^{\pm,2n}=\Trn\left(\chi_2^\pm \tau_\kappa (\chi_1^+) \,\bs 2^{\kappa q^{-1}p^{1-n}}\right),\qquad \\
 &\chi_2^{+,-,2n}=\Trn\left(\chi^{+,-} \tau_\kappa (\chi_2^+) \,\bs 1^{\kappa q^{-1}p^{n-1}}\right), \qquad
 &\chi_2^{+,-,2n}=\Trn\left(\chi^{+,-} \tau_\kappa (\chi_1^+) \bs 2^{\kappa q^{-1}p^{n-1}}\right).\ \,
\end{align*}
The results are slim prime infinite $qq$-characters
which we call prime strips. Multiplying by factors $\bs 2^{\sigma}$,$\bs 1^{\sigma}$ with general enough $\sigma$  to make the factors mutually generic,
we obtain degree zero $qq$-characters.
We call these slim prime $qq$-characters degree zero prime strips.

It turns out that we have constructed all slim prime $\gl_{2,1}$ $qq$-characters up to a shift.

\begin{prop}
The prime slim  $qq$-characters of $\gl_{2,1}$ type are either degree zero half-lines, lines, segments, prime strips, or prime rectangles. 

The linear  $qq$-characters of $\gl_{2,1}$ are either degree zero half-lines, lines,  segments, or squares.
\end{prop}
\begin{proof}
Let $\chi$ be a slim $qq$-character and $m\in\chi$. Let $m=m_1/m_2$ where both $m_1$ and $m_2$ contain no inverses. Without loss of generality, we assume that $m$ is such that $m_2$ contains the smallest possible number of variables. Classify the variables appearing in $m_1$ (i.e. positive powers in $m$) as follows. We have $\bs 1_{\sigma}$ (resp. $\bs 2_{\sigma}$) which are  $1$-dominant (resp. $2$-dominant) in their blocks  and the ones which are $1$-anti-dominant (resp. $2$-anti-dominant). We place dominant $\bs 1_{q\sigma}$ in pairs with anti-dominant $\bs 2_{q_1 \sigma}$ when such pairs exist. Similarly we place dominant $\bs 2_{q\sigma}$ in pairs with anti-dominant $\bs 1_{q_1 \sigma}$ when possible. As the result we have several unpaired dominant variables, several unpaired anti-dominant variables and several pairs. 

Generate from each unpaired variable $Y_{i,\sigma}$ a half-line $qq$-character of degree $(1,0)$ or $(0,1)$ and from each pair a line $qq$-character of degree $(1,1)$. Consider the product $\hat \chi$ of all these $qq$-characters. This product does not have to be generic, so it does  not have to be a $qq$-character. However, all monomials in $\chi$ are obtained from this product by multiplying by $m_2^{-1}$ and truncating the result. Note that every half line creates an inverse of a variable which we call a new inverse.

It is convenient to think that monomials in $\hat \chi$ are labeled by the set of integer points in a multi-dimensional simplex, formed 
by the Cartesian product of all participating lines and half lines.  

Now we consider the first possible truncation on each of the edges of this simplex. Thus we take $m$ and expand it in the direction of one of the half lines until we arrive at the first instance when we generate a monomial $m'$ with unmarked $\bs 1_\sigma$ or $\bs 2_\sigma$ which is cancelled by $m_2$.

Note that such a cancellation cannot truncate a line since, in such a case, $m'$ would have less variables than $m$. Therefore each such cancellation truncates a
half-line to a segment. Then $\hat \chi$ truncates to a sum with monomials labeled by integer points in a multidimensional parallelepiped (with possibly some infinite sides).

We claim that there are no further truncation. Indeed, the only possibility would be a new inverse. Suppose we have a new $\bs 1^\sigma$ produced by a half line starting from a dominant monomial. Then the dominant monomial of that half-line is $n=\bs 2_{q q_1^{-1}\sigma}$.  Then we should also have a $\bs 1_\sigma$ produced on the edge of our parallelepiped. Expanding this $\bs 1_\sigma$ (still on the edge), we obtain $\bs 2_{q^{-1}q_1^{-1}\sigma}$. If this positive power does not cancel, multiplying by $n$ leads to a block of length two and a contradiction since the initial $qq$-character was slim. If this positive power cancels, then the $m_2^{-1}$ contains $\bs 2^{q^{-1}q_1^{-1}\sigma}$ and $n$ cannot be expanded to start with, which is also a contradiction.
 
Finally, if we have more than 2 different half-lines or segments or lines then $\chi$ is not prime.
If one of the segments is a degree zero segment, then it is just a factor of $\chi$. If we have segments of degrees $(1,-1)$ and $(-1,1)$ then their product is a prime rectangle which splits as a factor. If we have segments of degrees $(1,-1)$ only, then we should have either a half-line of degree $(-1,0)$ and at least one monomial $\bs 2_\sigma$ which is a common factor or a full line and at least two common monomials of the form $\bs 2_\sigma$. Then the product of this monomial (or monomials), of the segment and of the half-line (or line) is a factor.

If there are no segments then the $qq$-character can be prime only if there is only one half line or line.
\end{proof}

An example of non-slim degree zero $qq$-characters is given by non-slim squares. 
It starts with the top monomial $U_{0,0}=\bs 1_{q,q^3,\dots,q^{2k-1}}^{\kappa q_1^{-1},\kappa q_1^{-1}q^{2},\dots,\kappa q_1^{-1} q^{2k-2}} \bs 2_{\kappa q,\kappa q^3,\dots, \kappa q^{2k-1}}^{q_1^{-1},q_1^{-1}q^{2},\dots, q_1^{-1} q^{2k-2}}$. The monomial $U_{0,0}$ has degree zero. Each step of the algorithm does not create any new positive powers. Thus the result is a $k\times k$ square which we now describe.

For $a,b=0,1,\dots,k-1$, we define
$$
U_{a,b}=\prod_{i=0}^{a-1} \bs 1_{q^{2i-1}}\bs 2^{q_1q^{2i}} \prod_{i=a}^{k-2} \bs 1_{q^{2i+1}}\bs 2^{q_1^{-1}q^{2i}} \prod_{i=0}^{b-1} \bs 2_{\kappa q^{2i-1}}\bs 1^{\kappa q_1q^{2i}} \prod_{i=b}^{k-2} \bs 2_{\kappa q^{2i+1}}\bs 1^{\kappa q_1^{-1}q^{2i}}.
$$

Then $\chi_k=\sum_{a,b=0}^{k-1}U_{a,b}$ is a degree zero non-slim $qq$-character. The graph of  character $\chi_4$ is given in Figure \ref{non-slim square pic}.

\begin{figure} [H] 
\begin{center}
\begin{tikzpicture}[scale=0.53]
\node at  (-12,3.5) {\small $U_{0,0}$};
\node at  (-8,3.5)  {\small $U_{0,1}$};
\node at  (-4,3.5)  {\small $U_{0,2}$};
\node at  (0,3.5)  {\small $U_{0,3}$};

\node at  (-12,0) {\small $U_{1,0}$};
\node at  (-8,0)  {\small $U_{1,1}$};
\node at  (-4,0)  {\small $U_{1,2}$};
\node at  (0,0)  {\small $U_{1,3}$};

\node at  (-12,-3.5) {\small $U_{2,0}$};
\node at  (-8,-3.5)  {\small $U_{2,1}$};
\node at  (-4,-3.5)  {\small $U_{2,2}$};
\node at  (0,-3.5)  {\small $U_{2,3}$};

\node at  (-12,-7) {\small $U_{3,0}$};
\node at  (-8,-7)  {\small $U_{3,1}$};
\node at  (-4,-7)  {\small $U_{3,2}$};
\node at  (0,-7)  {\small $U_{3,3}$};

\draw[blue, ->] (-11,3.5)--node[below]{{\small $A_{1,1}^{-1}$}} (-8.7,3.5);
\draw[blue, ->] (-7,3.5)-- node[below]{{\small $A_{1,q^2}^{-1}$}}(-4.7,3.5);
\draw[blue,->] (-3,3.5)-- node[below]{{\small $A_{1,q^4}^{-1}$}}(-0.7,3.5);

\draw[blue, ->] (-11,0)--node[below]{{\small $A_{1,1}^{-1}$}} (-8.7,0);
\draw[blue, ->] (-7,0)-- node[below]{{\small $A_{1,q^2}^{-1}$}}(-4.7,0);
\draw[blue,->] (-3,0)-- node[below]{{\small $A_{1,q^4}^{-1}$}}(-0.7,0);

\draw[blue, ->] (-11,-3.5)--node[below]{{\small $A_{1,1}^{-1}$}} (-8.7,-3.5);
\draw[blue, ->] (-7,-3.5)-- node[below]{{\small $A_{1,q^2}^{-1}$}}(-4.7,-3.5);
\draw[blue,->] (-3,-3.5)-- node[below]{{\small $A_{1,q^4}^{-1}$}}(-0.7,-3.5);

\draw[blue, ->] (-11,-7)--node[below]{{\small $A_{1,1}^{-1}$}} (-8.7,-7);
\draw[blue, ->] (-7,-7)-- node[below]{{\small $A_{2,q^2}^{-1}$}}(-4.7,-7);
\draw[blue,->] (-3,-7)-- node[below]{{\small $A_{1,q^4}^{-1}$}}(-0.7,-7);

\draw[red, ->] (-12,3.0)--node[right] {{\small $A_{2,\kappa}^{-1}$}} (-12,0.7);
\draw[red, ->] (-12,-0.5)--node[right ]{{\small $A_{1,\kappa q^2}^{-1}$}} (-12,-2.8);
\draw[red, ->] (-12,-4.0)--node[right ]{{\small $A_{2,\kappa q^4}^{-1}$}} (-12,-6.3);

\draw[red, ->] (-8,3.0)--node[right ]{{\small $A_{2,\kappa }^{-1}$}} (-8,0.7);
\draw[red, ->] (-8,-0.5)--node[right ]{{\small $A_{1,\kappa q^2}^{-1}$}} (-8,-2.8);
\draw[red, ->] (-8,-4.0)--node[right ]{{\small $A_{2,\kappa q^4}^{-1}$}} (-8,-6.3);

\draw[red, ->] (-4,3.0)--node[right ]{{\small $A_{2,\kappa }^{-1}$}} (-4,0.7);
\draw[red, ->] (-4,-0.5)--node[right ]{{\small $A_{1,\kappa q^2}^{-1}$}} (-4,-2.8);
\draw[red, ->] (-4,-4.0)--node[right ]{{\small $A_{2,\kappa q^4}^{-1}$}} (-4,-6.3);

\draw[red, ->] (0,3.0)--node[right ]{{\small $A_{2,\kappa }^{-1}$}} (0,0.7);
\draw[red, ->] (0,-0.5)--node[right ]{{\small $A_{1,\kappa q^2}^{-1}$}} (0,-2.8);
\draw[red, ->] (0,-4.0)--node[right ]{{\small $A_{2,\kappa q^4}^{-1}$}} (0,-6.3);

\end{tikzpicture}
\end{center}
\caption{The $4\times 4$ non-slim square $\gl_{2,1}$ $qq$-character.}\label{non-slim square pic}
\end{figure}
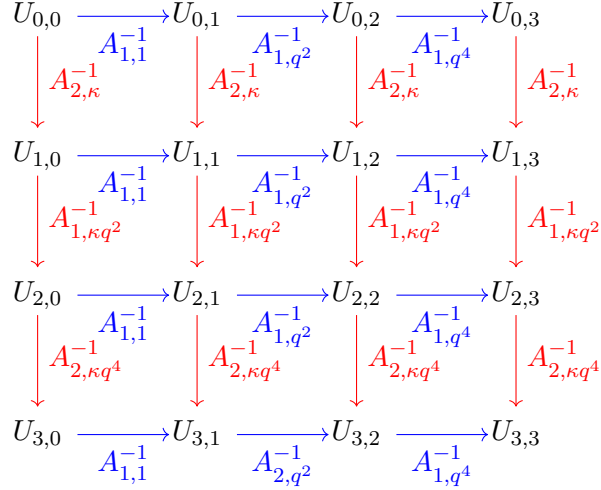

 Several more non-slim degree zero  $\gl_{2,1}$ $qq$-characters are given in Figures \ref{36 pic} and \ref{25 2 pic}.

\subsection{The cases of $\mathfrak{gl}_{n,n}$, $\mathfrak{gl}_{n+1,n}$, $\hat{\mathfrak{gl}}_{n,n}$.}  \label{sec gl}
We again work with two independent variables $q$ and $q_1$. We set $q_2=q^{-1}q_1^{-1}$.

Let $I=\{1,\dots,r\}$. Let  $C$ be an $r\times r$ matrix with all zero entries except
\begin{align}\label{n|n Cartan}
c_{ii}=q-q^{-1}, \qquad c_{2k-1,2k}=c_{2k,2k-1}=q_1-q_1^{-1}, \qquad c_{2k+1,2k}=c_{2k,2k+1}=q_2-q_2^{-1}.
\end{align}
We say that $C$ is of type $\mathfrak{gl}_{n,n}$ if $r=2n-1$ and of type $\mathfrak{gl}_{n+1,n}$ if $r=2n$.

Then for $i\in I$ we have
\begin{align}\label{n|n A}
    A_{i}= \begin{cases}  (\bs{i-1})_{q_2}^{q_2^{-1}} \, ({\bs i})_q^{q^{-1}}\,  (\bs{i+1})_{q_1}^{q_1^{-1}} & \qquad (i=2k-1),\\
    (\bs{i-1})_{q_1}^{q_1^{-1}}\,  ({\bs i})_q^{q^{-1}}\, (\bs{i+1})_{q_2}^{q_2^{-1}} &\qquad (i=2k),
        \end{cases}
\end{align}
where  by convention $\bs 0_\sigma=(\bs{r+1})_\sigma=1$.

We start with the dominant degree zero monomial $m_+=\bs 1_{q}^{qq_2^2}$ and apply the algorithm. The result is a slim $qq$-character with $r+1$ terms which we now explain.

For $i=0,\dots,r$, we set
\begin{align}\label{n|n V}
V_i= \begin{cases} (\bs{i})_{q_2q^{-k}}^{q_1q^{-k+1}}\, ({\bs {i+1}})_{q^{-k+1}}^{q_2^2q^{-k+1}}  & \qquad (i=2k), \\
& \\
(\bs{i})_{q^{-k-1}}^{q_2^2q^{-k+1}}\, ({\bs {i+1}})_{q_1^{-1}q^{-k}}^{q_1q^{-k}}  & \qquad  (i=2k+1). 
\end{cases}
\end{align}
Then 
$$
 V_{2k+1}=A^{-1}_{2k+1,q^{-k}} V_{2k}, \qquad 
V_{2k+2}=A^{-1}_{2k+2,q_2q^{-k}} V_{2k+1}.
$$
Thus $\chi=\sum_{i=0}^{r}V_i$ is a slim  linear $qq$-character which we call the vector $qq$-character of $\gl_{n,n}$ type if $r=2n-1$ and of $\gl_{n+1,n}$ type if $r=2n$.

The graph of $qq$-character $\chi$ is shown in Figure \ref{A vector pic}.

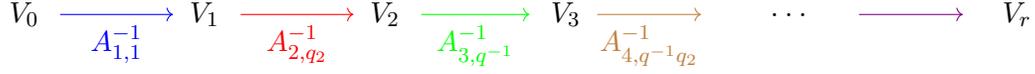
\begin{figure} [H] 
\begin{center}
\begin{tikzpicture}[scale=0.6]
\node at  (-12,-4) {\small $V_0$};
\node at  (-8,-4)  {\small $V_1$};
\node at  (-4,-4)  {\small $V_2$};
\node at  (0,-4)  {\small $V_3$};
\node at  (5,-4)  {$\dots$};
\node at  (8,-4)  {};
\node at  (10,-4)  {\small $V_r$};
\draw[blue, ->] (-11.2,-4)--node[below]{{\small $A_{1,1}^{-1}$}} (-8.7,-4);
\draw[red, ->] (-7.2,-4)-- node[below]{{\small $A_{2,q_2}^{-1}$}}(-4.7,-4);
\draw[green,->] (-3.2,-4)-- node[below]{{\small $A_{3,q^{-1}}^{-1}$}}(-0.8,-4);
\draw[brown, ->]  (0.7,-4)--node[below]{{\small $A_{4,q^{-1}q_2}^{-1}$}}(3,-4);
\draw[violet, ->] (6.5,-4)--(8.8,-4);
\end{tikzpicture}
\end{center}
\caption{The vector $\gl_{n,n}$ and  $\gl_{n+1,n}$ $qq$-characters.}\label{A vector pic}
\end{figure}

\medskip

Similarly, one constructs another $qq$-character $\chi^\vee$ with $r+1$ terms starting with dominant monomial 
$m_+=(\bs r)_q^{qq_2^2}$ if $r=2n$ and  $m_+=(\bs r)_q^{qq_1^2}$ if $r=2n-1$.

\medskip

Now we consider the affinization of the Cartan matrix adding one more color $0$. We do it only in the case $r=2n-1$. We have $\hat I=\{0,1,\dots,r\}\supset I$. The Cartan matrix elements $c_{ij}$ are given by the same equation \eqref{n|n Cartan}  where all indices are taken modulo $r+1$.\footnote{It is known that in this case one can introduce one extra independent parameter $q_3$ without changing the structure of $qq$-characters, see \cite{FJMV}, however, then $c_{0,1}=(q_2-q_2^{-1})q_3$ does not have the postulated form $\sigma-\sigma^{-1}$. Therefore we do not cover it here.}

The affine roots $A_i$ have form \eqref{n|n A}  where all indices are taken modulo $r+1$.

As far as we know in this case we do not have finite slim $qq$-characters. 
But there are infinite ones. We start with the monomial $\hat V_0=\bs 0^{q_2^{-1}}_{q_2}\bs 1_q^{qq^2}$. We declare it to be $1$-dominant and $0$-anti-dominant.  We now describe the resulting $qq$-character. 

Let $\hat V_i$ $(i\in\Z)$ be given by formula \eqref{n|n V} where ${\bs i}$ is taken modulo $r+1$. Set $\hat V_{i,\sigma}=\tau_\sigma (\hat V_i).$

We note the periodicity
$$
\hat V_{i+r+1}=\hat V_{i,Q},\qquad Q=q^{-n}.
$$

Then 
$$
\hat \chi=\sum_{i\in\Z} \hat V_i=\sum_{j\in \Z}\sum_{i=0}^{r}\hat V_{i,Q^j}
$$ 
is a slim $qq$-character which we call the vector $qq$-character of $\hat \gl_{n,n}$ type.

The graph of $\hat\chi$ is given in Figure \ref{mm hat vector}.

\begin{figure} [H] 
\begin{center}
\begin{tikzpicture}[scale=0.6]
\node at  (-14,-4)  {$\dots$};
\node at  (-12,-4) {\small $\hat V_{r,Q^{-1}}$};
\node at  (-8.5,-4)  {\small $\hat V_0$};
\node at  (-5,-4)  {\small $\dots$};
\node at  (-1.5,-4)  {\small $\hat V_r$};
\node at  (2,-4)  {\small $\hat V_{0,Q}$};
\node at  (5.5,-4)  {\small $\dots$};
\node at  (9,-4)  {\small $\hat V_{r,Q}$};
\node at  (12.5,-4)  {\small $\dots$};
\draw[ ->] (-11.0,-4)--node[below]{{\small $A_{0,qq_2}^{-1}$}} (-9,-4);
\draw[blue, ->] (-7.8,-4)-- node[below]{{\small $A_{1,1}^{-1}$}}(-5.7,-4);
\draw[violet,->] (-4.2,-4)-- node[below]{{\small $A_{r,qQ}^{-1}$}}(-2.0,-4);
\draw[ ->]  (-1,-4)--node[below]{{\small $A_{0,qq_2Q}^{-1}$}}(1.2,-4);
\draw[blue, ->] (2.9,-4)-- node[below]{{\small $A_{1,Q}^{-1}$}}(4.9,-4);
\draw[violet, ->] (6.3,-4)--node[below]{{\small $A_{r,qQ^2}^{-1}$}}(8.3,-4);
\draw[ ->] (9.8,-4)--node[below]{{\small $A_{r,qq_2Q^2}^{-1}$}}(11.8,-4);
\end{tikzpicture}
\end{center}
\caption{The vector $\hat \gl_{n,n}$ $qq$-character.}\label{mm hat vector}
\end{figure}
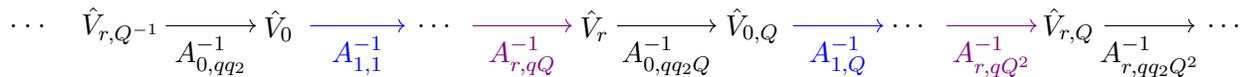

In particular, the restriction to the non-affine part has the form
$\rho_I(\hat \chi)=\sum_{i\in\Z} \tau_{Q^i}\chi$.

Similarly, one can construct another slim linear $qq$-character $\hat{\chi}^\vee$ starting from monomial $\bs r_q^{qq_1^2}\bs 0^{q_1^{-1}}_{q_1}$ which is $r$-dominant and $0$-anti-dominant.

\subsection{The cases of $\mathfrak{osp}_{2n,2n}$, $\mathfrak{osp}_{2n+2,2n}$, $\hat {\mathfrak{osp}}_{2n,2n}$, $\hat {\mathfrak{osp}}_{2n+2,2n}$  }\label{sec osp}

We stay with two independent variables $q$ and $q_1$. We set $q_2=q^{-1}q_1^{-1}$.

Let $I=\{1,\dots,r\}$, $r\geq 3$. The Cartan matrix is the same as  \eqref{n|n Cartan} except for $c_{ir}$ and $c_{ri}$ with $i=r-1,r-2$. For these elements we have
\begin{align}\label{n|n osp Cartan}
c_{r,r-2}=c_{r-1,r-2}, \qquad c_{r,r-1}=c_{r-1,r}= (-1)^r(q_1 q_2^{-1}-q_2q_1^{-1}).
\end{align}
Then the bottom corner $3\times 3$ submatrix of the Cartan matrix is
$$
\begin{pmatrix}
q-q^{-1} & q_1-q_1^{-1}& q_1-q_1^{-1} \\
q_1-q_1^{-1} & q-q^{-1} &q_1^{-1} q_2-q_1q_2^{-1}   \\
q_1-q_1^{-1} & q_1^{-1} q_2-q_1q_2^{-1} & q-q^{-1} 
\end{pmatrix} \ \ {\text{or}}\ \  
\begin{pmatrix}
q-q^{-1} & q_2-q_2^{-1}& q_2-q_2^{-1} \\
q_2-q_2^{-1} & q-q^{-1} &q_1 q_2^{-1}-q_1^{-1}q_2   \\
q_2-q_2^{-1} &q_1 q_2^{-1}-q_1^{-1}q_2  & q-q^{-1} 
\end{pmatrix},
$$
 where the first matrix corresponds to the case of odd $r$ and the second one to the case of even $r$. 

We say that $C$ is of type $\mathfrak{osp}_{2n+2,2n}$ if $r=2n+1$ and of type $\mathfrak{osp}_{2n,2n}$ if $r=2n$.

Then $A_i$, $i=1,2,\dots,r-3$, are given by \eqref{n|n A}. In addition if $r=2n+1$ then
\begin{align*}
&A_{r-2}= (\bs{r-2})_q^{q^{-1}} (\bs{r-3})_{q_2}^{q_2^{-1}}  (\bs{r-1})_{q_1}^{q_1^{-1}} (\bs{r})_{q_1}^{q_1^{-1}}, \\
&A_{r-1}= (\bs{r-1})_q^{q^{-1}}(\bs{r-2})_{q_1}^{q_1^{-1}} (\bs{r})_{q_1q_2^{-1}}^{q_1^{-1}q_2}, \\
&A_{r}= (\bs{r})_q^{q^{-1}}(\bs{r-2})_{q_1}^{q_1^{-1}} (\bs{r-1})^{q_1q_2^{-1}}_{q_1^{-1}q_2}. 
\end{align*}
The formulas for $A_{r-2}$, $A_{r-1}$, $A_r$ when $r=2n$ are obtained by exchanging $q_2$ and $q_1$.

We have a slim $qq$-character with $2r$ terms starting with dominant monomial $V_0=\bs 1_q^{qq_2^2}$, 
$\chi=V_0+V_1+\dots + V_{r-1}+V_{\overline{r-1}}+ \dots +V_{\overline{1}}+V_{\overline 0}$.

We give the formulas for the case of $r=2n+1$.
The monomials $V_i$ with $i=1,\dots,r-3$ are given by \eqref{n|n V}. In addition
\begin{equation}
\begin{aligned}\label{osp V r}
& V_{r-2}= (\bs{r-2})_{q^{-n}}^{q_2^2q^{-n+2}}
(\bs{r-1})^{q_1q^{-n+1}}_{q_1^{-1}q^{-n+1}}
(\bs{r})^{q_1q^{-n+1}}_{q_1^{-1}q^{-n+1}}, \\
& V_{r-1}=
(\bs{r-1})^{q_1q^{-n+1}}_{q_1^{-1}q^{-n-1}}
(\bs{r})^{q_2^3q^{-n+2}}_{q_1^{-1}q^{-n+1}}, \\
&
V_{\overline{r-1}}=
(\bs{r})^{q_1q^{-n+1}}_{q_1^{-1}q^{-n-1}}
(\bs{r-1})^{q_2^3q^{-n+2}}_{q_1^{-1}q^{-n+1}},\\ 
& V_{\overline{r-2}}= (\bs{r-2})^{q^{-n}}_{q_2^2q^{-n+2}}
(\bs{r-1})^{q_2^3q^{-n+2}}_{q_2q^{-n}}
(\bs{r})^{q_2^3q^{-n+2}}_{q_2q^{-n}}, \\
\end{aligned}
\end{equation}
and finally
\begin{align}\label{osp V overline}
V_{\overline{i}}= \begin{cases} (\bs{i})_{q_2q^{-2n+k}}^{q_2^3q^{-2n+k}}\, ({\bs {i+1}})^{q^{-2n+k-1}}_{q_2^2q^{-2n+k-1}}  & \qquad (i=2k), \\
& \\
(\bs{i})^{q^{-2n+k-1}}_{q_2^2q^{-2n+k+1}}\, ({\bs {i+1}})_{q_2q^{-2n+k-1}}^{q_2^3q^{-2n+k+1}}  & \qquad  (i=2k+1). 
\end{cases}
\end{align}

The graph of this $qq$-character is given in Figure \ref{osp vector}.

\begin{figure} [H] 
\begin{center}
\begin{tikzpicture}[scale=0.5]
\node at  (-12,-4) {\small $V_0$};
\node at  (-8,-4)  {\small $V_1$};
\node at  (-6,-4)  { $\dots$};
\node at  (-4,-4)  {\small $V_{r-3}$};
\node at  (0,-4)  {\small $V_{r-2}$};
\node at  (4,-1)  {\small $V_{r-1}$};
\node at  (4,-7)  {\small $V_{\overline{r-1}}$};
\node at  (8,-4)  {\small $V_{\overline{r-2}}$};

\node at  (14,-4)  {$\dots$};
\node at  (10,-4)  {};
\node at  (12,-4)  {\small $V_{\overline{r-3}}$};
\node at  (16,-4)  {\small $ V_{\overline{1}}$};
\node at  (20,-4)  {\small $ V_{\overline{0}}$};

\draw[blue, ->] (-11.2,-4)--node[below]{{\tiny $A_{1,1}^{-1}$}} (-8.7,-4);
\draw[magenta,->] (-3.1,-4)-- node[below]{{\tiny $A_{r-2,q^{-n+1}}^{-1}$}}(-0.8,-4);
\draw[brown, ->]  (0.7,-3.5)--node[left]{{\tiny $A_{r-1,q^{-n+1}q_2}^{-1}\ $}}(3,-1);
\draw[violet, ->]  (0.7,-4.5)--node[left]{{\tiny $A_{r,q^{-n+1}q_2}^{-1}$}}(3,-7);
\draw[brown, ->]  (5,-7)--node[right]{{\tiny $A_{r-1,q^{-n+1}q_2}^{-1}\ $}}(7.3,-4.4);
\draw[violet, ->]  (5,-1)--node[right]{{\tiny $A_{r,q^{-n+1}q_2}^{-1}$}}(7.3,-3.4);

\draw[magenta, ->] (9,-4)--node[below]{{\tiny $A_{r-2,q^{-n+1}q_2^2}^{-1}$}} (11,-4);
\draw[blue, ->] (16.5,-4)--node[below]{{\tiny $A_{1,q^{-2n}q_2^2}^{-1}$}} (19.5,-4);
\end{tikzpicture}
\end{center}
\caption{The vector $\mathfrak{osp}_{2n+2,2n}$  $qq$-character.}\label{osp vector}
\end{figure}
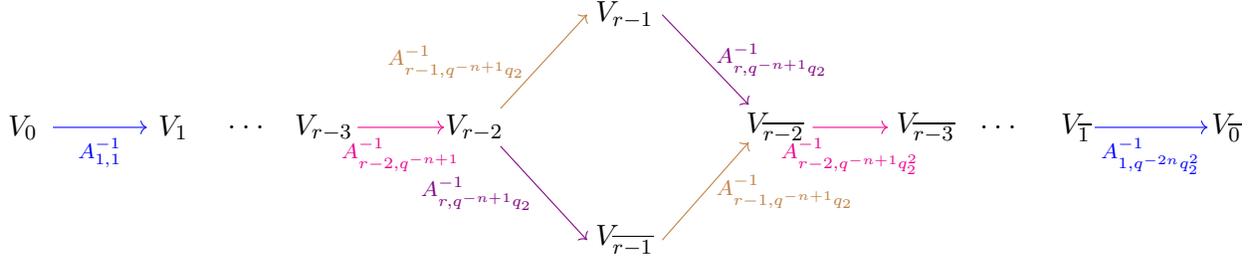

The formulas in the case $r=2n$ are similar. In particular the graph is the same (with two vertices less) the only difference is the shifts of the affine roots.

\medskip

Now we go to the affinization.
Let $\hat I=\{0,1,\dots,r\}$. The affine Cartan matrix $\hat C$ is determined by $\hat c_{ij}=c_{ij}$, $i,j\in I$ and
the other non-zero entries are $c_{00}=q-q^{-1}$, $c_{02}=c_{20}=q_1-q_1^{-1}$, $c_{01}=c_{10}=q_1^{-1}q_2-q_1q_2^{-1}$.
Then the left upper corner looks similar to the right bottom corner:
$$
\begin{pmatrix}
q-q^{-1} & q_1^{-1} q_2-q_1q_2^{-1} & q_1-q_1^{-1} \\
q_1^{-1} q_2-q_1q_2^{-1}  & q-q^{-1} & q_1-q_1^{-1}  \\
q_1-q_1^{-1} & q_1-q_1^{-1} & q-q^{-1} 
\end{pmatrix}.
$$

We start with a monomial $\hat V_0=\bs 1_q^{qq_2^2} \bs 0_{q^{-1}}^{q_1^2 q}$ which is $1$-dominant and $0$-anti-dominant and obtain an infinite slim $qq$-character which we now describe.

Let $\hat V_i=V_i$ and $\hat V_{\overline{i}}=V_{\overline{i}}$  be given by \eqref{n|n V} and by \eqref{osp V overline}  if $i=2,\dots, r-3$,  and by \eqref{osp V r} if $i=r-2,r-1$. 
In addition we set
$$
Q=q_2^2q^{-2n+4}
$$
and then
\begin{align*}
&\hat V_0=\bs 1_q^{qq_2^2} \bs 0_{q^{-1}}^{q_1^2 q}, \qquad 
&\hat V_{\overline {0}} = \bs 1_{q_2^2 q^{-2n+3}}^{q^{-2n+3}}0_{q_2^2q^{-2n+5}}^{q_2^4q^{-2n+5}},\qquad\qquad \\
&\hat V_1=\bs 2_{q_1^{-1}}^{q_1}
\bs 1_{q^{-1}}^{qq_2^2}\bs 0_{q^{-1}}^{qq_2^2}, 
\qquad
&\hat V_{\overline {1}}=\bs 2_{q_2q^{-2n+3}}^{q_2^3q^{-2n+5}}
\bs 0_{q_2^2 q^{-2n+5}}^{q^{-2n+3}} \bs 1_{q_2^2 q^{-2n+5}}^{q^{-2n+3}}.  \ 
\end{align*}

We also use the notation $\hat V_{i,\sigma}=\tau_\sigma V_i$ and 
$\hat V_{\overline{i},\sigma}=\tau_\sigma V_{\overline{i}}$.

Then we have a slim $qq$-character
$\hat \chi =\sum_{j\in \Z}\sum_{i=0}^{r-1} (\hat V_{i,Q^j}+ \hat V_{\overline{i},Q^j})$. 
We call this $qq$-character the vector $qq$-character of $\hat{\mathfrak{osp}}_{2n,2n}$ type.

The graph of the vector $qq$-character of $\hat{\mathfrak{osp}}_{2n,2n}$ type is given in Figure \ref{osp hat vector}.

\begin{figure} [H] 
\begin{center}
\begin{tikzpicture}[scale=0.45]

\node at  (-14,-4) {$\dots$};
\node at  (-12,-4) {\tiny $\hat V_{\overline{2},Q^{-1}} $};
\node at  (-8,-4)  {\tiny $\hat V_{\overline{1}, Q^{-1}}$};
\node at  (-6,-1)  {\tiny $\hat V_{\overline{0},Q^{-1}}$};
\node at  (-6,-7)  {\tiny $\hat V_0$};
\node at  (0,-4)  { $\dots$};
\node at  (-4,-4)  {\tiny $\hat V_{1}$};
\node at  (8,-4)  {\tiny $\hat  V_{\overline{r-2}}$};
\node at  (6,-1)  {\tiny $\hat V_{r-1}$};
\node at  (6,-7)  {\tiny $\hat V_{\overline{r-1}}$};
\node at  (4,-4)  {\tiny $\hat V_{r-2}$};

\node at  (12,-4)  {\tiny$\dots$};
\node at  (16,-4)  {\tiny $\hat  V_{\overline{1}}$};
\node at  (20,-4)  {\tiny $\hat  V_{1,Q}$};
\node at  (18,-1)  {\tiny $\hat  V_{\overline{0}}$};
\node at  (18,-7)  {\tiny $\hat  V_{0,Q}$};
\node at  (21.5,-4)  { $\dots$};

\draw[red, ->] (-11,-4)--node[below]{{\tiny $A_{2,q_2^{-1}}^{-1}$}} (-9,-4);
\draw[blue, ->] (-8,-3.5)--node[left]{{\tiny $A_{1,1}^{-1}$}} (-6.5,-1.5);
\draw[blue, ->] (-5.5,-6.5)--node[right]{{\tiny $A_{1,1}^{-1}$}} (-4.0,-4.5);
\draw[ ->] (-8,-4.5)--node[left]{{\tiny $A_{0,1}^{-1}$}} (-6.5,-6.5);
\draw[ ->] (-5.5,-1.5)--node[right]{{\tiny $A_{0,1}^{-1}$}} (-4.0,-3.5);

\draw[red,->] (-3.1,-4)-- node[below]{{\tiny $A_{2,q^2}^{-1}$}}(-0.8,-4);
\draw[magenta,->] (1.1,-4)--(3.2,-4); 
\draw[brown, ->]  (4,-3.5)--node[left]{{\tiny $A_{r-1,Q^{1/2}}^{-1}\ $}}(5.5,-1.5);
\draw[violet, ->]  (4,-4.5)--node[left]{{\tiny $A_{r,Q^{1/2}}^{-1}$}}(5.5,-6.5);
\draw[violet, ->]  (6.5,-1.5)--node[right]{{\tiny $A_{r-1,Q^{1/2}}^{-1}\ $}}(8,-3.5);
\draw[brown, ->]  (6.5,-6.5)--node[right]{{\tiny $A_{r,Q^{1/2}}^{-1}$}}(8,-4.5);

\draw[magenta, ->] (9,-4)--(11,-4);

\draw[red, ->] (13,-4)--node[below]{{\tiny $A_{2,Qq_2^{-1}}^{-1}$}} (15.5,-4);
\draw[blue, ->] (16,-3.5)--node[left]{{\tiny $A_{1,Q}^{-1}$}} (17.5,-1.5);
\draw[blue, ->] (18.5,-6.5)--node[right]{{\tiny $A_{1,Q}^{-1}$}} (20.0,-4.5);
\draw[ ->] (16,-4.5)--node[left]{{\tiny $A_{0,Q}^{-1}$}} (17.5,-6.5);
\draw[ ->] (18.5,-1.5)--node[right]{{\tiny $A_{0,Q}^{-1}$}} (20.0,-3.5);

\end{tikzpicture}
\end{center}
\caption{The vector $\hat{\mathfrak{osp}}_{2n+2,2n}$  $qq$-characters.}\label{osp hat vector}
\end{figure}

The vector $qq$-character of $\hat{\mathfrak{osp}}_{2n+2,2n}$ type corresponding to $r=2n$ is similar.

\subsection{The case of the vector representation of D$(2,1;\alpha)$}\label{sec D}
In this section we work with three independent variables: $q=q_0,q_1,q_2$. We also use $q_3=(q_0q_1q_2)^{-1}$ and $p_i=q^2q_i^2$, $i=1,2,3$.

Set $I=\{1,2,3\}$. In this section we study the following Cartan matrix:
\begin{align}\label{D Cartan}
&C=
\begin{pmatrix}
 q_0-q_0^{-1} & q_3- q_3^{-1} & q_2- q_2^{-1} \\
 q_3-  q_3^{-1} & q_0- q_0^{-1} & q_1- q_1^{-1} \\
 q_2-q_2^{-1} & q_1- q_1^{-1} & q_0- q_0^{-1} \\
\end{pmatrix}.
\end{align}

We have
\begin{align}\label{D affine roots}
A_{1}={\bf 1}_{q_0}^{q_0^{-1}}  {\bf 2}_{q_3}^{q_3^{-1}}  
{\bf 3}_{q_2}^{q_2^{-1}}, \qquad 
A_{2}=
{\bf 1}_{q_3}^{q_3^{-1}} 
{\bf 2}_{q_0}^{q_0^{-1}}  {\bf 3}_{q_1}^{q_1^{-1}} \,,\qquad
A_{3}= {\bf 1}_{q_2}^{q_2^{-1}}  {\bf 2}_{q_1}^{q_1^{-1}}  {\bf 3}_{q_0}^{q_0^{-1}} \,.
\end{align}
We call $qq$-characters corresponding to this data of D$(2,1;\alpha)$ type.

Note the natural symmetry under simultaneous permutations of colors $1, 2, 3$ and variables $q_1,q_2, q_3$.

Given a $qq$-character, we have three restriction maps $\rho_{\{1,2\}}$, $\rho_{\{1,3\}}$, and  $\rho_{\{2,3\}}$ which produce $qq$-characters of $\mathfrak{gl}_{2,1}$ type (with  $\mathfrak{gl}_{2,1}$  non-elliptic parameters $q_3,q_2$, and $q_1$ respectively).  

We start with a dominant monomial $V_{0,0}^{312}=\bs 3^{q_2^{-1}}_{q_0^2q_2}$. Then $\rho_{\{1,3\}}V_{0,0}^{312}$ coincides (up to a shift by $q_0q_2$) with the dominant monomial for the vector representation of  $\mathfrak{gl}_{2,1}$.

Applying the algorithm, we obtain the following result. 

For $a\in\Z_{\geq 0}$,  $b=0,1$, define the monomials $V_{a,b}^{312}$ by the formulas
\begin{equation}
\begin{aligned}\label{D vector monom}
&V^{312}_{2k,0}
={\bf 1}^{q_0 p_1^{-k}}_{q_0}
{\bf 2}^{q_3^{-1}}_{q_3^{-1} p_1^{-k}}{\bf 3}^{q_2^{-1} }_{q_0^2q_2 p_1^{-k}}\,, \\
&V^{312}_{2k+1,0}
={\bf 1}^{q_0q_2^2 p_1^{-k}}_{q_0 }
{\bf 2}^{q_3^{-1} }_{q_1^{-2}q_3^{-1} p_1^{-k}}{\bf 3}^{q_2^{-1} }_{q_2 p_1^{-k}}\,, \\
&V^{312}_{2k,1}
={\bf 1}^{q_0 p_1^{-k}}_{q_0^{-1}}
{\bf 2}^{q_3 }_{q_3^{-1} p_1^{-k}}{\bf 3}^{q_2 }_{q_0^2q_2 p_1^{-k}}\,, \\
&V^{312}_{2k+1,1}
={\bf 1}^{q_0q_2^2 p_1^{-k}}_{q_0^{-1} }
{\bf 2}^{q_3}_{q_1^{-2}q_3^{-1} p_1^{-k}}{\bf 3}^{q_2 }_{q_2 p_1^{-k}}\,.
\end{aligned}
\end{equation}

Then $\chi^{312}=\sum_{a,b} V^{312}_{a,b}$ is a slim linear $qq$-character which we call the vector $qq$-character of D$(2,1;\alpha)$ type. The graph is given in Figure \ref{D vector pic}.

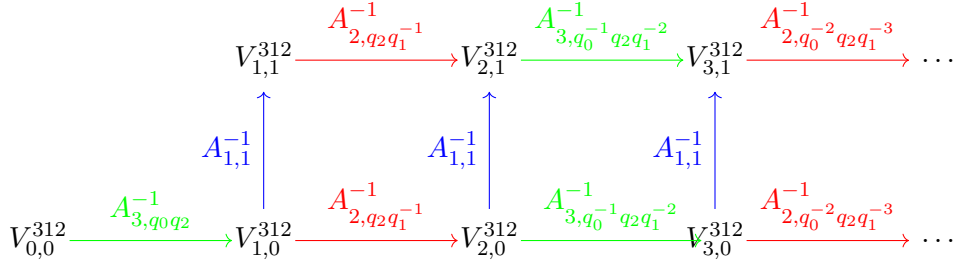
\begin{figure} [H]
\begin{center}
\begin{tikzpicture}[scale=0.6]

\node at  (5,4)  {\small $V^{312}_{1,1}$};
\node at  (10,4)  {\small$V^{312}_{2,1}$};

\node at  (0,0)  {\small $V^{312}_{0,0}$};
\node at  (5,0)  {\small $V^{312}_{1,0}$};
\node at  (10,0)  {\small $V^{312}_{2,0}$};

\node at  (15,4)  {\small$V^{312}_{3,1}$};
\node at  (15,0)  {\small $V^{312}_{3,0}$};

\node at  (20,4)  {\small$\dots$};
\node at  (20,0)  {\small $\dots$};

\draw[blue,->]  (5,0.7)--node[left]{{\small $A_{1,1}^{-1}$}}(5,3.3);
\draw[blue,->]  (10,0.7)--node[left]{{\small $A_{1,1}^{-1}$}}(10, 3.3);
\draw[blue,->]  (15,0.7)--node[left]{{\small $A_{1,1}^{-1}$}}(15, 3.3);

\draw[red,->] (5.7,4)-- node[above]{{\small $A_{2,q_2q_1^{-1}}^{-1}$}}(9.3,4);
\draw[green,->]  (10.7,4)--node[above]{{\small $A_{3,q_0^{-1}q_2q_1^{-2}}^{-1}$}}(14.3,4);

\draw[green,->] (0.7,0)-- node[above]{{\small $A_{3,q_0q_2}^{-1}$}}(4.3,0);
\draw[red,->] (5.7,0)-- node[above]{{\small $A_{2,q_2q_1^{-1}}^{-1}$}}(9.3,0);
\draw[green,->]  (10.7,0)--node[above]{{\small $A_{3,q_0^{-1}q_2q_1^{-2}}^{-1}$}}(14.7,0);

\draw[red,->] (15.7,0)-- node[above]{{\small $A_{2,q_0^{-2}q_2q_1^{-3}}^{-1}$}}(19.3,0);
\draw[red,->] (15.7,4)-- node[above]{{\small $A_{2,q_0^{-2}q_2q_1^{-3}}^{-1}$}}(19.3,4);

\end{tikzpicture}
\end{center}
\caption{The vector D$(2,1;\alpha)$ $qq$-character.}\label{D vector pic}
\end{figure}

Note that the $\gl_{2,1}$ character $\rho_{\{1,3\}}\chi^{312}$ consists
of one vector $qq$-character (with three terms) described in Section \ref{sec gl} (also discussed in Section \ref{sec 21} as a degree zero segment) and an infinite sum of shifts of a prime $2\times 2$  rectangle described in Section \ref{sec 21}.  The $\gl_{2,1}$ character $\rho_{\{2,3\}}\chi^{312}$ gives two  $\gl_{2,1}$ degree zero half-line $qq$-characters whose graphs in Figure \ref{D vector pic} are the horizontal half-lines. 
The $\gl_{2,1}$ character $\rho_{\{1,2\}}\chi^{312}$ is a sum of the trivial $qq$-character with an infinite sum of shifts of a prime $2\times 2$  rectangle.

\medskip

Using the symmetries we obtain six slim linear $qq$-characters $\chi^{abc}$, where $\{a,b,c\}=\{1,2,3\}$.

We also can construct similar 
 six slim linear $qq$-characters $\chi_{abc}$, where $\{a,b,c\}=\{1,2,3\}$ which have an anti-dominant monomial. For example, the $qq$-character $\chi_{312}$ starts at the anti-dominant monomial
 $\bs 3^{q_2}_{q_0^{-2}q_2^{-1}}$. The formulas for other monomials are obtained from \eqref{D vector monom} by replacing $q_i$ and $p_i$ with $q_i^{-1}$, $p_i^{-1}$. The graph of $\chi_{312}$ is obtained from Figure \ref{D vector pic} by changing the direction of all edges.
 
\subsection{The case of the vector representation of $\hat{\textrm{D}}(2,1;\alpha)$}\label{sec D hat}
We affinize the results of the previous section.  

We have the same three independent variables: $q=q_0,q_1,q_2$. We still use $q_3=(qq_1q_2)^{-1}$ and $p_i=q^2q_i^2$, $i=1,2,3$.

Let $\hat I=\{0,1,2,3\}$. We set 
\begin{align}\label{D hat Cartan}
& \hat C=
\begin{pmatrix}
q_0 - q_0^{-1} & q_1- q_1^{-1} & q_2- q_2^{-1} & q_3- q_3^{-1} \\
q_1- q_1^{-1} & q_0- q_0^{-1} & q_3- q_3^{-1} & q_2- q_2^{-1} \\
q_2- q_2^{-1} & q_3-  q_3^{-1} & q_0- q_0^{-1} & q_1- q_1^{-1} \\
q_3- q_3^{-1} & q_2- q_2^{-1} & q_1- q_1^{-1} & q_0- q_0^{-1} \\
\end{pmatrix}.
\end{align}

This deformed Cartan matrix produces the affine roots:
\begin{align*}
&A_{0}={\bf 0}_{q_0}^{q_0^{-1}} {\bf 1}_{q_1}^{q_1^{-1}} 
{\bf 2}_{q_2}^{q_2^{-1}}  {\bf 3}_{q_3}^{q_3^{-1}} \,,\quad
A_{1}={\bf 0}_{q_1}^{q_1^{-1}}  
{\bf 1}_{q_0}^{q_0^{-1}}  {\bf 2}_{q_3}^{q_3^{-1}}  
{\bf 3}_{q_2}^{q_2^{-1}} \,,\\
&A_{2}={\bf 0}_{q_2}^{q_2^{-1}} 
{\bf 1}_{q_3}^{q_3^{-1}} 
{\bf 2}_{q_0}^{q_0^{-1}}  {\bf 3}_{q_1}^{q_1^{-1}} \,,\quad
A_{3}={\bf 0}_{q_3}^{q_3^{-1}}  {\bf 1}_{q_2}^{q_2^{-1}}  {\bf 2}_{q_1}^{q_1^{-1}}  {\bf 3}_{q_0}^{q_0^{-1}} \,.
\end{align*}
We call $qq$-characters corresponding to this Cartan matrix of $\hat {\textrm D}(2,1;\alpha)$ type.

We still have symmetries given by permutations of colors $1,2,3$ with simultaneous permutations of variables $q_1,q_2, q_3$.

The restriction $\rho_{\{1,2,3\}} A_i$, $i=1,2,3$, gives the affine roots from the previous section. Moreover, for any $J\subset\hat I$ the restrictions $\rho_J$ reproduce the affine roots of $\gl_{2,1}$ type if $|J|=2$ and of
D$(2,1;\alpha)$ type if $|J|=3$.

We start with the monomial $m=V_{0,0}={\bf 0}^{q_1^{-1}q_3^{-2}}_{q_1}{\bf 3}^{q_2^{-1}}_{q_0^2q_2}$ which we declare $3$-dominant and $0$-anti-dominant. Then the algorithm  produces the following result.

Set
\begin{equation}
\begin{aligned}\label{D hat vector monom}
&V^{312}_{2k,2\ell}
={\bf 0}^{q_1^{-1}q_3^{-2} p_1^{-k}}_{q_1 p_1^{-\ell}}{\bf 1}^{q_0 p_1^{-k}}_{q_0 p_1^{-\ell}}
{\bf 2}^{q_3^{-1} p_1^{-\ell}}_{q_3^{-1} p_1^{-k}}{\bf 3}^{q_2^{-1} p_1^{-\ell}}_{q_0^2q_2 p_1^{-k}}\,, \\
&V^{312}_{2k+1,2\ell}
={\bf 0}^{q_1^{-1} p_1^{-k}}_{q_1 p_1^{-\ell}}{\bf 1}^{q_0q_2^2 p_1^{-k}}_{q_0 p_1^{-\ell}}
{\bf 2}^{q_3^{-1} p_1^{-\ell}}_{q_1^{-2}q_3^{-1} p_1^{-k}}{\bf 3}^{q_2^{-1} p_1^{-\ell}}_{q_2 p_1^{-k}}\,, \\
&V^{312}_{2k,2\ell+1}
={\bf 0}^{q_1^{-1}q_3^{-2} p_1^{-k}}_{q_1^{-1} p_1^{-\ell}}{\bf 1}^{q_0 p_1^{-k}}_{q_0^{-1} p_1^{-\ell}}
{\bf 2}^{q_3 p_1^{-\ell}}_{q_3^{-1} p_1^{-k}}{\bf 3}^{q_2 p_1^{-\ell}}_{q_0^2q_2 p_1^{-k}}\,, \\
&V^{312}_{2k+1,2\ell+1}
={\bf 0}^{q_1^{-1} p_1^{-k}}_{q_1^{-1} p_1^{-\ell}}{\bf 1}^{q_0q_2^2 p_1^{-k}}_{q_0^{-1} p_1^{-\ell}}
{\bf 2}^{q_3 p_1^{-\ell}}_{q_1^{-2}q_3^{-1} p_1^{-k}}{\bf 3}^{q_2 p_1^{-\ell}}_{q_2 p_1^{-k}}\,.
\end{aligned}
\end{equation}

Then the sum $\hat\chi^{312}=\sum_{a\geq b} V_{a,b}^{312}$ is a linear $qq$-character which we call the vector $qq$-character of $\hat {\textrm D}(2,1;\alpha)$ type. The graph of the vector $qq$-character is pictured in Figure \ref{D hat vector pic}.
\begin{figure} [H]
\begin{center}
\begin{tikzpicture}[scale=0.65]

\node at  (8,6)  {$\dots$};
\node at  (8,-6)  {$\dots$};
\node at  (-6,-6)  {$\dots$};

\node at  (4,4)  {\small $V^{312}_{1,1}$};
\node at  (8,4)  {\small $V^{312}_{2,1}$};
\node at  (12,4)  {\small $V^{312}_{3,1}$};
\node at  (16,4)  {\small $\dots$};

\node at  (0,0)  {\small $V^{312}_{0,0}$};
\node at  (4,0)  {\small $V^{312}_{1,0}$};
\node at  (8,0)  {\small $V^{312}_{2,0}$};
\node at  (12,0)  {\small $V^{312}_{3,0}$};
\node at  (16,0)  {\small $\dots$};
\node at  (-4,-4) {\small $V^{312}_{-1,-1}$};
\node at  (0,-4)  {\small $V^{312}_{0,-1}$};
\node at  (4,-4)  {\small $V^{312}_{1,-1}$};
\node at  (8,-4)  {\small $V^{312}_{2,-1}$};
\node at  (12,-4)  {\small $V^{312}_{3,-1}$};
\node at  (16,-4)  {\small $\dots$};

\draw[->]  (0,-3.5)--node[left]{{\small $A_{0,q_0q_1}^{-1}$}}(0,-0.5);
\draw[->]  (4,-3.5)--node[left]{{\small $A_{0,q_0q_1}^{-1}$}}(4,-0.5);
\draw[->]  (8,-3.5)--node[left]{{\small $A_{0,q_0q_1}^{-1}$}}(8,-0.5);
\draw[->]  (12,-3.5)--node[left]{{\small $A_{0,q_0q_1}^{-1}$}}(12,-0.5);

\draw[blue,->]  (4,0.5)--node[left]{{\small $A_{1,1}^{-1}$}}(4,3.5);
\draw[blue,->]  (8,0.5)--node[left]{{\small $A_{1,1}^{-1}$}}(8,3.5);
\draw[blue,->]  (12,0.5)--node[left]{{\small $A_{1,1}^{-1}$}}(12,3.5);

\draw[red,->] (4.7,4)-- node[above]{{\small $A_{2,q_2q_1^{-1}}^{-1}$}}(7.3,4);
\draw[green,->]  (8.7,4)--node[above]{{\small $A_{3,q_0q_2p_1^{-1}}^{-1}$}}(11.3,4);

\draw[green,->] (0.7,0)-- node[above]{{\small $A_{3,q_0q_2}^{-1}$}}(3.3,0);
\draw[red,->] (4.7,0)-- node[above]{{\small $A_{2,q_2q_1^{-1}}^{-1}$}}(7.3,0);
\draw[green,->]  (8.7,0)--node[above]{{\small $A_{3,q_0q_2p_1^{-1}}^{-1}$}} (11.3,0);

\draw[red,->] (-3,-4)--node[below]{{\small $A_{2,q_0^2q_1q_2}^{-1}$}} (-0.7,-4);
\draw[green,->] (0.7,-4)-- node[below]{{\small $A_{3,q_0q_2}^{-1}$}}(3.2,-4);
\draw[red,->] (4.7,-4)-- node[below]{{\small $A_{2,q_2q_1^{-1}}^{-1}$}}(7.2,-4);
\draw[green,->]  (8.7,-4)--node[below]{{\small $A_{3,q_0q_2p_1^{-1}}^{-1}$}}(11.3,-4);
\draw[red,->] (12.7,4)--(14.7,4);
\draw[red,->] (12.7,0)--(14.7,0);
\draw[red,->] (12.7,-4)--(14.7,-4);
\end{tikzpicture}
\caption{The vector $\hat {\textrm D}(2,1;\alpha)$ character.}\label{D hat vector pic}
\end{center}
\end{figure}
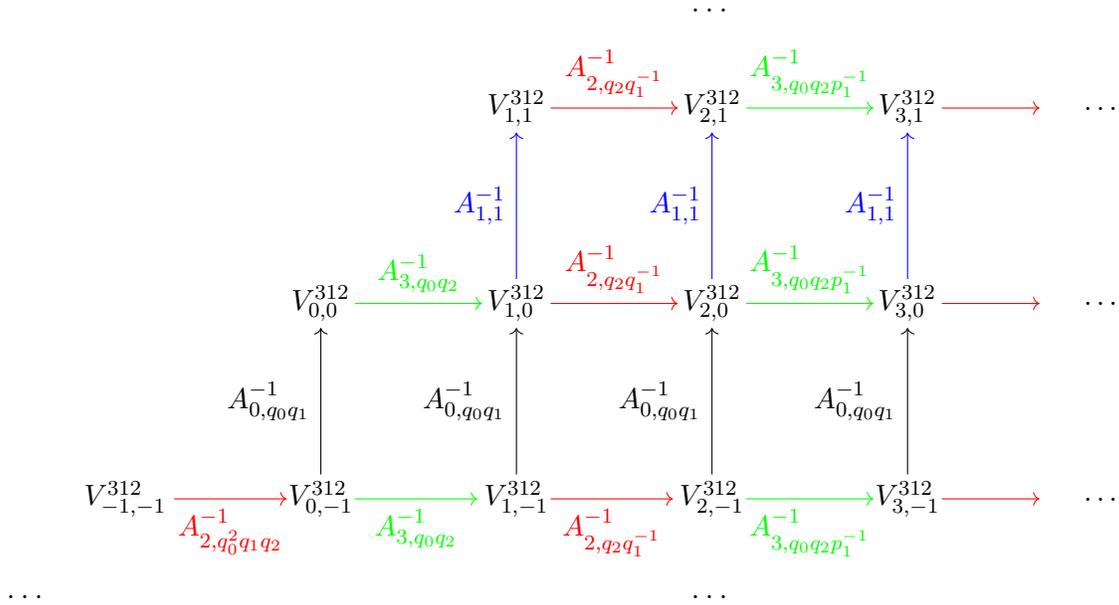

The restrictions $\rho_{\{1,2,3\}}\hat \chi^{312}$ and $\rho_{\{0,2,3\}}\hat \chi^{312}$ are sums of shifts of the vector $qq$-characters of D$(2,1;\alpha)$ type. 

The restrictions $\rho_{\{0,1,3\}}\hat \chi^{312}$ and $\rho_{\{0,1,2\}}\hat \chi^{312}$ are the sums of shifts of lowest weight analogs of vector $qq$-characters of D$(2,1;\alpha)$ type produced from an anti-dominant monomial.

Using the symmetries we obtain linear $qq$-characters $\hat \chi^{a,b,c}$, where $a,b,c$ are distinct elements of $\hat I=\{0,1,2,3\}$. Among those we have coincidences, for example, $\hat \chi^{312}$ is a shift of $\hat \chi^{203}$. In total we have 12 distinct vector $qq$-characters of $\hat {\textrm D}(2,1;\alpha)$ type.

\medskip

We describe a different construction of the vector representations of $\hat{\textrm D}(2,1;\alpha)$ type.
Define monomials of degree $(-1,-1,1,1)$
\begin{align}\label{R}
R_1^\pm=\bs 0^{q_3^{\pm 1}}\bs 1^{q_2^{\pm 1}}\bs 2_{ q_1^{\mp 1}}\bs 3_{ q_0^{\mp 1}} 
\end{align}
and monomials of degree $(1,1,-1,-1)$
\begin{align}\label{T}
T_1^{\pm}=\bs 0_{q_1^{\mp 1}}\bs 1_{q_0^{\mp1}}\bs 2^{q_3^{\pm 1}}\bs 3^{q_2^{\pm 1}}. 
\end{align}
As always, we set $R^\pm_{1,\sigma}=\tau_\sigma R_1^{\pm 1}$ and $T^\pm_{1,\sigma}=\tau_\sigma T_1^{\pm 1}$. 

Then it is easy to check that $\chi_R=\sum_{i\in\Z} (R_{1,p^i}^+ + R^-_{1,p^i})$ and 
$\chi_T=\sum_{i\in\Z} (T_{1,p^i}^+ + T^-_{1,p^i})$ are $qq$-characters of degrees $(-1,-1,1,1)$ and $(1,1,-1,-1)$ respectively. 

The restrictions $\rho_{\{2,3\}}\chi_R$ and  $\rho_{\{0,1\}}\chi_T$ are simply line $qq$-characters of type $\gl_{2,1}$.

It follows that a generic product $\tau_\kappa(\chi_R) \chi_T$ is a linear $qq$-character of degree 0. And for special monomials $\kappa$, the product truncates. The $qq$-character $\chi^{312}$ is a truncation of  $\tau_{q_0q_2}(\chi_R) \chi_T$.

\section{The 18, 66, 130.}\label{sec:18,66,130}
In this section we study finite degree zero tame $qq$-characters of type D$(2,1;\alpha)$ and their restrictions to $\gl_{2,1}$.
In particular our Cartan matrix is \eqref{D  Cartan} and the affine roots are given by \eqref{D affine roots}.

\subsection{18}
The smallest possible nontrivial degree zero $qq$-character of type $ D(2,1;\alpha)$ turns out to have 18 terms. It is constructed from the dominant  monomial 
$m_+^{18}={\bf 1}^{q_1}_{q_1^{-1}} {\bf 2}^{q_2}_{q_2^{-1}} {\bf 3}^{q_3}_{q_3^{-1}}$.
In fact $m_+^{18}=\rho^{\{1,2,3\}}A_0^{-1}$, where $A_0$ is the affine root of type $\hat {\textrm D}(2,1;\alpha)$.

The $qq$-character $\chi^{18}=\sum_{i=1}^{18}v_i$ turns out to be slim but not linear. We call it the adjoint $qq$-character of D$(2,1;\alpha)$ type.

The monomials $v_i$ and the graph are shown in Figure \ref{18 pic}, where we set
\begin{align*}
A^{-1}_{i,\pm}=A^{-1}_{i,q_0^{-1}q_i^{\pm 1}}\,.
\end{align*}

\begin{figure} [ht]
\begin{center}
\begin{tikzpicture}[scale=0.60]
\coordinate  (+++) at  (0,8);

\coordinate  (002) at (-10,4);
\coordinate  (020) at (0,4);
\coordinate  (200) at (10,4);

\coordinate  (-++) at (-10,0);
\coordinate  (+-+) at (0,0);
\coordinate  (++-) at (10,0);

\coordinate  (L) at (-10,-6);
\coordinate  (C) at (0,-6);
\coordinate  (CC) at (2,-4);
\coordinate  (R) at (10,-6);

\coordinate  (+--) at (-10,-10);
\coordinate  (-+-) at (0,-10);
\coordinate  (--+) at (10,-10);

\coordinate  (00-2) at (-10,-14);
\coordinate  (0-20) at (0,-14);
\coordinate  (-200) at (10,-14);

\coordinate  (---) at  (0,-18);

\node at (0,10) {${  }$};

\node at  (+++) { 
$v_1={\bf 1}^{q_1}_{q_1^{-1}} {\bf 2}^{q_2}_{q_2^{-1}} {\bf 3}^{q_3}_{q_3^{-1}}$};

\node at  (002) 
{$v_2={\bf 1}^{q_1}_{q_0^{-2}q_1^{-1}}
{\bf 2}^{q_2q_3^2}_{q_2^{-1}} {\bf 3}^{q_2^2q_3}_{q_3^{-1}}$};

\node at  (020) 
{$v_3= {\bf 1}^{q_1q_3^2}_{q_1^{-1}}
{\bf 2}^{q_2}_{q_0^{-2}q_2^{-1}} {\bf 3}^{q_1^2q_3}_{q_3^{-1}}$};

\node at  (200)
{$ v_4={\bf 1}^{q_1q_2^2}_{q_1^{-1}}
{\bf 2}^{q_1^2q_2}_{q_2^{-1}} {\bf 3}^{q_3}_{q_0^{-2}q_3^{-1}}$};

\node at  (-++) 
{$v_5=
{\bf 1}^{q_1q_3^2}_{q_0^{-2}q_1^{-1}}
{\bf 2}^{q_2q_3^2}_{q_0^{-2}q_2^{-1}} 
{\bf 3}^{q_1^2q_3,q_2^2q_3}_{q_3^{-1},q_3}$};

\node at  (+-+) 
{$v_6={\bf 1}^{q_1q_2^2}_{q_0^{-2}q_1^{-1}}
{\bf 2}^{q_2q_1^2,q_2q_3^2}_{q_2^{-1},q_2} 
{\bf 3}^{q_2^2q_3}_{q_0^{-2}q_3^{-1}}$};

\node at  (++-) 
{$v_7=
{\bf 1}^{q_1q_2^2,q_1q_3^2}_{q_1^{-1},q_1} 
{\bf 2}^{q_1^2q_2}_{q_0^{-2}q_2^{-1}}
{\bf 3}^{q_1^2q_3}_{q_0^{-2}q_3^{-1}}$};

\node at (L) 
{$
v_8={\bf 3}^{q_1^2q_3,q_2^2q_3}_{q_0^{-2}q_3,q_3^{-1}}$};

\node at (C) 
{$
v_9={\bf 2}^{q_1^2q_2,q_3^2q_2}_{q_0^{-2}q_2,q_2^{-1}}$};

\node at (CC) 
{$
v_{11}={\bf 1}^{q_1q_2^2, q_1q_3^2}_{q_0^{-2}q_1^{-1},q_1}
{\bf 2}^{q_2q_1^2, q_2q_3^2}_{q_0^{-2}q_2^{-1},q_2}
{\bf 3}^{q_3q_1^2, q_3q_2^2}_{q_0^{-2}q_3^{-1},q_3}
$};

\node at (R) 
{$v_{10}={\bf 1}^{q_2^2q_1,q_3^2q_1}_{q_0^{-2}q_1,q_1^{-1}}$};

\node at  (+--) 
{$v_{12}=
{\bf 1}^{q_1q_2^2}_{q_1}
{\bf 2}^{q_2q_1^2}_{q_2} 
{\bf 3}^{q_1^2q_3,q_2^2q_3}_{q_0^{-2}q_3^{-1},q_0^{-2}q_3}$};

\node at  (-+-) 
{$v_{13}={\bf 1}^{q_1q_3^2}_{q_1}
{\bf 2}^{q_1^2q_2,q_3^2q_2}_{q_0^{-2}q_2,q_0^{-2}q_2}
{\bf 3}^{q_3q_1^2}_{q_3} $};

\node at (--+) 
{$v_{14}=
{\bf 1}^{q_2^2q_1,q_3^2q_1}_{q_0^{-2}q_1^{-1},q_0^{-2}q_1}
{\bf 2}^{q_2q_3^2}_{q_2}
{\bf 3}^{q_3q_2^2}_{q_3} $};

\node at  (00-2) 
{$v_{15}=
{\bf 1}^{q_0^{-2}q_1^{-1}}_{q_1}
{\bf 2}^{q_2q_1^2}_{q_0^{-2}q_2} 
{\bf 3}^{q_1^2q_3}_{q_0^{-2}q_3}$};

\node at  (0-20) 
{$v_{16}={\bf 1}^{q_1 q_2^2}_{q_0^{-2}q_1}
{\bf 2}^{q_0^{-2}q_2^{-1}}_{q_2} 
{\bf 3}^{q_2^2q_3}_{q_0^{-2}q_3}$};

\node at  (-200) 
{$v_{17}=
{\bf 1}^{q_1 q_3^2}_{q_0^{-2}q_1}
{\bf 2}^{q_3^2 q_2}_{q_0^{-2}q_2} 
{\bf 3}^{q_0^{-2}q_3^{-1}}_{q_3}$};

\node at  (---) 
{$v_{18}=
{\bf 1}^{q_0^{-2} q_1^{-1}}_{q_0^{-2} q_1}
{\bf 2}^{q_0^{-2} q_2^{-1}}_{q_0^{-2} q_2}
{\bf 3}^{q_0^{-2} q_3^{-1}}_{q_0^{-2} q_3}
$};

\draw[green, ->] (0.5,7.5)--node[above]{$A^{-1}_{3,-}$} (9.5,4.7);
\draw[red, ->] (0,7.5)-- node[right]{$A^{-1}_{2,-}$}(0,4.7);
\draw[blue, ->] (-0.5,7.5)--node[above]{$A^{-1}_{1,-}$} (-9.5,4.7);

\draw[red, ->] (-10,3.3)--node[left]{$A^{-1}_{2,-}$}(-10,0.8);
\draw[green, ->] (-9.5,3.3)--node[above]{$A^{-1}_{3,-}$}(-0.25,0.8);
\draw[green,->] (0.25,3.3)--node[below]{$A^{-1}_{3,-}$}(9.5,0.8);
\draw[blue,->] (-0.25,3.3)--node[below]{$A^{-1}_{1,-}$}(-9.5,0.8);
\draw[blue, ->] (9.5,3.3)--node[above]{$A^{-1}_{1,-}$}(0.25,0.8);
\draw[red, ->] (10,3.3)--node[right]{$A^{-1}_{2,-}$}(10,0.8);

\draw[green, ->] (-10,-0.7)--node[left]{$A^{-1}_{3,+}$} (-10,-5.3);
\draw[red, ->] (0,-0.7)--node[above]{$A^{-1}_{2,+}\hspace{22pt}$} (0,-5.3);
\draw[blue, ->] (10,-0.7)--node[right]{$A^{-1}_{1,+}$}(10,-5.3);
\draw[green, ->] (-9.5,-0.7)--node[below]{$A^{-1}_{3,-}$}(2.0,-3.3);
\draw[red, ->] (0.5,-0.7)--node[right]{$A^{-1}_{2,-}$}(2.2,-3.3);
\draw[blue, ->] (9.5,-0.7)--node[below]{$A^{-1}_{1,-}$}(2.4,-3.3);

\draw[green, ->] (2.0,-4.7)--node[above]{$A^{-1}_{3,+}$} (-9.5,-9.3);
\draw[red, ->] (2.2,-4.7)--node[right]{$A^{-1}_{2,+}$} (0.5,-9.3);
\draw[blue, ->] (2.4,-4.7)--node[above]{$A^{-1}_{1,+}$} (9.5,-9.3);
\draw[green, ->] (-10,-6.7)--node[left]{$A^{-1}_{3,-}$} (-10,-9.3);
\draw[red, ->] (0,-6.7)--node[left]{$A^{-1}_{2,-}$} (0,-9.3);
\draw[blue,->] (10,-6.7)--node[right]{$A^{-1}_{1,-}$} (10,-9.3);

\draw[red, ->] (-10,-10.7)--node[left]{$A^{-1}_{2,+}$}(-10,-13.3);
\draw[blue, ->] (-9.5,-10.7)--node[above]{$A^{-1}_{1,+}$}(-0.25,-13.3);
\draw[green, ->] (-0.25,-10.7)--node[below]{$A^{-1}_{3,+}$}(-9.5,-13.3);
\draw[blue, ->] (0.25,-10.7)--node[above]{$A^{-1}_{1,+}$}(9.5,-13.3);
\draw[green, ->] (9.5,-10.7)--node[below]{$A^{-1}_{3,+}$}(0.25,-13.3);
\draw[red, ->] (10,-10.7)--node[right]{$A^{-1}_{2,+}$}(10,-13.3);

\draw[blue,->] (-9.5,-14.7)--node[below]{$A^{-1}_{1,+}$}(-0.5,-17.3);
\draw[red,->] (0,-14.7)--node[right]{$A^{-1}_{2,+}$}(0,-17.3);
\draw[green, ->] (9.5,-14.7)--node[below]{$A^{-1}_{3,+}$}(0.5,-17.3);

\end{tikzpicture}
\caption{The adjoint D$(2,1;\alpha)$  
$qq$-character. \label{18 pic}}
\bigskip
\bigskip
\end{center}
\end{figure}
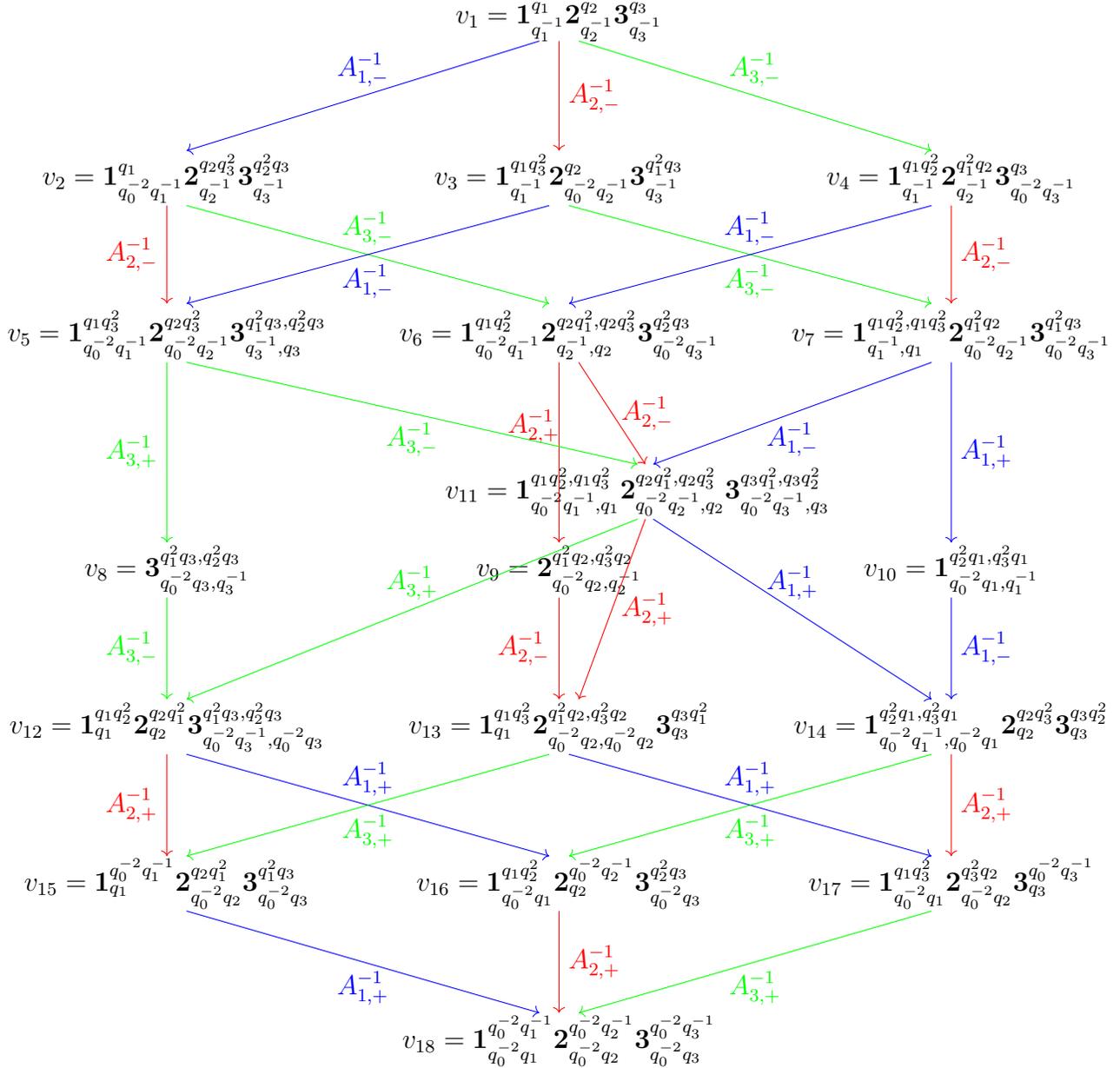

The graph has a symmetry with respect to changes of colors and variables.  In our picture  the symmetry which exchanges colors $1$ and $3$  corresponds to the reflection about the vertical line through vertices $v_1,v_3,v_6,v_{11},v_9,v_{13},v_{16}$ and $v_{18}$. Here and everywhere we do not picture the symmetry line.

It is also symmetric with respect to the middle horizontal line connecting vertices $v_8,v_9,v_{10}$ and $v_{11}$ (with colors preserved and arrows reversed). 

By the depth counting, the adjoint $qq$-character has the form $18=1+3+3+4+3+3+1$.

\medskip

We note that the restriction $\rho^{\{2,3\}}\chi^{18}$ decomposes by the rule
$$
18=9+4+4+1.
$$
Here the $1$ is a trivial (polynomial) module, the two fours are squares and the $9$ is a generic product of the vector and covector $\gl_{2,1}$ $qq$-characters.

\subsection{66}
The next smallest simple finite degree zero $qq$-character of D$(2,1;\alpha)$ type has $66$ monomials. Up to a shift there are three of such $qq$-characters obtained by permutations of colors. We describe the one with the dominant monomial
$$
m_+^{66,1}=m^{18}_+ \tau_{p_1} (v_2)=\bs 1_{q_1^{-1}}^{q_1p_1} \bs 2^{q_2}_{q_2^{-1}p_1}\bs 3^{q_3}_{q_3^{-1}p_1},
$$
where color $1$ is chosen and the symmetry between colors $2$ and $3$ is preserved.

Starting from this monomial one produces a slim simple (but not linear) $qq$-character $\chi^{66,1}$ with $66$ terms.

\begin{figure} [htbp]
\begin{center}
\begin{tikzpicture}[scale=0.38]
\coordinate  (A) at  (0,8);
\node [above] at  (A) { $(1,2)$};
\node at  (A) { $\bullet$};

\coordinate  (B1) at (-4,5);
\coordinate  (B2) at (0,5);
\coordinate  (B3) at (4,5);
\node [left] at  (B1) { $(1,5)$};
\node at  (B1) { $\bullet$};
\node [above] at  (B2) { $(2,2)$};
\node at  (B2) { $\bullet$}; 
\node [right] at  (B3) {$(1,6)$};
\node at  (B3) { $\bullet$};

\coordinate  (C1) at (-10,2);
\coordinate  (C2) at (-4,2);
\coordinate  (C3) at (0,2);
\coordinate  (C4) at (4,2);
\coordinate  (C5) at (10,2);
\node [left] at  (C1) { $(1,8)$};
\node at  (C1) { $\bullet$};
\node [left] at  (C2) { $(2,5)$};
\node at  (C2) { $\bullet$};
\node [above] at  (C3) { $(1,11)$};
\node at  (C3) { $\bullet$};
\node [right] at  (C4) { $(2,6)$};
\node at  (C4) { $\bullet$};
\node [right] at  (C5) { $(1,9)$};
\node at  (C5) { $\bullet$};

\coordinate  (D1) at (-16,-3);
\coordinate  (D2) at (-10,-3);
\coordinate  (D3) at (-8,-3);
\coordinate  (D4) at (-1,-3);
\coordinate  (D5) at (1,-3);
\coordinate  (D6) at (8,-3);
\coordinate  (D7) at (10,-3);
\coordinate  (D8) at (16,-3);
\node [left] at  (D1) { $(3,8)$};
\node at  (D1) { $\bullet$};
\node [left] at  (D2) { $(2,8)$};
\node at  (D2) { $\bullet$};
\node [right] at  (D3) { $(1,12)$};
\node at  (D3) { $\bullet$};
\node [left] at  (D4) { $(2,11)$};
\node at  (D4) { $\bullet$};
\node [right] at  (D5) { $(1,14)$};
\node at  (D5) { $\bullet$};
\node [left] at  (D6) { $(1,13)$};
\node at  (D6) { $\bullet$};
\node [right] at  (D7) { $(2,9)$};
\node at  (D7) { $\bullet$};
\node [right] at  (D8) { $(4,9)$};
\node at  (D8) { $\bullet$};

\coordinate  (E1) at (-16,-9);
\coordinate  (E2) at (-14,-9);
\coordinate  (E3) at (-8,-9);
\coordinate  (E4) at (-6,-9);
\coordinate  (E5) at (0,-8.5);
\coordinate  (E6) at (0,-9.5);
\coordinate  (E7) at (8,-9);
\coordinate  (E8) at (6,-9);
\coordinate  (E9) at (14,-9);
\coordinate  (E10) at (16,-9);
\node [left] at  (E1) { $(5,8)$};
\node at  (E1) { $\bullet$};
\node [above] at  (E2) { $(3,12)$};
\node at  (E2) { $\bullet$};
\node [left] at  (E3) { $(2,12)$};
\node at  (E3) { $\bullet$};
\node [right] at  (E4) { $(1,16)$};
\node at  (E4) { $\bullet$};
\node [above] at  (E5) { $(2,14)$};
\node at  (E5) { $\bullet$};
\node [below] at  (E6) { $(1,15)$};
\node at  (E6) { $\bullet$};
\node [right] at  (E7) { $(1,17)$};
\node at  (E7) { $\bullet$};
\node [left] at  (E8) { $(2,13)$};
\node at  (E8) { $\bullet$};
\node [above] at  (E9) { $(4,13)$};
\node at  (E9) { $\bullet$};
\node [right] at  (E10) { $(6,9)$};
\node at  (E10) { $\bullet$};

\coordinate  (F1) at (-18,-17);
\coordinate  (F2) at (-15.0,-17);
\coordinate  (F3) at (-13.0,-17);
\coordinate  (F5) at (-7.0,-17.5);
\coordinate  (F4) at (-7.0,-16.5);
\coordinate  (F6) at (1,-17);
\coordinate  (F7) at (-1,-17);
\coordinate  (F8) at (7,-16.5);
\coordinate  (F9) at (7,-17.5);
\coordinate  (F10) at (15.0,-17);
\coordinate  (F11) at (13.0,-17);
\coordinate  (F12) at (18,-17);
\node [left] at  (F1) { $(8,8)$};
\node at  (F1) { $\bullet$};
\node [above] at  (F2) { $(5,12)$};
\node at  (F2) { $\bullet$};
\node [below] at  (F3) { $(3,16)$};
\node at  (F3) { $\bullet$};
\node [right] at  (F4) { $(2,16)$};
\node at  (F4) { $\bullet$};
\node [left] at  (F5) { $(3,15)$};
\node at  (F5) { $\bullet$};
\node [below] at  (F6) { $(1,18)$};
\node at  (F6) { $\bullet$};
\node [above] at  (F7) { $(2,15)$};
\node at  (F7) { $\bullet$};
\node [left] at  (F8) { $(2,17)$};
\node at  (F8) { $\bullet$};
\node [right] at  (F9) { $(4,15)$};
\node at  (F9) { $\bullet$};
\node [below] at  (F10) { $(4,17)$};
\node at  (F10) { $\bullet$};
\node [above] at  (F11) { $(6,13)$};
\node at  (F11) { $\bullet$};
\node [right] at  (F12) { $(9,9)$};
\node at  (F12) { $\bullet$};

\coordinate  (Aa) at  (0,-42);
\node [below] at  (Aa) { $(15,18)$};
\node at  (Aa) { $\bullet$};

\coordinate  (B1a) at (-4,-39);
\coordinate  (B2a) at (0,-39);
\coordinate  (B3a) at (4,-39);
\node [left] at  (B1a) { $(12,18)$};
\node at  (B1a) { $\bullet$};
\node [below] at  (B2a) { $(15,15)$};
\node at  (B2a) { $\bullet$}; 
\node [right] at  (B3a) {$(13,18)$};
\node at  (B3a) { $\bullet$};

\coordinate  (C1a) at (-10,-36);
\coordinate  (C2a) at (-4,-36);
\coordinate  (C3a) at (0,-36);
\coordinate  (C4a) at (4,-36);
\coordinate  (C5a) at (10,-36);
\node [left] at  (C1a) { $(8,18)$};
\node at  (C1a) { $\bullet$};
\node [left] at  (C2a) { $(12,15)$};
\node at  (C2a) { $\bullet$};
\node [below] at  (C3a) { $(11,18)$};
\node at  (C3a) { $\bullet$};
\node [right] at  (C4a) { $(13,15)$};
\node at  (C4a) { $\bullet$};
\node [right] at  (C5a) { $(9,18)$};
\node at  (C5a) { $\bullet$};

\coordinate  (D1a) at (-16,-31);
\coordinate  (D2a) at (-10,-31);
\coordinate  (D3a) at (-8,-31);
\coordinate  (D4a) at (-1.0,-31);
\coordinate  (D5a) at (1.0,-31);
\coordinate  (D6a) at (8,-31);
\coordinate  (D7a) at (10,-31);
\coordinate  (D8a) at (16,-31);
\node [left] at  (D1a) { $(8,16)$};
\node at  (D1a) { $\bullet$};
\node [left] at  (D2a) { $(8,15)$};
\node at  (D2a) { $\bullet$};
\node [right] at  (D3a) { $(5,18)$};
\node at  (D3a) { $\bullet$};
\node [left] at  (D4a) { $(11,15)$};
\node at  (D4a) { $\bullet$};
\node [right] at  (D5a) { $(7,18)$};
\node at  (D5a) { $\bullet$};
\node [left] at  (D6a) { $(6,18)$};
\node at  (D6a) { $\bullet$};
\node [right] at  (D7a) { $(9,15)$};
\node at  (D7a) { $\bullet$};
\node [right] at  (D8a) { $(9,17)$};
\node at  (D8a) { $\bullet$};

\coordinate  (E1a) at (-16,-25);
\coordinate  (E2a) at (-14,-25);
\coordinate  (E3a) at (-8.5,-25);
\coordinate  (E4a) at (-6.5,-25);
\coordinate  (E5a) at (0,-24.5);
\coordinate  (E6a) at (0,-25.5);
\coordinate  (E7a) at (8.5,-25);
\coordinate  (E8a) at (6.5,-25);
\coordinate  (E9a) at (14,-25);
\coordinate  (E10a) at (16,-25);
\node [left] at  (E1a) { $(8,12)$};
\node at  (E1a) { $\bullet$};
\node [below] at  (E2a) { $(5,16)$};
\node at  (E2a) { $\bullet$};
\node [left] at  (E3a) { $(5,15)$};
\node at  (E3a) { $\bullet$};
\node [right] at  (E4a) { $(3,18)$};
\node at  (E4a) { $\bullet$};
\node [above] at  (E5a) { $(2,18)$};
\node at  (E5a) { $\bullet$};
\node [below] at  (E6a) { $(7,15)$};
\node at  (E6a) { $\bullet$};
\node [right] at  (E7a) { $(4,18)$};
\node at  (E7a) { $\bullet$};
\node [left] at  (E8a) { $(6,15)$};
\node at  (E8a) { $\bullet$};
\node [below] at  (E9a) { $(6,17)$};
\node at  (E9a) { $\bullet$};
\node [right] at  (E10a) { $(9,13)$};
\node at  (E10a) { $\bullet$};

\draw[red] (A)-- (B1);
\draw[blue] (A)-- (B2);
\draw[green] (A)--(B3);
\draw[red] (Aa)-- (B1a);
\draw[blue] (Aa)-- (B2a);
\draw[green] (Aa)--(B3a);

\draw[green] (B1)--(C1);
\draw[blue] (B1)--(C2);
\draw[green] (B1)--(C3);
\draw[red] (B2)--(C2);
\draw[green] (B2)--(C4);
\draw[red] (B3)--(C3);
\draw[blue] (B3)--(C4);
\draw[red] (B3)--(C5);

\draw[green] (B1a)--(C1a);
\draw[blue] (B1a)--(C2a);
\draw[green] (B1a)--(C3a);
\draw[red] (B2a)--(C2a);
\draw[green] (B2a)--(C4a);
\draw[red] (B3a)--(C3a);
\draw[blue] (B3a)--(C4a);
\draw[red] (B3a)--(C5a);

\draw[red] (C1) -- (D1);
\draw[blue] (C1) -- (D2);
\draw[green] (C1) -- (D3);
\draw[green] (C2) -- (D2);
\draw[green] (C2) -- (D4);
\draw[green] (C3) -- (D3);
\draw[blue] (C3) -- (D4);
\draw[blue] (C3) -- (D5);
\draw[red] (C3) -- (D6);
\draw[red] (C4) -- (D4);
\draw[red] (C4) -- (D7);
\draw[red] (C5) -- (D6);
\draw[blue] (C5) -- (D7);
\draw[green] (C5) -- (D8);

\draw[red] (C1a) -- (D1a);
\draw[blue] (C1a) -- (D2a);
\draw[green] (C1a) -- (D3a);
\draw[green] (C2a) -- (D2a);
\draw[green] (C2a) -- (D4a);
\draw[green] (C3a) -- (D3a);
\draw[blue] (C3a) -- (D4a);
\draw[blue] (C3a) -- (D5a);
\draw[red] (C3a) -- (D6a);
\draw[red] (C4a) -- (D4a);
\draw[red] (C4a) -- (D7a);
\draw[red] (C5a) -- (D6a);
\draw[blue] (C5a) -- (D7a);
\draw[green] (C5a) -- (D8a);

\draw[blue] (D1) -- (E1);
\draw[green] (D1) -- (E2);
\draw[red] (D2) -- (E1);
\draw[green] (D2) -- (E3);
\draw[red] (D3) -- (E2);
\draw[blue] (D3) -- (E3);
\draw[blue] (D3) -- (E4);
\draw[red] (D3) -- (E6);
\draw[green] (D4) -- (E3);
\draw[blue] (D4) -- (E5);
\draw[red] (D4) -- (E8);
\draw[green] (D5) -- (E4);
\draw[blue] (D5) -- (E5);
\draw[red] (D5) -- (E7);
\draw[green] (D6) -- (E6);
\draw[blue] (D6) -- (E7);
\draw[blue] (D6) -- (E8);
\draw[green] (D6) -- (E9);
\draw[red] (D7) -- (E8);
\draw[green] (D7) -- (E10);
\draw[red] (D8) -- (E9);
\draw[blue] (D8) -- (E10);

\draw[blue] (D1a) -- (E1a);
\draw[green] (D1a) -- (E2a);
\draw[red] (D2a) -- (E1a);
\draw[green] (D2a) -- (E3a);
\draw[red] (D3a) -- (E2a);
\draw[blue] (D3a) -- (E3a);
\draw[blue] (D3a) -- (E4a);
\draw[red] (D3a) -- (E5a);
\draw[green] (D4a) -- (E3a);
\draw[blue] (D4a) -- (E6a);
\draw[red] (D4a) -- (E8a);
\draw[green] (D5a) -- (E4a);
\draw[blue] (D5a) -- (E6a);
\draw[red] (D5a) -- (E7a);
\draw[green] (D6a) -- (E5a);
\draw[blue] (D6a) -- (E7a);
\draw[blue] (D6a) -- (E8a);
\draw[green] (D6a) -- (E9a);
\draw[red] (D7a) -- (E8a);
\draw[green] (D7a) -- (E10a);
\draw[red] (D8a) -- (E9a);
\draw[blue] (D8a) -- (E10a);

\draw[green] (E1) -- (F1);
\draw[green] (E1) -- (F2);
\draw[blue] (E2) -- (F2);
\draw[blue] (E2) -- (F3);
\draw[red] (E2) -- (F5);
\draw[red] (E3) -- (F2);
\draw[red] (E3) -- (F7);
\draw[blue] (E3) -- (F4);
\draw[red] (E4) -- (F3);
\draw[blue] (E4) -- (F4);
\draw[red] (E4) -- (F6);
\draw[green] (E5) -- (F4);
\draw[red] (E5) -- (F8);
\draw[red] (E6) -- (F5);
\draw[blue] (E6) -- (F7);
\draw[blue] (E6) -- (F6);
\draw[green] (E6) -- (F9);
\draw[green] (E7) -- (F6);
\draw[blue] (E7) -- (F8);
\draw[blue] (E8) -- (F8);
\draw[green] (E7) -- (F10);
\draw[green] (E8) -- (F7);
\draw[green] (E8) -- (F11);
\draw[green] (E9) -- (F9);
\draw[blue] (E9) -- (F10);
\draw[blue] (E9) -- (F11);
\draw[red] (E10) -- (F11);
\draw[red] (E10) -- (F12);

\draw[green] (E1a) -- (F1);
\draw[green] (E1a) -- (F2);
\draw[blue] (E2a) -- (F2);
\draw[blue] (E2a) -- (F3);
\draw[red] (E2a) -- (F4);
\draw[red] (E3a) -- (F2);
\draw[red] (E3a) -- (F7);
\draw[blue] (E3a) -- (F5);
\draw[red] (E4a) -- (F6);
\draw[blue] (E4a) -- (F5);
\draw[red] (E4a) -- (F3);
\draw[green] (E5a) -- (F8);
\draw[red] (E5a) -- (F4);
\draw[red] (E6a) -- (F9);
\draw[blue] (E5a) -- (F7);
\draw[blue] (E5a) -- (F6);
\draw[green] (E6a) -- (F5);
\draw[green] (E7a) -- (F6);
\draw[blue] (E7a) -- (F9);
\draw[blue] (E8a) -- (F9);
\draw[green] (E7a) -- (F10);
\draw[green] (E8a) -- (F7);
\draw[green] (E8a) -- (F11);
\draw[green] (E9a) -- (F8);
\draw[blue] (E9a) -- (F10);
\draw[blue] (E9a) -- (F11);
\draw[red] (E10a) -- (F11);
\draw[red] (E10a) -- (F12);

\end{tikzpicture}
\caption{The $66$ term $qq$-character of type ${D}(2,1;\alpha)$.} \label{66 pic}
\label{fig-66}
\end{center}
\end{figure}
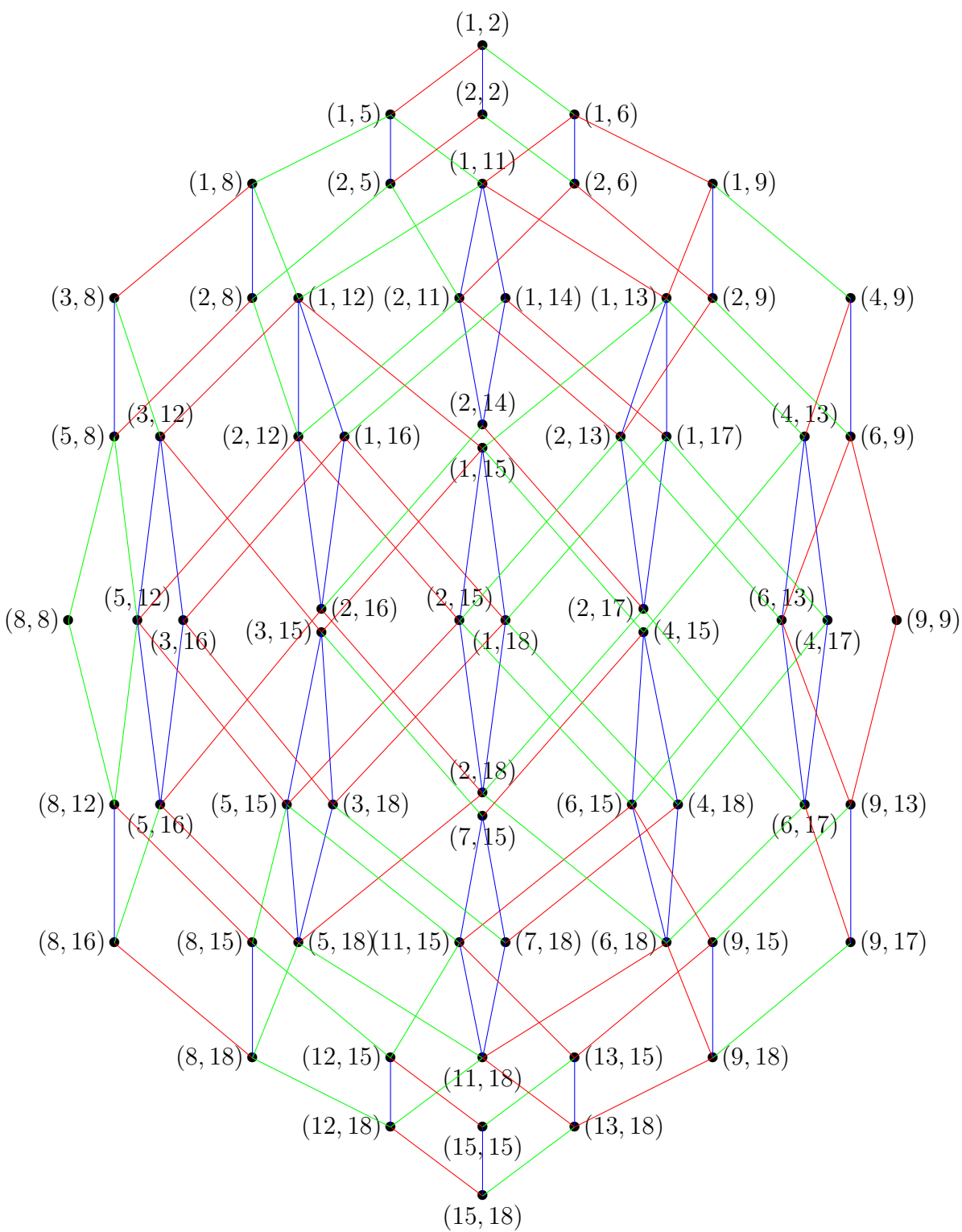
Alternatively, the character $\chi^{66,1}$ can be obtained as the combinatorial fusion:
$$
\chi^{66,1}=\chi^{18}*\tau_{p_1}\chi^{18}.
$$

We describe the monomials and the graph in Figure \ref{66 pic}. 
In this graph we use a notation $(i,j)=v_i\tau_{p_1}(v_j)$, where $v_i$ are monomials in the adjoint $qq$-character $\chi^{18}$, see Figure \ref{18 pic}. All arrows are directed downwards. The corresponding shifts of $A_i^{-1}$ are readily found from the picture of the adjoint $qq$-character as well.

The graph has a reflection symmetry across the vertical line connecting the nodes $(1,2)$ and $(15,18)$ which changes colors $2$ and $3$ and variables $q_2$ and $q_3$. To visualize it one has to imagine left vertices of the short diagonals of blue parallelograms to be below others and the right ones above.

The graph has also a reflection symmetry across the horizontal line connecting the nodes $(8,8)$ and $(9,9)$. This symmetry does not change colors but changes the orientation of edges.

By the depth counting, the 66 term $qq$-character has the decomposition $$66=1+3+5+8+10+12+10+8+5+3+1.$$

\medskip

The $\gl_{2,1}$ restriction $\rho^{\{2,3\}}\chi^{66,1}$ decomposes by the rule
$$
66=25+16+16+9,
$$
where both $16$ are prime squares while $25$ and $9$ are non-prime squares obtained by generic product of two segments.

The $\gl_{2,1}$ restrictions $\rho^{\{1,2\}}\chi^{66,1}$ and $\rho^{\{1,3\}}\chi^{66,1}$ decompose by the rule
$$
66 = 16+12+12+9+4+4+4+4+1.
$$
Here the 16 is a generic product of two squares, the twelves are products of a square by the vector $qq$-characters, the fours are squares.

\bigskip

\subsection{130}\label{130 sec}
In this section we discuss the $qq$-character with the dominant monomial
$$
m_+^{130}=m_+^{18} \ \tau_{q_0^{2}}(m_+^{18}) =\bs 1_{q_1,\, q^2 q_1}^{q_1^{-1},\,q^2 q_1^{-1}}\bs 2_{q_2,\, q^2 q_2}^{q_2^{-1},\,q^2 q_2^{-1}}\bs 3_{q_3,\, q^2 q_3}^{q_3^{-1},\,q^2 q_3^{-1}}.
$$

Then the algorithm produces a finite degree zero $qq$-character $\chi^{130}$ with 130 terms. 

Alternatively, the character $\chi^{130}$ can be obtained as the combinatorial fusion:
$$
\chi^{130,1}=\chi^{18}*\tau_{q_0^2}\chi^{18}.
$$

The $qq$-character $\chi^{130}$  is clearly not slim and not linear. It does have the symmetry with respect to permutations of colors together with the variables.

The graph of that $qq$-character is too large to picture.
By the depth counting it decomposes as
$$
130=1+3+6+10+15+18+24+18+15+10+6+3+1.
$$
Instead we describe the restriction $\rho_{\{2,3\}}\chi^{130}$.
As a $\gl_{2,1}$ $qq$-character we have the decomposition
\begin{align*}
130&=2(1^2+2^2+3^2+5^2)+4^2+6^2
\end{align*}

In this section we abbreviate $(i,j)=v_i  \tau_{q^{2}}(v_j)$, where $v_i$ are monomials from the adjoint 18-term $qq$-character. Also, all arrows are oriented downwards. 

The $6^2$ term  $qq$-character is given in Figure \ref{36 pic}.
Note that the graph has a symmetry interchanging colors $2$ and $3$ which corresponds to the reflection about the vertical line through nodes $(2,2)$ and $(15,15)$. The graph has also a symmetry about the horizontal line connecting $(8,8)$ and $(9,9)$ (the vertices $(15,2)$ and $(11,11)$ should be placed to the line and vertices $(12,6)$ and $(13,5)$ should be moved off the line to be images of each other).

\begin{figure} [t]
\begin{center}
\begin{tikzpicture}[scale=0.4]
\coordinate  (A) at  (0,12);

\coordinate  (B1) at (-2,9);
\coordinate  (B2) at (2,9);

\coordinate  (C2) at (-7.5,6);
\coordinate  (C1) at (-3.7,6);
\coordinate  (C3) at (0,6);

\coordinate  (C5) at (3.7,6);
\coordinate  (C4) at (7.5,6);

\coordinate  (D1) at (-10,3);

\coordinate  (D3) at (-6,3);
\coordinate  (D2) at (-2,3);

\coordinate  (D5) at (2,3);
\coordinate  (D4) at (6,3);
\coordinate  (D6) at (10,3);

\coordinate  (E1) at (-12,0);
\coordinate  (E2) at (-8,0);
\coordinate  (E3) at (-4,0);
\coordinate  (E5) at (0,-0.5);
\coordinate  (E6) at (4,0);
\coordinate  (E4) at (0,0.5);
\coordinate  (E7) at (8,0);
\coordinate  (E8) at (12,0);

\coordinate  (F1) at (-10,-3);
\coordinate  (F2) at (-6,-3);
\coordinate  (F3) at (-2,-3);
\coordinate  (F4) at (2,-3);
\coordinate  (F5) at (6,-3);
\coordinate  (F6) at (10,-3);

\coordinate  (G1) at (-7.5,-6);
\coordinate  (G2) at (-3.7,-6);
\coordinate  (G3) at (0,-6);
\coordinate  (G4) at (3.7,-6);
\coordinate  (G5) at (7.5,-6);

\coordinate  (H1) at (-2,-9);
\coordinate  (H2) at (2,-9);

\coordinate  (I) at (0,-12);

\node [above] at (A) {$(2,2)$};
\node at  (A) {$\bullet$};

\node [left] at  (B1) {$(5,2)$};
\node  at  (B1) {$\bullet$};
\node [above] at  (B2) {$(6,2)$};
\node at (B2) {$\bullet$};

\node [left]at  (C1) {$(8,2)$};
\node at  (C1) {$\bullet$};
\node  [above] at (C2) {$(5,5)$};
\node at  (C2) {$\bullet$};
\node [below] at  (C3) {$(11,2)$};
\node at  (C3) {$\bullet$};
\node [above] at  (C4) {$(6,6)$};
\node at  (C4) {$\bullet$};
\node [above] at  (C5) {$(9,2)$};
\node at  (C5) {$\bullet$};

\node [left] at  (D1) {$(8,5)$};
\node at  (D1) {$\bullet$};
\node [above] at  (D2) {$(12,2)$};
\node at  (D2) {$\bullet$};
\node [below] at  (D3) {$(11,5)$};
\node at  (D3) {$\bullet$};
\node [right] at (D4) {$(11,6)$};
\node at  (D4) {$\bullet$};
\node [right] at (D5) {$(13,2)$};
\node at  (D5) {$\bullet$};
\node [right] at (D6) {$(9,6)$};
\node at  (D6) {$\bullet$};

\node [left] at  (E1) {$(8,8)$};
\node at  (E1) {$\bullet$};
\node [below] at  (E2) {$(12,5)$};
\node at  (E2) {$\bullet$};
\node [below] at  (E3) {$(12,6)$};
\node at  (E3) {$\bullet$};
\node [left] at (E4) {$(15,2)$};
\node at  (E4) {$\bullet$};
\node [right] at (E5) {$(11,11)$};
\node at  (E5) {$\bullet$};
\node [right] at (E6) {$(13,5)$};
\node at  (E6) {$\bullet$};
\node [right] at (E7) {$(13,6)$};
\node at  (E7) {$\bullet$};
\node [right] at (E8) {$(9,9)$};
\node at  (E8) {$\bullet$};

\node [left] at  (F1) {$(12,8)$};
\node at  (F1) {$\bullet$};
\node [below] at  (F2) {$(12,11)$};
\node at  (F2) {$\bullet$};
\node [below] at  (F3) {$(15,5)$};
\node at  (F3) {$\bullet$};
\node [right] at (F4) {$(15,6)$};
\node at  (F4) {$\bullet$};
\node [right] at (F5) {$(13,11)$};
\node at  (F5) {$\bullet$};
\node [right] at (F6) {$(13,9)$};
\node at  (F6) {$\bullet$};

\node [left]at  (G1) {$(12,12)$};
\node at  (G1) {$\bullet$};
\node  [above] at (G2) {$(15,8)$};
\node at  (G2) {$\bullet$};
\node [below] at  (G3) {$(15,11)$};
\node at  (G3) {$\bullet$};
\node [above] at  (G4) {$(15,9)$};
\node at  (G4) {$\bullet$};
\node [above] at  (G5) {$(13,13)$};
\node at  (G5) {$\bullet$};

\node [below] at  (H1) {$(15,12)$};
\node at  (H1) {$\bullet$};
\node [below] at  (H2) {$(15,13)$};
\node at  (H2) {$\bullet$};

\node [below] at  (I) {$(15,15)$};
\node at  (I) {$\bullet$};

\draw[red] (A)-- (B1);
\draw[green] (A)-- (B2);

\draw[green] (B1)--(C1);
\draw[red] (B1)--(C2);
\draw[green] (B1)--(C3);
\draw[red] (B2)--(C3);
\draw[green] (B2)--(C4);
\draw[red] (B2)--(C5);

\draw[red] (C1) -- (D1);
\draw[green] (C1) -- (D2);
\draw[green] (C2) -- (D1);
\draw[green] (C2) -- (D3);
\draw[green] (C3) -- (D2);
\draw[red] (C3) -- (D3);
\draw[green] (C3) -- (D4);
\draw[red] (C3) -- (D5);
\draw[red] (C4) -- (D4);
\draw[red] (C4) -- (D6);
\draw[red] (C5) -- (D5);
\draw[green] (C5) -- (D6);

\draw[green] (D1) -- (E1);
\draw[green] (D1) -- (E2);
\draw[red] (D2) -- (E2);
\draw[green] (D2) -- (E3);
\draw[red] (D2) -- (E4);
\draw[green] (D3) -- (E2);
\draw[green] (D3) -- (E5);
\draw[red] (D3) -- (E6);
\draw[green] (D4) -- (E3);
\draw[red] (D4) -- (E5);
\draw[red] (D4) -- (E7);
\draw[green] (D5) -- (E4);
\draw[red] (D5) -- (E6);
\draw[green] (D5) -- (E7);
\draw[red] (D6) -- (E7);
\draw[red] (D6) -- (E8);

\draw[green] (E1) -- (F1);
\draw[green] (E2) -- (F1);
\draw[green] (E2) -- (F2);
\draw[red] (E2) -- (F3);
\draw[red] (E3) -- (F2);
\draw[red] (E3) -- (F4);
\draw[red] (E4) -- (F3);
\draw[green] (E4) -- (F4);
\draw[green] (E5) -- (F2);
\draw[red] (E5) -- (F5);
\draw[green] (E6) -- (F3);
\draw[green] (E6) -- (F5);
\draw[green] (E7) -- (F4);
\draw[red] (E7) -- (F5);
\draw[red] (E7) -- (F6);
\draw[red] (E8) -- (F6);

\draw[green] (F1) -- (G1);
\draw[red] (F1) -- (G2);
\draw[green] (F2) -- (G1);
\draw[red] (F2) -- (G3);
\draw[green] (F3) -- (G2);
\draw[green] (F3) -- (G3);
\draw[red] (F4) -- (G3);
\draw[red] (F4) -- (G4);
\draw[green] (F5) -- (G3);
\draw[red] (F5) -- (G5);
\draw[green] (F6) -- (G4);
\draw[red] (F6) -- (G5);

\draw[red] (G1) -- (H1);
\draw[green] (G2) -- (H1);
\draw[green] (G3) -- (H1);
\draw[red] (G3) -- (H2);
\draw[red] (G4) -- (H2);
\draw[green] (G5) -- (H2);

\draw[red] (H1) -- (I);
\draw[green] (H2) -- (I);

\end{tikzpicture}
\caption{The 36 term $\gl_{2,1}$ $qq$-character.}\label{36 pic}
\end{center}
\end{figure}
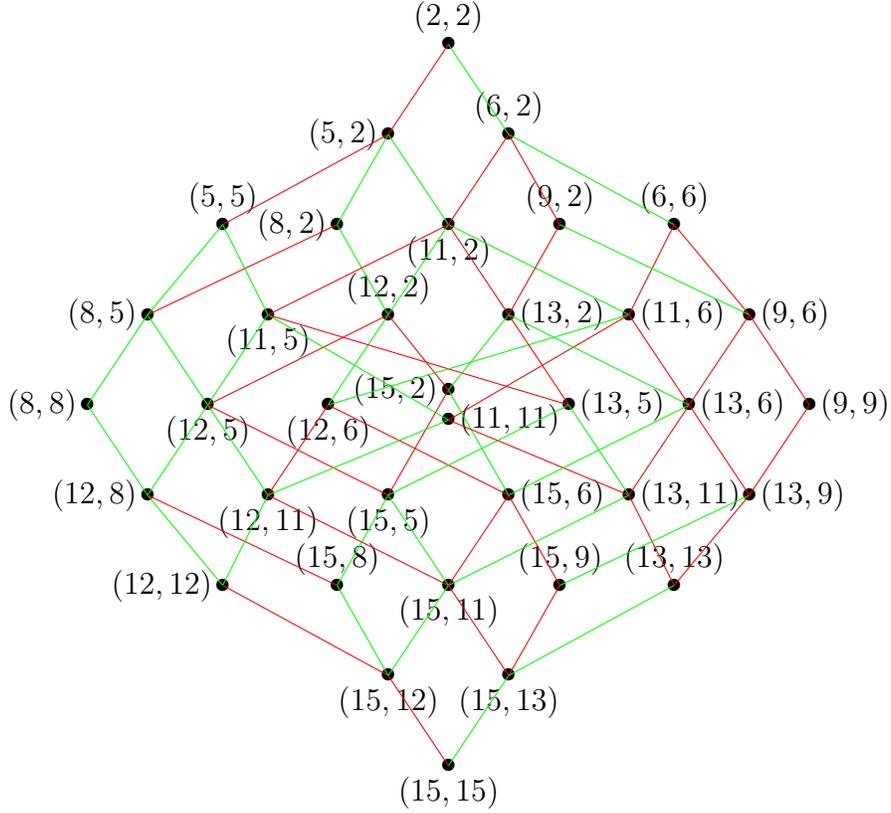

The two $5^2$ term  $qq$-characters are given in Figure \ref{25 2 pic}.
Note that the graphs again have symmetries interchanging colors $2$ and $3$ which correspond to the reflection about the vertical line through nodes $(2,1)$, $(15,7)$ for the first graph and through $(14,2)$, $(18,15)$ for the second one. The second graph is the reflection of the first one with respect the horizontal line through vertices $(12,3)$ and $(13,4)$.

\begin{figure} [ht]
\begin{center}
\begin{tikzpicture}[scale=0.4]

\coordinate  (A) at  (0,7);

\coordinate  (B1) at (-2,4);
\coordinate  (B2) at (2,4);

\coordinate  (C1) at (-7,1);
\coordinate  (C2) at (-3.5,1);
\coordinate  (C3) at (0,1);
\coordinate  (C4) at (3.5,1);
\coordinate  (C5) at (7,1);

\coordinate  (D1) at (-9,-3);
\coordinate  (D2) at (-5,-2.5);
\coordinate  (D3) at (-1.5,-3);
\coordinate  (D4) at (1.5,-3);
\coordinate  (D5) at (5,-2.5);
\coordinate  (D6) at (9,-3);

\coordinate  (E1) at (-9,-6);
\coordinate  (E2) at (-5,-6);
\coordinate  (E4) at (0,-6.5);
\coordinate  (E3) at (0,-5.5);
\coordinate  (E5) at (5,-6);
\coordinate  (E6) at (9,-6);

\coordinate  (F1) at (-6,-9);
\coordinate  (F2) at (-2,-9);
\coordinate  (F3) at (2,-9);
\coordinate  (F4) at (6,-9);

\coordinate  (G) at (0,-12);

\node [above] at (A) {$(2,1)$};
\node at  (A) {$\bullet$};

\node [left] at  (B1) {$(5,1)$};
\node  at  (B1) {$\bullet$};
\node [above] at  (B2) {$(6,1)$};
\node at (B2) {$\bullet$};

\node [left] at  (C1) {$(8,1)$};
\node  at  (C1) {$\bullet$};
\node [above] at  (C2) {$(5,3)$};
\node at (C2) {$\bullet$};
\node  [above] at (C3) {$(11,1)$};
\node at  (C3) {$\bullet$};
\node  [right] at (C4) {$(6,4)$};
\node at (C4) {$\bullet$};
\node  [right] at (C5) {$(9,1)$};
\node at (C5) {$\bullet$};

\node [left]at  (D1) {$(8,3)$};
\node at  (D1) {$\bullet$};
\node  [above] at (D2) {$(12,1)$};
\node at  (D2) {$\bullet$};
\node [below] at  (D3) {$(11,3)$};
\node at  (D3) {$\bullet$};
\node [above] at  (D4) {$(11,4)$};
\node at  (D4) {$\bullet$};
\node [above] at  (D5) {$(13,1)$};
\node at  (D5) {$\bullet$};
\node [right] at  (D6) {$(9,4)$};
\node at  (D6) {$\bullet$};

\node [left]at  (E1) {$(12,3)$};
\node at  (E1) {$\bullet$};
\node  [above] at (E2) {$(12,4)$};
\node at  (E2) {$\bullet$};
\node [right] at  (E3) {$(11,7)$};
\node at  (E3) {$\bullet$};
\node [left] at  (E4) {$(15,1)$};
\node at  (E4) {$\bullet$};
\node [above] at  (E5) {$(13,3)$};
\node at  (E5) {$\bullet$};
\node [right] at  (E6) {$(13,4)$};
\node at  (E6) {$\bullet$};

\node [left] at  (F1) {$(12,7)$};
\node at  (F1) {$\bullet$};
\node [below] at  (F2) {$(15,3)$};
\node at  (F2) {$\bullet$};
\node [below] at  (F3) {$(15,4)$};
\node at  (F3) {$\bullet$};
\node [right] at (F4) {$(13,7)$};
\node at  (F4) {$\bullet$};

\node [below] at  (G) {$(15,7)$};
\node at  (G) {$\bullet$};

\draw[red] (A)-- (B1);
\draw[green] (A)-- (B2);

\draw[green] (B1)--(C1);
\draw[red] (B1)--(C2);
\draw[green] (B1)--(C3);
\draw[red] (B2)--(C3);
\draw[green] (B2)--(C4);
\draw[red] (B2)--(C5);

\draw[red] (C1) -- (D1);
\draw[green] (C1) -- (D2);
\draw[green] (C2) -- (D1);
\draw[green] (C2) -- (D3);
\draw[green] (C3) -- (D2);
\draw[red] (C3) -- (D3);
\draw[green] (C3) -- (D4);
\draw[red] (C3) -- (D5);
\draw[red] (C4) -- (D4);
\draw[red] (C4) -- (D6);
\draw[red] (C5) -- (D5);
\draw[green] (C5) -- (D6);

\draw[green] (D1) -- (E1);
\draw[red] (D2) -- (E1);
\draw[green] (D2) -- (E2);
\draw[red] (D2) -- (E4);
\draw[green] (D3) -- (E1);
\draw[green] (D3) -- (E3);
\draw[red] (D3) -- (E5);
\draw[green] (D4) -- (E2);
\draw[red] (D4) -- (E3);
\draw[red] (D4) -- (E6);
\draw[green] (D5) -- (E4);
\draw[red] (D5) -- (E5);
\draw[green] (D5) -- (E6);
\draw[red] (D6) -- (E6);

\draw[green] (E1) -- (F1);
\draw[red] (E1) -- (F2);
\draw[red] (E2) -- (F1);
\draw[red] (E2) -- (F3);
\draw[green] (E3) -- (F1);
\draw[red] (E3) -- (F4);
\draw[red] (E4) -- (F2);
\draw[green] (E4) -- (F3);
\draw[green] (E5) -- (F2);
\draw[green] (E5) -- (F4);
\draw[green] (E6) -- (F3);
\draw[red] (E6) -- (F4);

\draw[red] (F1) -- (G);
\draw[green] (F2) -- (G);
\draw[red] (F3) -- (G);
\draw[green] (F4) -- (G);
\end{tikzpicture}
\end{center}

\bigskip

\begin{center}
\begin{tikzpicture}[scale=0.4]

\coordinate  (A) at  (0,7);
\node [above] at (A) {$(14,2)$};
\node at  (A) {$\bullet$};

\coordinate  (B1) at (-6,4);
\coordinate  (B2) at (-2,4);
\coordinate  (B3) at (2,4);
\coordinate  (B4) at (6,4);
\node [left] at  (B1) {$(14,5)$};
\node at  (B1) {$\bullet$};
\node [below] at  (B2) {$(16,2)$};
\node at  (B2) {$\bullet$};
\node [below] at  (B3) {$(17,2)$};
\node at  (B3) {$\bullet$};
\node [right] at (B4) {$(14,6)$};
\node at  (B4) {$\bullet$};

\coordinate  (C1) at (-8,1);
\coordinate  (C2) at (-4.5,1);
\coordinate  (C3) at (0,0.5);
\coordinate  (C4) at (0,1.5);
\coordinate  (C5) at (4,1);
\coordinate  (C6) at (8,1);
\node [left] at  (C1) {$(16,5)$};
\node  at  (C1) {$\bullet$};
\node [above] at  (C2) {$(17,5)$};
\node at (C2) {$\bullet$};
\node  [left] at (C3) {$(14,11)$};
\node at  (C3) {$\bullet$};
\node  [right] at (C4) {$(18,2)$};
\node at (C4) {$\bullet$};
\node  [right] at (C5) {$(16,6)$};
\node at (C5) {$\bullet$};
\node  [right] at (C6) {$(17,6)$};
\node at (C6) {$\bullet$};

\coordinate  (D1) at (-9,-3);
\coordinate  (D2) at (-5,-2.5);
\coordinate  (D3) at (-1.5,-3);
\coordinate  (D4) at (1.5,-3);
\coordinate  (D5) at (5,-2.5);
\coordinate  (D6) at (9,-3);
\node [left]at  (D1) {$(16,8)$};
\node at  (D1) {$\bullet$};
\node  [above] at (D2) {$(18,5)$};
\node at  (D2) {$\bullet$};
\node [below] at  (D3) {$(16,11)$};
\node at  (D3) {$\bullet$};
\node [above] at  (D4) {$(17,11)$};
\node at  (D4) {$\bullet$};
\node [above] at  (D5) {$(18,6)$};
\node at  (D5) {$\bullet$};
\node [right] at  (D6) {$(17,9)$};
\node at  (D6) {$\bullet$};

\coordinate  (E1) at (-6,-6);
\coordinate  (E2) at (-3,-6);
\coordinate  (E3) at (0,-6);
\coordinate  (E4) at (3,-6);
\coordinate  (E5) at (6,-6);
\node [left]at  (E1) {$(18,8)$};
\node at  (E1) {$\bullet$};
\node  [above] at (E2) {$(16,12)$};
\node at  (E2) {$\bullet$};
\node [below] at  (E3) {$(18,11)$};
\node at  (E3) {$\bullet$};
\node [above] at  (E4) {$(17,13)$};
\node at  (E4) {$\bullet$};
\node [above] at  (E5) {$(18,9)$};
\node at  (E5) {$\bullet$};

\coordinate  (F1) at (-2,-9);
\coordinate  (F2) at (2,-9);
\node [left] at  (F1) {$(18,12)$};
\node  at  (F1) {$\bullet$};
\node [above] at  (F2) {$(18,13)$};
\node at (F2) {$\bullet$};

\coordinate  (G) at (0,-12);
\node [below] at  (G) {$(18,15)$};
\node at  (G) {$\bullet$};

\draw[red] (A)-- (B1);
\draw[green] (A)-- (B2);
\draw[red] (A)--(B3);
\draw[green] (A)--(B4);

\draw[green] (B1)--(C1);
\draw[red] (B1)--(C2);
\draw[green] (B1)--(C3);
\draw[red] (B2)--(C1);
\draw[red] (B2)--(C4);
\draw[green] (B2)--(C5);
\draw[red] (B3)--(C2);
\draw[green] (B3)--(C4);
\draw[green] (B3)--(C6);
\draw[red] (B4)--(C3);
\draw[green] (B4)--(C5);
\draw[red] (B4)--(C6);

\draw[green] (C1) -- (D1);
\draw[red] (C1) -- (D2);
\draw[green] (C1) -- (D3);
\draw[green] (C2) -- (D2);
\draw[green] (C2) -- (D4);
\draw[green] (C3) -- (D3);
\draw[red] (C3) -- (D4);
\draw[red] (C4) -- (D2);
\draw[green] (C4) -- (D5);
\draw[red] (C5) -- (D3);
\draw[red] (C5) -- (D5);
\draw[red] (C6) -- (D4);
\draw[green] (C6) -- (D5);
\draw[red] (C6) -- (D6);

\draw[red] (D1) -- (E1);
\draw[green] (D1) -- (E2);
\draw[green] (D2) -- (E1);
\draw[green] (D2) -- (E3);
\draw[green] (D3) -- (E2);
\draw[red] (D3) -- (E3);
\draw[green] (D4) -- (E3);
\draw[red] (D4) -- (E4);
\draw[red] (D5) -- (E3);
\draw[red] (D5) -- (E5);
\draw[red] (D6) -- (E4);
\draw[green] (D6) -- (E5);

\draw[green] (E1) -- (F1);
\draw[red] (E2) -- (F1);
\draw[green] (E3) -- (F1);
\draw[red] (E3) -- (F2);
\draw[green] (E4) -- (F2);
\draw[red] (E5) -- (F2);

\draw[red] (F1) -- (G);
\draw[green] (F2) -- (G);

\end{tikzpicture}
\caption{Two $25$ term $\gl_{2,1}$ $qq$-characters.}\label{25 2 pic}
\end{center}
\end{figure}
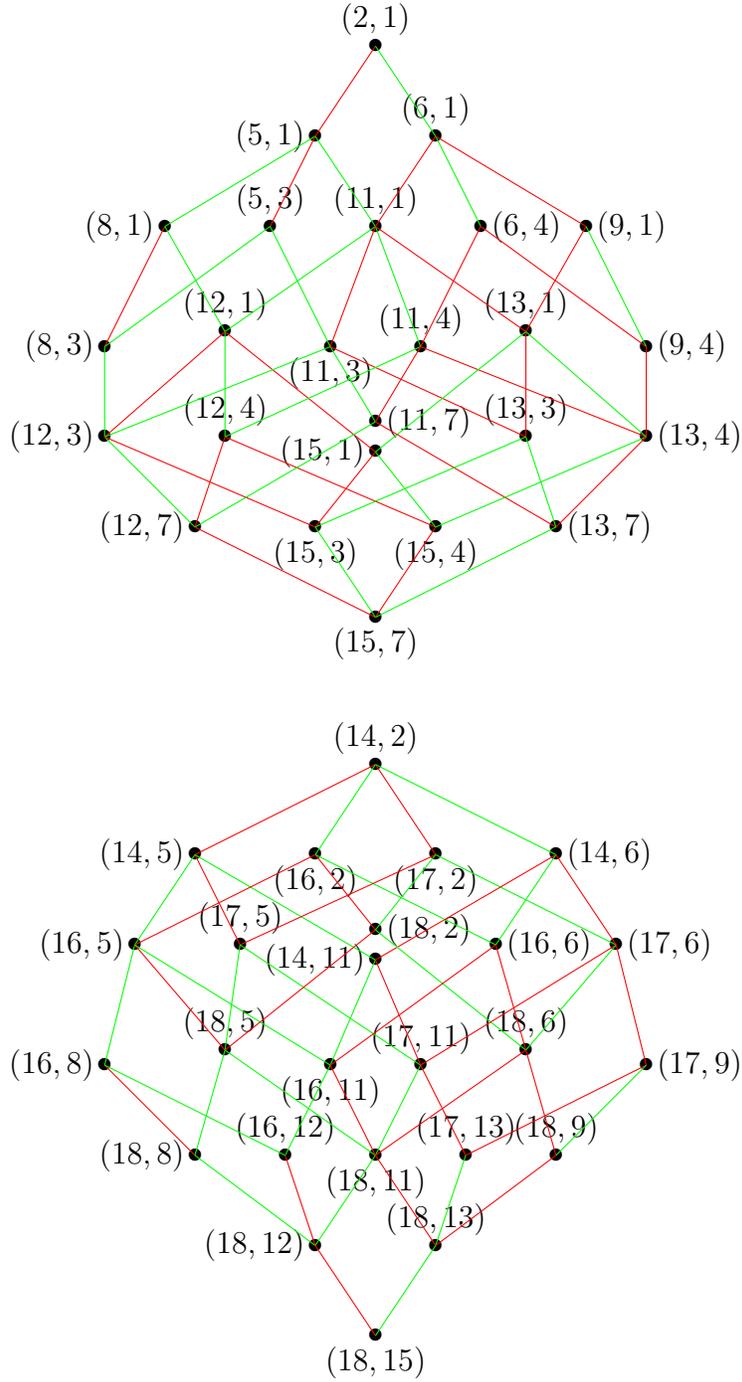

\bigskip

The $16$-term $qq$-character is just a generic product of two squares. 
The two $9$-term $qq$-characters are non-slim squares. Then we have two squares and two trivial $qq$-characters. All of them with the corresponding monomials are given in Figure \ref{small in 130 pic}.

\begin{figure} [ht]
\begin{center}
\begin{tikzpicture}[scale=0.4]
\coordinate  (A) at  (0,8);

\coordinate  (B1) at (-8,5);
\coordinate  (B2) at (-3,5);
\coordinate  (B3) at (3,5);
\coordinate  (B4) at (8,5);

\coordinate  (C1) at (-8,0);
\coordinate  (C2) at (-4,0);
\coordinate  (C3) at (-0.5,0);
\coordinate  (C4) at (0.5,0);
\coordinate  (C5) at (4,0);
\coordinate  (C6) at (8,0);

\coordinate  (D1) at (-8,-5);
\coordinate  (D2) at (-3.5,-5);
\coordinate  (D3) at (3.5,-5);
\coordinate  (D4) at (8,-5);

\coordinate  (E) at (0,-8);

\node [above] at (A) {$(14,1)$};
\node at  (A) {$\bullet$};

\node [left] at  (B1) {$(14,3)$};
\node  at  (B1) {$\bullet$};
\node [above] at  (B2) {$(16,1)$};
\node at (B2) {$\bullet$};
\node  [above] at (B3) {$(17,1)$};
\node at  (B3) {$\bullet$};
\node  [right] at (B4) {$(14,4)$};
\node at (B4) {$\bullet$};

\node [left]at  (C1) {$(16,3)$};
\node at  (C1) {$\bullet$};
\node  [above] at (C2) {$(17,3)$};
\node at  (C2) {$\bullet$};
\node [below] at  (C3) {$(18,1)$};
\node at  (C3) {$\bullet$};
\node [above] at  (C4) {$(14,7)$};
\node at  (C4) {$\bullet$};
\node [above] at  (C5) {$(16,4)$};
\node at  (C5) {$\bullet$};
\node [right] at  (C6) {$(17,4)$};
\node at  (C6) {$\bullet$};

\node [left] at  (D1) {$(16,7)$};
\node at  (D1) {$\bullet$};
\node [below] at  (D2) {$(18,3)$};
\node at  (D2) {$\bullet$};
\node [below] at  (D3) {$(17,4)$};
\node at  (D3) {$\bullet$};
\node [right] at (D4) {$(17,7)$};
\node at  (D4) {$\bullet$};

\node [below] at  (E) {$(18,7)$};
\node at  (E) {$\bullet$};

\draw[red] (A)-- (B1);
\draw[red] (A)-- (B3);
\draw[green] (A)-- (B2);
\draw[green] (A)-- (B4);

\draw[green] (B1)--(C1);
\draw[red] (B1)--(C2);
\draw[green] (B1)--(C4);
\draw[red] (B2)--(C1);
\draw[red] (B2)--(C3);
\draw[green] (B2)--(C5);
\draw[red] (B3)--(C2);
\draw[green] (B3)--(C3);
\draw[green] (B3)--(C6);
\draw[red] (B4)--(C4);
\draw[green] (B4)--(C5);
\draw[red] (B4)--(C6);

\draw[green] (C1) -- (D1);
\draw[red] (C1) -- (D2);
\draw[green] (C2) -- (D2);
\draw[green] (C2) -- (D4);
\draw[red] (C3) -- (D2);
\draw[green] (C3) -- (D3);
\draw[green] (C4) -- (D1);
\draw[red] (C4) -- (D4);
\draw[red] (C5) -- (D1);
\draw[red] (C5) -- (D3);
\draw[green] (C6) -- (D3);
\draw[red] (C6) -- (D4);

\draw[red] (D1) -- (E);
\draw[green] (D2) -- (E);
\draw[red] (D3) -- (E);
\draw[green] (D4) -- (E);

\end{tikzpicture}

\end{center}

\bigskip

\begin{center}
\begin{tabular}{c}

\begin{minipage}{0.33\hsize}
\begin{center}
\begin{tikzpicture}[scale=0.4]
\coordinate  (A) at  (0,4);
\node [above] at (A) {$(1,1)$};
\node at  (A) {$\bullet$};

\coordinate  (B1) at  (-2,2);
\coordinate  (B2) at  (2,2);
\node [left] at (B1) {$(3,1)$};
\node at  (B1) {$\bullet$};
\node [right] at (B2) {$(4,1)$};
\node at  (B2) {$\bullet$};

\coordinate  (C1) at  (-4,0);
\coordinate  (C2) at  (0,0);
\coordinate  (C3) at  (4,0);
\node [left] at (C1) {$(3,3)$};
\node at  (C1) {$\bullet$};
\node [right] at (C2) {$(7,1)$};
\node at  (C2) {$\bullet$};
\node [right] at (C3) {$(4,4)$};
\node at  (C3) {$\bullet$};

\coordinate  (D1) at  (-2,-2);
\coordinate  (D2) at  (2,-2);
\node [left] at (D1) {$(7,3)$};
\node at  (D1) {$\bullet$};
\node [right] at (D2) {$(7,4)$};
\node at  (D2) {$\bullet$};

\coordinate  (E) at  (0,-4);
\node [below] at (E) {$(7,7)$};
\node at  (E) {$\bullet$};

\draw[red] (A) -- (B1);
\draw[green] (A) -- (B2);

\draw[red] (B1) -- (C1);
\draw[green] (B1) -- (C2);
\draw[red] (B2) -- (C2);
\draw[green] (B2) -- (C3);

\draw[green] (C1) -- (D1);
\draw[red] (C2) -- (D1);
\draw[green] (C2) -- (D2);
\draw[red] (C3) -- (D2);

\draw[green] (D1) -- (E);
\draw[red] (D2) -- (E);

\end{tikzpicture}
\end{center}
\end{minipage}

\begin{minipage}{0.33\hsize}
\begin{center}
\begin{tikzpicture}[scale=0.4]
\coordinate  (A) at  (0,4);
\node [above] at (A) {$(14,14)$};
\node at  (A) {$\bullet$};

\coordinate  (B1) at  (-2,2);
\coordinate  (B2) at  (2,2);
\node [left] at (B1) {$(16,14)$};
\node at  (B1) {$\bullet$};
\node [right] at (B2) {$(17,14)$};
\node at  (B2) {$\bullet$};

\coordinate  (C1) at  (-4,0);
\coordinate  (C2) at  (0,0);
\coordinate  (C3) at  (4,0);
\node [left] at (C1) {$(16,16)$};
\node at  (C1) {$\bullet$};
\node [right] at (C2) {$(18,14)$};
\node at  (C2) {$\bullet$};
\node [right] at (C3) {$(17,17)$};
\node at  (C3) {$\bullet$};

\coordinate  (D1) at  (-2,-2);
\coordinate  (D2) at  (2,-2);
\node [left] at (D1) {$(18,16)$};
\node at  (D1) {$\bullet$};
\node [right] at (D2) {$(18,17)$};
\node at  (D2) {$\bullet$};

\coordinate  (E) at  (0,-4);
\node [below] at (E) {$(18,18)$};
\node at  (E) {$\bullet$};

\draw[green] (A) -- (B1);
\draw[red] (A) -- (B2);

\draw[green] (B1) -- (C1);
\draw[red] (B1) -- (C2);
\draw[green] (B2) -- (C2);
\draw[red] (B2) -- (C3);

\draw[red] (C1) -- (D1);
\draw[green] (C2) -- (D1);
\draw[red] (C2) -- (D2);
\draw[green] (C3) -- (D2);

\draw[red] (D1) -- (E);
\draw[green] (D2) -- (E);

\end{tikzpicture}
\end{center}
\end{minipage}

\end{tabular}
\end{center}

\bigskip 

\begin{center}
\begin{tabular}{c}

\begin{minipage}{0.33\hsize}
\begin{center}
\begin{tikzpicture}[scale=0.4]
\coordinate  (A) at  (0,2);
\node [above] at (A) {$(10,1)$};
\node at  (A) {$\bullet$};

\coordinate  (B1) at  (-2,0);
\coordinate  (B2) at  (2,0);
\node [left] at (B1) {$(10,3)$};
\node at  (B1) {$\bullet$};
\node [right] at (B2) {$(10,4)$};
\node at  (B2) {$\bullet$};

\coordinate  (C) at  (0,-2);
\node [below] at (C) {$(10,7)$};
\node at  (C) {$\bullet$};

\draw[red] (A) -- (B1);
\draw[green] (A) -- (B2);

\draw[green] (B1) -- (C);
\draw[red] (B2) -- (C);
\end{tikzpicture}
\end{center}
\end{minipage}

\begin{minipage}{0.33\hsize}
\begin{center}
\begin{tikzpicture}[scale=0.4]
\coordinate  (A) at  (0,2);
\node [above] at (A) {$(14,10)$};
\node at  (A) {$\bullet$};

\coordinate  (B1) at  (-2,0);
\coordinate  (B2) at  (2,0);
\node [left] at (B1) {$(16,10)$};
\node at  (B1) {$\bullet$};
\node [right] at (B2) {$(17,10)$};
\node at  (B2) {$\bullet$};

\coordinate  (C) at  (0,-2);
\node [below] at (C) {$(18,10)$};
\node at  (C) {$\bullet$};

\draw[green] (A) -- (B1);
\draw[red] (A) -- (B2);

\draw[red] (B1) -- (C);
\draw[green] (B2) -- (C);
\end{tikzpicture}
\end{center}
\end{minipage}

\end{tabular}
\end{center}

\bigskip

\begin{center}
\begin{tabular}{c}
\begin{minipage}{0.33\hsize}
\begin{center}
\begin{tikzpicture}[scale=0.4]
\node [above] at (0,0) {$(10,0)$};
\node at (0,0) {$\bullet$};
\end{tikzpicture}
\end{center}
\end{minipage}

\begin{minipage}{0.33\hsize}
\begin{center}
\begin{tikzpicture}[scale=0.4]
\node [above] at (0,0) {$(1,18)$};
\node at (0,0) {$\bullet$};
\end{tikzpicture}
\end{center}
\end{minipage}
\end{tabular}
\caption{$16+9+9+4+4+1+1$ in the $130$.}\label{small in 130 pic}
\end{center}
\end{figure}
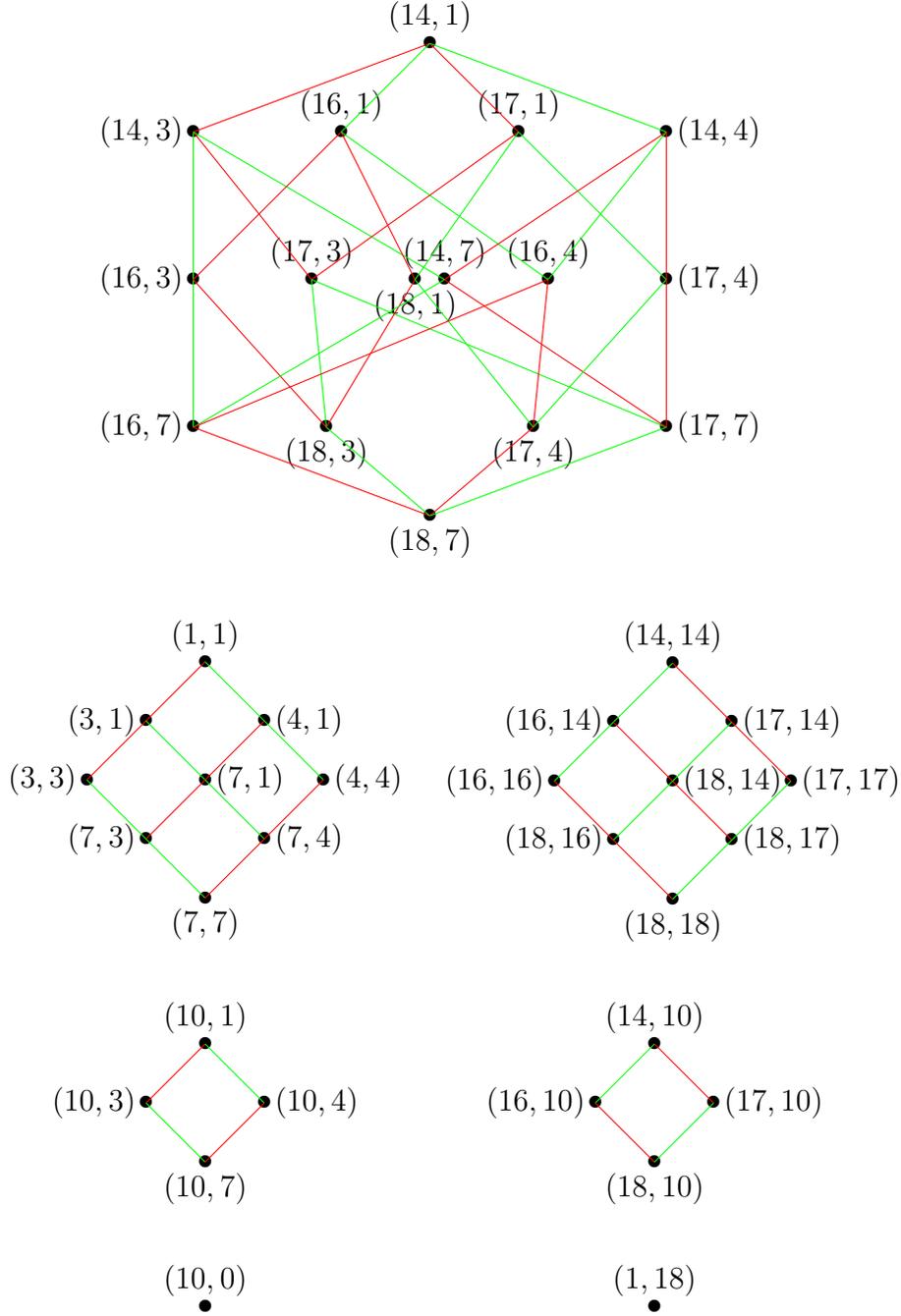

\subsection{Resonances}
In the previous sections we worked with generic parameters $q_0,q_1,q_2$. If the parameters satisfy some relation, the $qq$-characters truncate. Here we discuss a family of such examples.

We start with type D$(2,1;\alpha)$ vector $qq$-characters. 

Note that if $q_1=q_2$ then the D$(2,1;\alpha)$ Cartan matrix \eqref{D Cartan} becomes the $\mathfrak{osp}_{4,2}$ Cartan matrix \eqref{n|n osp Cartan} with colors $1$ and $3$ interchanged. In particular, the infinite vector $qq$-character $\chi^{312}$ of D$(2,1;\alpha)$ type starting from dominant monomial $\bs 3_{q_0^2q_1}^{q_1^{-1}}$, see Figure \ref{D vector pic}, truncates to a 6 term finite $qq$-character of $\mathfrak{osp}_{4,2}$ type, see Figure \ref{osp vector}.
This generalizes as follows.

Fix $k\in\Z_{>0}$ and assume that
\begin{align}\label{res 1}
p_1^k=p_2.
\end{align}
Then the monomials in \eqref{D vector monom} simplify. In particular, $V_{2k,0}^{312}=\bs 1_{q_0}^{q_0^{-1}q_2^{-2}}\bs 2_{q_3q_1^2}^{q_3^{-1}}$, does not have factors of color $3$ and 
$V_{2k+1,1}^{312}=\bs 3^{q_2}_{q_0^{-2}q_2^{-1}}$ does not have variables of colors $2$ and $3$. As a result the vector $qq$-character of D$(2,1;\alpha)$ type truncates and
we have a slim $qq$-character with $4k+2$ terms:
$\chi^{312}=\sum_{s=0}^{2k} V_{s,0}^{312}+\sum_{s=1}^{2k+1}V_{s,1}^{312}$.

The graph of this $qq$-character is given in Figure \ref{D vector truncated pic}.

\begin{figure} [ht]
\begin{center}
\begin{tikzpicture}[scale=0.6]

\node at  (4,4)  {\small $V^{312}_{1,1}$};
\node at  (12.5,4)  {$\dots$};
\node at  (-1,0)  {\small $V^{312}_{0,0}$};
\node at  (4,0)  {\small $V^{312}_{1,0}$};
\node at  (9,0)  {\small $V^{312}_{2,0}$};
\node at  (9,4)  {\small $V^{312}_{2,1}$};
\node at  (16,0)  {\small $V^{312}_{2k-1,0}$};
\node at  (16,4)  {\small $V^{312}_{2k-1,1}$};
\node at  (12.5,0)  {$\dots$};
\node at  (21,4)  {\small$V^{312}_{2k,1}$};
\node at  (21,0)  {\small $V^{312}_{2k,0}$};
\node at  (26,4)  {\small$V^{312}_{2k+1,1}$};

\draw[blue,->]  (4,0.7)--node[left]{{\small $A_{1,1}^{-1}$}}(4,3.3);
\draw[blue,->]  (9,0.7)--node[left]{{\small $A_{1,1}^{-1}$}}(9, 3.3);
\draw[blue,->]  (16,0.7)--node[left]{{\small $A_{1,1}^{-1}$}}(16, 3.3);
\draw[blue,->]  (21,0.7)--node[left]{{\small $A_{1,1}^{-1}$}}(21, 3.3);

\draw[red,->] (4.7,4)-- node[above]{{\small $A_{2,q_2q_1^{-1}}^{-1}$}}(8.3,4);
\draw[green] (9.7,4)-- (10.7,4);
\draw[green] (9.7,0)-- (10.7,0);
\draw[green,->] (14,4)-- (15,4);
\draw[green,->] (14,0)-- (15,0);

\draw[green,->] (-0.3,0)-- node[above]{{\small $A_{3,q_0q_2}^{-1}$}}(3.3,0);
\draw[red,->] (4.7,0)-- node[above]{\small $A_{2,q_2q_1^{-1}}^{-1}$}(8.3,0);

\draw[red,->] (17,0)-- node[above]{\small $A_{2,q_1q_2^{-1}}^{-1}$}(20.3,0);
\draw[red,->] (17,4)-- node[above]{\small $A_{2,q_1q_2^{-1}}^{-1}$}(20.3,4);
\draw[green,->] (21.7,4)-- node[above]{\small $A_{3,q_0^{-1}q_2^{-1}}^{-1}$}(25,4);
\end{tikzpicture}
\end{center}
\caption{The truncated vector D$(2,1;\alpha)$ $qq$-character.}\label{D vector truncated pic}
\end{figure}
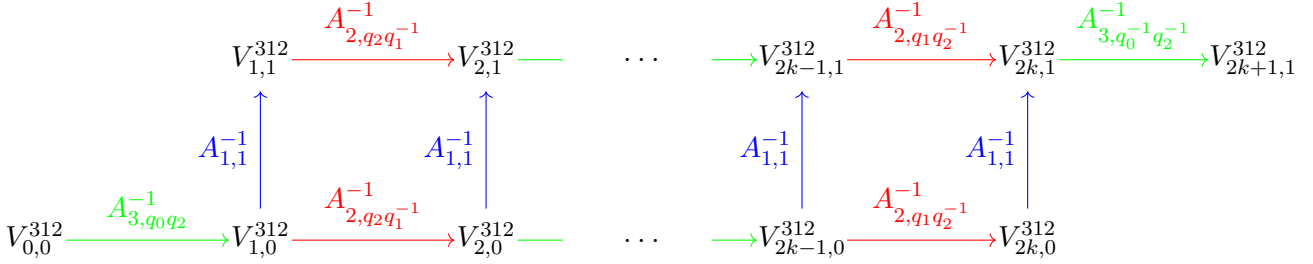

We also note that under resonance \eqref{res 1}, the $qq$-character $\hat\chi^{312}$ of $\hat {\textrm D}(2,1;\alpha)$, see Figure \ref{D hat vector pic}, also truncates to the $qq$-character
$\hat\chi=\sum_{0\leq a-b\leq 2k} \hat V_{ab}^{312}$, where $\hat V_{ab}^{312}$ are given by \eqref{D hat vector monom}.

On the other hand the adjoint $qq$-character $\chi^{18}$ does not truncate. Moreover, it can be obtained as combinatorial fusion:
$$
\chi^{18}=\chi^{312}_{q_0^{-1}q_1}* \tau_{q_2q_3}\chi^{312}.
$$
Note that with a specialization of parameters the combinatorial fusion product changes, as some non-constant monomials become constant.

In particular, for the top monomial we have
$$
m_+^{18}=\tau_{q_0^{-1}q_1}(V^{312}_{0,0})    \tau_{q_2q_3} (V^{312}_{2k-1,0}).
$$

\medskip

Another observation is that if 
 $q_2^2=1$ then the D$(2,1;\alpha)$ Cartan matrix \eqref{D Cartan} becomes the $\gl_{2,2}$ Cartan matrix \eqref{n|n Cartan}, with colors $1$ and $3$ interchanged. In particular, the infinite vector $qq$-character $\chi^{312}$ of D$(2,1;\alpha)$ type starting from dominant monomial $\bs 3_{q_0^2q_1}^{q_1^{-1}}$, see Figure \ref{D vector pic}, truncates to the 4 term finite vector $qq$-character of $\gl_{2,2}$ type, see Figure \ref{A vector pic}.
This also generalizes as follows.

Fix $k\in\Z_{>0}$ and assume that
\begin{align}\label{res 2}
    p_1^k=q_0^2.
\end{align}
Then we have the following cancellations: 
$V_{2k-1,0}=\bs 1_{q_0}^{q_0^{-1}q_3^{-2}}\bs 3^{q_2^{-1}}_{q_2^{-1}q_3^{-2}},$  and $V_{2k,1}=\bs 2_{q_3q_2^2}^{q_3}$.

As a result the vector $qq$-character of D$(2,1;\alpha)$ type truncates and
we have a slim $qq$-character with $4k$ terms:
$\chi^{312}=\sum_{s=0}^{2k-1} V_{s,0}^{312}+\sum_{s=1}^{2k}V_{s,1}^{312}$.

Moreover, in this case, the resonance condition \eqref{res 2} is invariant under swapping $q_2$ and $q_3$. Therefore, we have another slim $qq$-character $\chi^{213}$ with $4k$ terms obtained by truncation of $\chi^{213}$.

We also note that under resonance \eqref{res 2}, the $qq$-character $\hat \chi^{312}$ of $\hat {\textrm D}(2,1;\alpha)$, see Figure \ref{D hat vector pic}, also truncates to the $qq$-character
$\hat\chi=\sum_{0\leq a-b\leq 2k-1} \hat V_{ab}^{312}$, where $\hat V_{ab}^{312}$ are given by \eqref{D hat vector monom}.

If $k>1$, the adjoint $qq$-character $\chi^{18}$ does not truncate under this resonance either. And again it can be obtained as combinatorial fusion:
$$
\chi^{18}=\tau_{q_0^{-1}q_1}(\chi^{213})*\tau_{q_2q_3}(\chi^{312}).
$$
In particular, for the top monomial we have
$$
m_+^{18}=\tau_{q_0^{-1}q_1}(V^{213}_{0,0})    \tau_{q_2q_3} (V^{312}_{2k-2,0}).
$$

\section{Slim $qq$-characters and representations of quantum groups.}\label{sec q char}
The combinatorics of $qq$-characters is similar but not identical to that of $q$-characters. In this section we clarify this relation. Our conclusion is that setting $q_0=1$ in a slim $qq$-character one obtains a $q$-character of an appropriate quantum group (under some  technical assumptions).

\subsection{Quantum group of a general Cartan matrix of fermionic type}

We specialize at $q=1$.

Let $C$ be a general Cartan matrix of fermionic type as in Section \ref{sec terminology}. Then after specialization $q=1$, we have $c_{ii}=0$. We assume that  $C$ remains non-degenerate after the specialization.

Let 
$$g_{ij}(z,w)=z-\sigma_{ij}w.
$$
Starting from Cartan matrix $C$, we define an algebra $U_C$. Let $\tilde R$ be the field of rational functions in variables $q_1,q_2,...$  with complex coefficients. 

Let $U_C$ be the algebra over $\tilde R$ generated by coefficients of the series $E_i(z)=\sum_{j\in\Z}E_{i,j}z^{-j}$, $F_i(z)=\sum_{j\in\Z}F_{i,j}z^{-j}$, $K_i^{\pm}(z)=\sum_{\pm j\in\Z_{\geq 0}}K_{i,j}z^{-j}$  $(i\in I)$, subject to the relations
\begin{align*}
    K_i^\pm (z) K_j^\pm (w) = K_j^\pm (w)K_i^\pm (z), \qquad  K_i^\pm (z) K_j^\mp (w) = K_j^\mp (w)K_i^\pm (z), \qquad K_{i,0}^+K_{i,0}^-=1,
    \end{align*}
    \begin{align*}
     g_{ij}(z,w) K_i^\pm(z)E_j(w)+g_{ij}(w,z) E_j(w)K_i^\pm(z)=0, \\
     g_{ij}(w,z) K_i^\pm(z)F_j(w)+g_{ij}(z,w) F_j(w)K_i^\pm(z)=0,
 \end{align*}   
 \begin{align*}
     g_{ij}(z,w) E_i(z)E_j(w)-g_{ij}(w,z) E_j(w)E_i(z)=0, \\
     g_{ij}(w,z) F_i(z)F_j(w)-g_{ij}(z,w) F_j(w)F_i(z)=0,
 \end{align*}  
 \begin{align*}
    E_i(z)E_i(w)+ E_i(w)E_i(z)=0, \\
     F_i(z)F_i(w)+F_i(w)F_i(z)=0,
 \end{align*}  
\begin{align*}
    E_i(z)F_j(w)+F_j(w)E_i(z)=\delta_{ij}\delta(z/w)(K^+_i(z)-K_i^-(w)).
\end{align*}
Here $\delta(z)=\sum_{i\in\Z} z^i$ is the delta function.

Several remarks are in order here.
\setlist[enumerate]{label={\alph*.}}
\begin{enumerate}
\item The algebra $U_C$ is written in Drinfeld generators.

\item  The level of the algebra $U_C$ is set to be zero.

\item Generators $E_{i,j}$ with the same index $i$ skew-commute and so do $F_{i,j}$. In other words $U_C$ is a superalgebra and $E_i(z)$, $F_i(z)$ are fermionic currents. 

\item The Serre relations are omitted. In most examples, the Serre relations are known, but it seems to be rather difficult to write the Serre relations in full generality. We do not discuss them in this text. 

\item  In all examples studied in this text, $U_C$ is isomorphic (or expected to be isomorphic) to the standard quantum affine algebra (without the Serre relations). 
\end{enumerate}

The algebra $U_C$ carries a topological coproduct given by 
\begin{equation}
\begin{aligned}\label{coproduct}
&\Delta K_i^\pm(z)=K_i^\pm(z)\otimes K_i^{\pm}(z), \\
&\Delta E_i(z)=E_i(z)\otimes K_i^{-}(z)+ 1\otimes E_i(z), \\
&\Delta F_i(z)=K_i^+(z)\otimes F_i(z)+ F_i(z)\otimes 1.
\end{aligned}
\end{equation}
We note that the coproduct is a homomorphism of superalgebras and in the tensor product $A\otimes B$ of superalgebras we follow the usual sign rule: we have $(a_1\otimes b_1)(a_2\otimes b_2)=(-1)^{|b_1||a_2|} a_1a_2\otimes b_1b_2$, $a_i\in A$, $b_i\in B$.

For $\sigma \in R$, we have a shift of spectral parameter isomorphism
$\tau_\sigma: U_C \to U_C$  given by
$$
\tau_\sigma(K_i^\pm (z))=K_i^\pm (\sigma z), \qquad 
\tau_\sigma(F_i(z))=F_i (\sigma z), \qquad 
\tau_\sigma(E_i(z))=E_i (\sigma z).
$$

For $J\subset I$, let $C_J$ denote the matrix obtained from $C$ by deleting all rows and columns corresponding to all $i\not \in J$.
Then we have a map $\rho_J:\ U_{C_J}\to U_C$,
$$
K_j^{\pm}(z) \mapsto K_j^{\pm}(z), \qquad E_j(z)\mapsto E_j(z), \qquad 
 F_j(z)\mapsto F_j(z)\qquad (j\in J).$$
 Thus any $U_C$-module $M$ is also a $U_{C_J}$-module which we denote by
 $\rho_JM$.
 
\medskip 

In the cases when $C$ is of type $\gl_{n+1,n}$  (resp. $\gl_{n,n}$) the algebra $U_C$ coincides with the quantum affine algebra $U_{q_1}\hat{\gl}_{n+1,n}$ (resp. $U_{q_1}\hat{\gl}_{n,n}$) (without Serre relations), see \cite{Y}. In the affine case, when $C$ is of type $\hat{\gl}_{n,n}$, $U_C$ is expected to be related to the quantum toroidal superalgebra of type $\gl_{n,n}$. The quantum toroidal superalgebras of type $\gl_{m,n}$, with $m\neq n$ and with all possible choices of the Cartan matrix, were defined and studied in\cite{BM}\footnote{The quantum toroidal algebras of type $\gl_{m,n}$ depend on two independent parameters, our setting corresponds to the case $d=1$ in \cite{BM}.}. The case of $m=n$ is not discussed there, though the relations of $U_C$ are expected to hold. 

To the best of our knowledge, the Drinfeld realizations of quantum affine superalgebras $\hat{\mathfrak{osp}}_{2m+2,2m}$  and  $\hat{\mathfrak{osp}}_{2m,2m}$ have not been established. 

When $C$ is of type D$(2,1;\alpha)$, the algebra $U_C$ coincides with the quantum affine superalgebra of type D$(2,1;\alpha)$ (without Serre relations), see \cite{HSTY}\footnote{Our parameters $q_i$ are related to $q$ and $x(=\alpha)$ in \cite{HSTY} by
$q_1=q^{-x}$, $q_2=q^{x+1}$, $q_3=q^{-1}$.}. When $C$ is of type $\hat {\textrm D}(2,1;\alpha)$, the algebra $U_C$ is expected to coincide with the quantum toroidal superalgebra of type D$(2,1;\alpha)$ (without Serre relations), whose definition is yet to be established, see \cite{FJM1}.

We do not know if there are interesting examples beyond the standard supersymmetric quantum affine and quantum toroidal superalgebras.

\subsection{Admissible  untangled $qq$-characters.}
We introduce terminology related to the specialized $qq$-characters.

We call a degree zero $qq$-character admissible if after the specialization $q=1$ it remains generic with all non-zero coefficients one and if there are no cancellations of variables. 

If a $qq$-character is admissible, then clearly it is slim. All explicit examples of slim $qq$-character we discuss in this paper are admissible.

\medskip

We now choose an admissible $qq$-character $\chi$ and specialize to $q=1$ in the rest of the section.  We keep the concepts of dominant and anti-dominant monomials and variables. Note that any variable $Y_{i,\sigma}$ occurring in a monomial $m\in\chi$ is either dominant or anti-dominant. Also a graph of $\chi$ descends to a graph of the specialized character.

The restriction of any admissible $qq$-character $\chi$ to any color $i\in I$,  $\rho_{\{i\}}\chi$, is a sum of the form
$\sum_s 2^{l_s} \prod_{j=1}^{l_s}Y_{i,\sigma_{j,s}}Y_{i,\tau_{j,s}}^{-1}$. Here each summand  corresponds to a generic product of $l_s$ blocks of length 2. Then the graph of such a summand is an $l_s$-dimensional cube.

\medskip

Let $i,j\in I$, and let $\nu,\tau\in R$ be monomials in $\chi$ such that $\nu\neq \tau\sigma_{ij}^{\pm1}$. A square in $\chi$ associated to  colors $i$ and $j$ and monomials $\nu, \tau$ is a set of four distinct monomials $m_1,m_2,m_3,m_4$ which are connected on the graph of $\chi$ as shown in Figure \ref{square pic}. That is we have arrows $m_1{\xrightarrow {i,\nu}} m_2$, $m_1{\xrightarrow {j,\tau}} m_3$, $m_2{\xrightarrow {j,\tau}} m_4$, and $m_3{\xrightarrow {i,\nu}} m_4$.

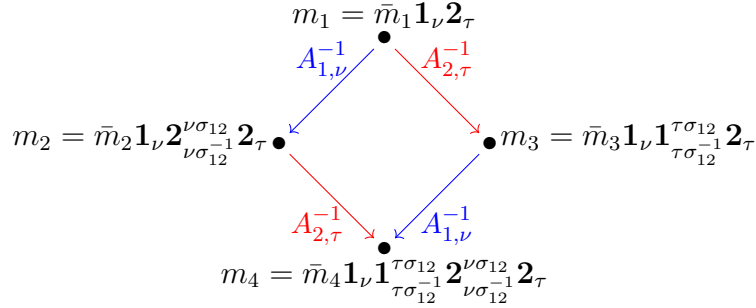
\begin{figure}[H]
\begin{center}
\begin{tikzpicture}[scale=0.7]
\coordinate  (A) at  (0,2);
\node [above] at (A) {$m_1=\bar m_1\bs 1_\nu \bs 2_\tau$};
\node at  (A) {$\bullet$};

\coordinate  (B1) at  (-2,0);
\coordinate  (B2) at  (2,0);
\node [left] at (B1) {$m_2=\bar m_2\bs 1_\nu \bs 2_{\nu\sigma_{12}^{-1}}^{\nu\sigma_{12}}\bs 2_\tau$};
\node at  (B1) {$\bullet$};
\node [right] at (B2) {$m_3=\bar m_3\bs 1_\nu \bs 1_{\tau\sigma_{12}^{-1}}^{\tau\sigma_{12}}\bs 2_\tau$};
\node at  (B2) {$\bullet$};

\coordinate  (C) at  (0,-2);
\node [below] at (C) {$m_4=\bar m_4\bs 1_\nu \bs 1_{\tau\sigma_{12}^{-1}}^{\tau\sigma_{12}}\bs 2_{\nu\sigma_{12}^{-1}}^{\nu\sigma_{12}}\bs 2_\tau$};
\node at  (C) {$\bullet$};

\draw[blue, ->] (-0.2,1.8)--node[above]{{\small $A_{1,\nu}^{-1}\ \ $}} (-1.8,0.2);
\draw[red, ->] (0.2,1.8)--node[above]{{\small $\ \ A_{2,\tau}^{-1}$}} (1.8,0.2);

\draw[red, ->] (-1.8,-0.2)--node[below]{{\small $A_{2,\tau}^{-1}\ \ \ $}} (-0.2,-1.8);
\draw[blue, ->] (1.8,-0.2)--node[below]{{\small $\ \ A_{1,\nu}^{-1}$}} (0.2, -1.8);
\end{tikzpicture}
\end{center}
\caption{A square associated to colors $1$, $2$, and monomials $\nu$, $\tau$. }\label{square pic}
\end{figure}

Define the constant of the square by the formula
$$
c_{\nu\tau}^{ij}=\frac{\nu-\sigma_{ij}\tau}{\tau-\sigma_{ij}\nu}. 
$$
We note that if $\sigma_{ij}=1$ then $c_{\nu\tau}^{ij}=-1$. We also note that $c_{\nu\tau}^{ij} c_{\tau\nu}^{ji}=1$. 

Let $S$ be the set of edges of the graph of $\chi$. We call a map  $a:\ S\to \tilde R\backslash \{0\}$ a supplement of $\chi$ if 
for any square associated to colors $i,j$ and monomials $\nu,\tau$ we have
$a(m_1,m_3)a(m_3,m_4)=c^{ij}_{\nu\tau} a(m_1,m_2)a(m_2,m_4)$.

We note that if a supplement exists, and if $T$ is a spanning tree of the graph of $\chi$, then there exists a supplement which equals one on any edge in $T$.

\medskip
Let $i,j\in I$ be two distinct colors, and $\tau,\nu$ two monomials such that $\nu\neq \tau\sigma_{ij}^{\pm1}$. Consider the corresponding square made of four monomials $m_1,m_2,m_3,m_4$ connected by four arrows as in Figure \ref{square pic}.

If a graph  of a $qq$-character $\chi$ contains exactly three of the  monomials $m_1,m_2,m_3,m_4$ connected by two arrows, then we say that the graph of $\chi$ has an incomplete square.  

In other words, an incomplete square means that a variable of color $i$ which was dominant (antidominant) became antidominant (dominant) after going along an edge of color  $j$, $j\neq i$.

Note that by the definition of the $qq$-character all squares of one color (that is where $j=i$) are automatically complete.

\medskip

We call a $qq$-character $\chi$ untangled if $\chi$ has a graph which has a supplement and no incomplete squares.

\begin{conj}
Any admissible $qq$-character with a dominant monomial is untangled.
\end{conj}

\subsection{A family of tame $U_C$ modules}
We construct a $U_C$-module from an admissible untangled $qq$-character.

For all variables $q_i\in R$ we add to $\tilde R$ their square roots  $q_i^{1/2}$. 

Fix a non-zero evaluation parameter $u\in\tilde R$. For $i\in I$ define the multiplicative homomorphisms $f_i$ mapping monomials in $\mc Y$ to rational functions in $\tilde R(z)$,  by the rule
$$
f_i(Y_{i,\sigma})=\frac{1}{\sigma^{-1/2}z-\sigma^{1/2}u}, \qquad 
f_i(Y_{j,\sigma})=1 \quad (j\neq i).
$$
If $m$ is a degree zero monomial, then $f_i(m)$ is a rational function of $u/z$ which we also write as $f_i(m;u/z)$. It is well defined at $z=0$ and at $z=\infty$, and $f_i(m;0)f_i(m;\infty)=1$.
  
Given an admissible $qq$-character $\chi=\sum_sm_s$, consider the vector space $M_{\chi,u}$ with basis $\{v_{m_s}\}$ labelled by monomials $m_s$.  

Assume we have a graph of $\chi$ with a supplement $a:S\to \tilde R$ of $\chi$ assigning to each edge of the graph of $\chi$ a non-zero element of $\tilde R$. Define a map $b:S\to \tilde R $ of $\chi$ as follows. For an edge $m{\xrightarrow {i,\tau}} n$, $m,n\in\chi,$ we set
$$
b(m,n)=a(m,n)^{-1}\, \mathop {\res}_{z=\tau u} f_i(m;u/z).
$$

Define the action of generators of $U_C$ on $M_{\chi,u}$  by the formulas
\begin{equation}
\begin{aligned}\label{action}
&K_i^{\pm}(z) v_{m_j} = f_i^\pm(m_j;u/z) v_{m_j}, \\
&F_i(z) v_{m_j} =\sum_{s, \ m_j {\xrightarrow {i,\tau_s}} m_s} a(m_j,m_s)\delta(\tau_s u/z) v_{m_s}, \\
&E_i(z) v_{m_j} =\sum_{s, \ m_s {\xrightarrow {i,\tau_s}} m_j} b(m_s,m_j)\delta(\tau_s u/ z) v_{m_s}.
\end{aligned}
\end{equation}
Here $f_i^\pm(m_j;u/z)$ stands for the expansion of rational function $f_i(m_j;u/z)$ at $z^{\mp 1}=0$ and the sum is taken over all edges of color $i$ of the graph of $\chi$ starting (for $F_i(z)$) or ending (for $E_i(z)$) at $m_j$.

\begin{thm}\label{module thm} Let $\chi$ be an admissible untangled $qq$-character. Then
formulas \eqref{action} define a $U_C$-module structure on $M_{\chi,u}$. 
\end{thm}
\begin{proof}
Note that
$$
f_i(A_{j,\tau};u/z)=\sigma_{ij}\frac{1-\sigma_{ij}^{-1}\tau u/z}{1-\sigma_{ij}\tau u/z}=-\frac{g_{ij}(\tau u,z)}{g_{ij}(z,\tau u)}.
$$
Using this it is easy to check that
the quadratic relations involving $K_i^\pm(z)$ with all generators are satisfied.

We check that all other relations in $U_C$ hold as well. All relations are quadratic. 
Apply a relation to a vector $v_m$. Then we get a lot of terms involving $v_{m'}$, where $m'$ is either $m$ (in $E_i(z)F_i(w)v_m$) or obtained from $m$ by walking along two edges of the graph of $\chi$. We show that the total result is zero due to cancellations in squares. 

Consider the square on Figure \ref{square pic}. Then the part of the action related to this square is given in Figure \ref{square action pic}.
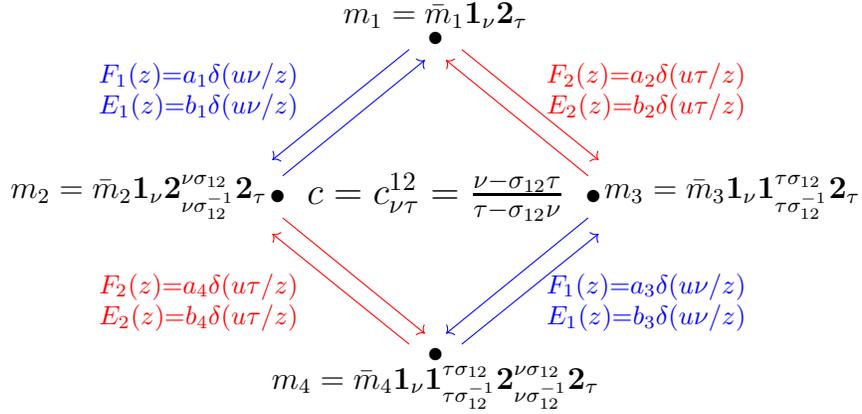
\begin{figure}[H]
\begin{center}
\begin{tikzpicture}[scale=0.7]
\coordinate  (A) at  (0,3);
\node [above] at (A) {$m_1=\bar m_1\bs 1_\nu \bs 2_\tau$};
\node at  (A) {$\bullet$};

\coordinate  (B1) at  (-3,0);
\coordinate  (B2) at  (3,0);
\node [left] at (B1) {$m_2=\bar m_2\bs 1_\nu \bs 2_{\nu\sigma_{12}^{-1}}^{\nu\sigma_{12}}\bs 2_\tau$};
\node at  (B1) {$\bullet$};
\node [right] at (B2) {$m_3=\bar m_3\bs 1_\nu \bs 1_{\tau\sigma_{12}^{-1}}^{\tau\sigma_{12}}\bs 2_\tau$};
\node at  (B2) {$\bullet$};

\coordinate  (C) at  (0,-3);
\node [below] at (C) {$m_4=\bar m_4\bs 1_\nu \bs 1_{\tau\sigma_{12}^{-1}}^{\tau\sigma_{12}}\bs 2_{\nu\sigma_{12}^{-1}}^{\nu\sigma_{12}}\bs 2_\tau$};
\node at  (C) {$\bullet$};

\draw[blue, <-] (-0.2, 2.6)-- (-2.9,0.4);
\draw[blue, ->] (-0.5,2.8)-- (-3.1,0.7);
\draw[red, <-] (0.2,2.6)-- (2.9,0.4);
\draw[red, ->] (0.5,2.8)-- (3.1,0.7);
\node[red] at (4,2) {\large $F_2(z)=a_2\delta(u\tau/z)\atop E_2(z)=b_2\delta(u\tau/z)$};
\node[blue] at (-4.5,2) {\large $F_1(z)=a_1\delta(u\nu/z)\atop E_1(z)=b_1\delta(u\nu/z)$};
\draw[red, <-] (-3.1,-0.7)--(-0.5,-2.8);
\draw[red, ->] (-2.9,-0.4)--(-0.2, -2.6);
\draw[blue, <-] (3.1,-0.7)-- (0.5, -2.8);
\draw[blue, ->] (2.9,-0.4)-- (0.2, -2.6);
\node[blue] at (4,-2) {\large $F_1(z)=a_3\delta(u\nu/z)\atop E_1(z)=b_3\delta(u\nu/z)$};
\node[red] at (-4.5,-2) {\large $F_2(z)=a_4\delta(u\tau/z)\atop E_2(z)=b_4\delta(u\tau/z)$};
\node at (-0,0) {\large $c=c_{\nu\tau}^{12}=\frac{\nu-\sigma_{12}\tau}{\tau-\sigma_{12}\nu}$};
\end{tikzpicture}
\end{center}
\caption{The action in the square. }\label{square action pic}
\end{figure}
Here $a_i$ are provided by the values of the supplement $a$ on the corresponding edges and the values of $b_i$ are given by
$$
b_1=\mathop{\res}_{z=\nu u} f_1(m_1)/a_1, \qquad b_3=\mathop{\res}_{z=\nu u} f_1(m_3)/a_3,
$$
$$
b_2=\mathop{\res}_{z=\tau u} f_2(m_1)/a_2, \qquad b_4=\mathop{\res}_{z=\tau u} f_2(m_2)/a_4.
$$
This assures that the commutators $E_1(z)F_1(w)+F_1(w)E_1(z)$ and  $E_2(z)F_2(w)+F_2(w)E_2(z)$ coincide with terms $\delta(u\nu/z)$ in $K_1^+(z)-K_1^-(z)$ and $\delta(u\tau/z)$ in $K_2^+(z)-K_2^-(z)$, respectively, applied to $m_i$. 

Note that
$$
\mathop{\res}_{z=\nu u} f_1(m_3)=-c\mathop{\res}_{z=\nu u} f_1(m_1), \qquad \mathop{\res}_{z=\tau u} f_2(m_1)=-c\mathop{\res}_{z=\tau u} f_2(m_2),
$$
and that $a_2a_3=ca_1a_4$ by the definition of supplement.
Therefore, we have
$$
\frac{a_2a_3}{a_1a_4}=\frac{b_3b_2}{b_1b_4}=c, \qquad \frac{b_2a_1}{a_3b_4}= \frac{b_1a_2}{a_4b_3}=-1.
$$

The first two relations are equivalent to $FF$ and $EE$ relations applied to $m_1$ and $m_4$ respectively. The last two relations are equivalent to  $E_2(z)F_1(w)+F_1(w)E_2(z)=0$ and  $E_1(z)F_2(w)+F_2(w)E_1(z)=0$ applied to $m_3$ and $m_2$ respectively.

Due to our assumption, one of the vertices of a square can be missing only if $\nu=\tau\sigma_{ij}^{\pm1}$ in which case the relations are satisfied for trivial reasons.
For example, let $\nu=\tau\sigma_{12}$ and then $m_2$ is missing since we cannot have $(\bs 2_\tau)^2$.
In such a case, for example, $F_2(w)F_1(z)m_1$ does not produce $m_4$. But $g_{12}(w,z)F_1(z)F_2(w)m_1$ does not produce $m_4$ either since  $(z-\sigma_{12}w)\delta(u\nu/z)\delta(u\tau/w)=0$.

\end{proof}

Clearly, if the graph of $\chi$ used to construct $M_{\chi,u}$ is connected then $M_{\chi,u}$ is a simple $U_C$-module. Indeed, since all monomials are distinct, generators $K_i^{\pm}(z)$ have a joint simple spectrum, therefore any submodule $M$ has a basis of monimial vectors $v_m$. But if one of $v_m$ is in $M$, then all of them are in $M$ and thus
$M=M_{\chi,u}$. In particular, if $\chi$ is a simple $qq$-character, then $M_{\chi,u}$ is a simple $U_C$-module.

If $\chi$ is an admissible untangled $qq$-character then $q=1$ specialization of the $qq$-character $\chi$ is the $q$-character of $M_{\chi,u}$. The algorithm for construction of $qq$-characters described in Section \ref{sec algorithm} specializes to the algorithm for construction of $q$-characters in \cite{FM}. 

\subsection{The case of D$(2,1;\alpha)$}
Very little is known about the representation theory of quantum affine algebra of type D$(2,1;\alpha)$ and its toroidal version. We give a few remarks.

 Apart from the trivial one dimensional module, the irreducible representations of exceptional Lie superalgebra D$(2,1;\alpha)$ are parameterized by triples $(a,b,c)\in\Z_{\geq 0}$.
The even part of D$(2,1;\alpha)$ is $\sln_2\oplus \sln_2 \oplus \sln_2$. We denote by $[a,b,c]$  the irreducible $\sln_2\oplus \sln_2 \oplus \sln_2$-module of dimension $(a+1)(b+1)(c+1)$ which is a tensor product of three irreducible $\sln_2$ modules of highest weight $a$, $b$, and $c$.

When restricted to the even subalgebra, a generic finite-dimensional representation $L_{(a,b,c)}$  corresponding to triple $(a,b,c)$ has the following structure, see \cite{J}.

Suppose $a,b,c$ are not all equal. This is a typical case when the Kac module is irreducible. Then we have
\begin{align*}
L_{(a,b,c)}= 2[a,b,c] +\hspace{-5pt}\sum_{\epsilon\in\{-1,1\}}([a+2\epsilon,b,c]+[a,b+2\epsilon,c]+[a,b,c+2\epsilon])+\hspace{-10pt}\sum_{\epsilon_1,\epsilon_2\epsilon_3\in\{-1,1\}} \hspace{-5pt} [a+\epsilon_1,b+\epsilon_2,c+\epsilon_2].
\end{align*}
Here all terms with a negative component must be omitted. In addition, the multiplicity 2 of $[a,b,c]$ is replaced with 1 if exactly one of the numbers $a,b,c$ is 0 and with 0 if two of them are zero.

Of course, we have the symmetry permuting $a,b$ and $c$.

\medskip

Let now $a=b=c$. This is an atypical case. We have 
\begin{align*}
L_{(a,a,a)}= & ([a,a,a] +[a+2,a,a]+[a,a+2,a]+[a,a,a+2])\\ +&([a+1,a+1, a+1]+[a+1,a,a]+[a,a+1,a]+[a,a,a+1]),
\end{align*}
where the term $[a,a,a]$ has to be dropped if $a=0$.

\medskip

We work with the Dynkin diagram where all three roots are fermionic. These roots change the weight with respect to  algebra
$\sln_2\oplus \sln_2 \oplus \sln_2$ by $(1,-1,-1),(-1,1,-1)$ and $(-1,-1,1)$. In both typical and atypical cases, the highest weight vector of $L_{(a,b,c)}$ is the highest weight vector of $[a+1,b+1,c+1]$. 

Algebra D$(2,1;\alpha)$ also has three distinguished choices of Borel subalgebra with one fermionic simple root and two bosonic ones. For the three distinguished Borel subalgebras, the highest weight vector of 
$L_{(a,b,c)}$ is the highest weight vector of $[a+2,b,c]$ or  $[a,b+2,c]$ or $[a,b,c+2]$.

 The parity of highest weight vector can be chosen. Let the parity of highest weight vector of the $L_{(a,b,c)}$ be chosen even in the distinguished case (and therefore odd in the fermionic case). We denote the module where the parity is reversed by $L_{(a,b,c)}^p$.

 Then we have the following graded dimensions. In the typical case,
\begin{align*}
\dim L_{(a,b,c)}=(8(a+1)(b+1)(c+1),8(a+1)(b+1)(c+1)).
\end{align*}

In the atypical case, 
$$
\dim L_{(a,a,a)}=(4(a+1)^3+6(a+1)^2-\delta_{a,0}, 4(a+1)^3+6(a+1)^2-2).
$$
In particular, for $a>1$, we have $\dim L_{(a,a,a)}+\dim L_{(a-1,a-1,a-1)}^p=(8(a+1)^3,8(a+1)^3)$ which is the graded character of the Kac module.

\medskip

The adjoint module $L_{0,0,0}$ has dimension 17 and in the principal gradation we have $17=1+3+3+3+3+3+1$. 

The next smallest module is $L_{1,0,0}$ which has dimension 32, which in the principal gradation is $32=1+3+4+5+6+5+4+3+1$.

We have the following decomposition:
$$
L_{(0,0,0)}\otimes L_{(0,0,0)}= L_{(2,0,0)}+L_{(0,2,0)}+L_{(0,0,2)}+L_{(1,1,1)}^p+L_{(0,0,0)},
$$
so, $17^2=48*3+128+17$.

\medskip

We construct the $18$-dimensional and $66$-dimensional representations of quantum affine algebra of of D$(2,1;\alpha)$ type.

For the $18$-dimensional module, it is enough to provide the supplement.
We set the value of the supplement on all edges in Figure \ref{18 pic}, to be 1 except 
\begin{itemize} 
\item edges on the second level (from vertices $v_2, v_3, v_4$ to vertices $v_5,v_6,v_7$) where we set the value of supplement to be $(q_i-q_i^{-1})^{-1}$ on edges of color $i$; 

\item edges on the second to the last level (from vertices $v_{12}, v_{13}, v_{14}$ to vertices $v_{15},v_{16},v_{17}$) where we set the value of supplement to be $(q_i-q_i^{-1})$ on edges of color $i$; 

\item on edges $(v_5,v_8)$, $(v_6,v_{11})$, $(v_7,v_{10})$ where we set the value of supplement to be $-1$.
\end{itemize}
It is easy to check that such an assignment is indeed a supplement.

To construct a $66$ dimensional module, we consider the tensor product $M_{\chi^{18},uq_1^2}\otimes M_{\chi^{18},u}$.
Comultiplication \eqref{coproduct} is not well-defined on this tensor product. 
However, it is well-defined on the 66 terms inside this tensor product listed in Figure \ref{66 pic}. Moreover, the action along all edges which connect a vertex from the 66 terms to some vertex which is not in the 66 is well-defined and either $F(z)$  or  $E(z)$ action along such an edge vanishes. Thus,
the 66 module is found as a subquotient.

\medskip

After restriction to finite type D$(2,1;\alpha)$ the above modules decompose as $18=17+1$ and $66=48+17+1$ and therefore can be thought as affinizations of $L_{(0,0,0)}$ and $L_{(2,0,0)}$. One can argue that these are minimal affinizations in the sense of \cite{C}. 

On the other hand, it is easy to see that there is no $q$-character which has the structure of the 32-dimensional $L_{(1,0,0)}$ module which suggests that $L_{(1,0,0)}$ has no finite-dimensional affinization in contrast to the even case where every finite-dimensional irreducible module of finite type had an affinization.

\section{Free field realization}\label{sec:freefield}
The $qq$-characters are a combinatorial abstraction of basic currents which appear in the free field construction of
deformed $W$-algebras. 
We are now in a position to elucidate this connection. 

From now on, we regard the parameters $q,q_1,q_2,\ldots$ as non-zero complex numbers.

\subsection{Vertex operators}\label{subsec:VO}
Fix a Cartan matrix $C=\bigl(\sigma_{ij}-\sigma_{ij}^{-1}\bigr)_{i,j\in I}$.
We introduce parameters $k_{ij}\in\C$ by setting $\sigma_{ij}=q^{k_{ij}}$, 
where $k_{ji}=k_{ij}$, $k_{ii}=1$. 
We assume that the matrix $C^{[n]}=\bigl(q^{n k_{ij}}-q^{-n k_{ij}}\bigr)_{i,j\in I}$ is non-degenerate
for all $n\neq0$. 

Consider a Heisenberg algebra with generators $\{\sss_{i,n}\mid i\in I,\ n\in\Z\backslash\{0\}\}$ 
subject to the commutation relations 
\begin{align}
&[\sss_{i,n},\sss_{j,m}]=
-\frac{1}{n}\frac{q^{n k_{ij}}-q^{-n k_{ij}}}{q^n-q^{-n}}
\, \delta_{n+m,0}\quad (n,m\neq 0) \,.\label{sin}
\end{align}
We use also a dual set of generators $\{\ssy_{i,n}\mid i\in I\ n\in \Z\backslash\{0\}\}$  
\begin{align*}
&\ssy_{i,n}=-\sum_{l\in I}(q^n-q^{-n})({C^{[n]}}^{-1})_{i,l}\sss_{l,n}\,\quad (n\neq 0), 
\end{align*}
so that 
\begin{align}
&[\sss_{i,n},\ssy_{j,m}]=\frac{1}{n}\delta_{i,j}\delta_{n+m,0}\quad (n,m\neq 0)\,.\label{yin}
\end{align}

The matrix $K=\bigl(k_{ij}\bigr)_{i,j\in I}$ may be degenerate in general. 
We proceed as in the case of Kac-Moody algebras \cite{K}. 
Let $\ell=\mathrm{rank}\,K$. 
Choose a decomposition $I=I_0\sqcup I_1$ 
and corresponding submatrices $K_{a,b}=\bigl(k_{ij}\bigr)_{i\in I_a,j\in I_b}$, $a,b\in\{0,1\}$, 
in such a way that $K_{1,1}$ is an $\ell\times \ell$ matrix of rank $\ell$. 
Let $I_0^*=\{i^*\mid i\in I_0\}$ be a copy of $I_0$. 
We extend $K$ to a non-degenerate matrix indexed by 
$\hat{I}=I_0\sqcup I_1\sqcup I_0^*$, 
\begin{align*}
\widehat{K}
=\begin{pmatrix}
K_{0,0} & K_{0,1} & -1 \\
K_{1,0} & K_{1,1} & 0 \\
-1 & 0 & 0\\
\end{pmatrix}\,,
\end{align*}
where $1$ stands for the $(n-\ell)\times (n-\ell)$ unit matrix. 
Consider a Heisenberg algebra with generators $\{\sss_{i,0}, Q_{\sss_i}\}_{i\in \hat{I}}$ such that
\begin{align*}
&[\sss_{i,0},\sss_{j,0}]=[Q_{\sss_i},Q_{\sss_j}]=0\,,
\quad 
[\sss_{i,0},Q_{\sss_j}]=(\widehat{K})_{i,j}\quad (i,j\in \hat{I}).
\end{align*}
We define $\ssy_{i,0}$, $Q_{\ssy_i}$ for $i\in I=I_0\sqcup I_1$ by 
\begin{align*}
&\ssy_{i,0}
=\begin{cases}
\sss_{i^*,0} & (i\in I_0),\\
-\sum_{k\in \hat{I}}(\widehat{K}^{-1})_{i,k}s_{k,0} & (i\in I_1),\\ 
\end{cases}
\\
&Q_{\ssy_i}=
\begin{cases}
Q_{\sss_{i^*}} & (i\in I_0),\\
-\sum_{k\in \hat{I}}(\widehat{K}^{-1})_{i,k}Q_{s_k}&(i\in I_1).\\
\end{cases}
\end{align*}
Then the following hold for all $i,j \in I$:
\begin{align*}
[\sss_{i,0},Q_{\ssy_j}]=-\delta_{i,j}\,,\quad [\ssy_{j,0},Q_{\sss_i}]=-\delta_{i,j}\,.
\end{align*}

We shall refer to $\sss_{i,n}, \ssy_{i,n}$ ($n\neq 0$) as oscillators, and 
$\sss_{i,0},\ssy_{i,0},Q_{\sss_i},Q_{\ssy_i}$ as zero modes.

Notation being as above, we introduce the following vertex operators:
\begin{align}
&S_i(z)=e^{Q_{\sss_i}}z^{\sss_{i,0}}:e^{\sum_{n\neq0}\sss_{i,n}z^{-n}}:\,,\label{Si}
\\
&Y_i(z)=e^{Q_{\ssy_i}}z^{\ssy_{i,0}}:e^{\sum_{n\neq0}\ssy_{i,n}z^{-n}}:\,,\label{Yi}
\\
&A_i(z)=:\frac{S_i(q^{-1}z)}{S_i(q z)}:=
q^{-2\sss_{i,0}}:e^{\sum_{n\neq0}(q^n-q^{-n})\sss_{i,n}z^{-n}}:\,.
\label{Ai}
\end{align}
Here the standard normal ordering rule is in force: 
creation operators $\sss_{i,-n}$, $\ssy_{i,-n}$  
($n>0$), $e^{Q_{\sss_i}}$, $e^{Q_{\ssy_i}}$
are placed to the left and 
annihilation operators $\sss_{i,n}$, $\ssy_{i,n}$    
($n>0$), $\sss_{i,0}$,  $\ssy_{i,0}$ 
are to the right.
We call $A_i(z)$ affine root currents, and $S_i(z)$ screening currents. 

Quite generally, a product of two vertex operators $V(z), W(z)$ takes the form 
\begin{align*}
V(z)W(w)=z^{\alpha}\varphi_{V,W}(w/z):V(z)W(w):, 
\end{align*}
where $\alpha\in\C$, and $\varphi_{V,W}(w/z)$ is a formal power series in $w/z$. 
We call $z^{\alpha}\varphi_{V,W}(w/z)$  the contraction of $V(z), W(z)$ and use the symbol $\cont{V(z)}{W(w)}$
to denote it. For instance 
\begin{align}
&\cont{S_i(z)}{S_i(w)}=z-w\,,\label{contSS}
\\
&\cont{S_i(z)}{Y_i(w)}=\frac{1}{z-w}\,,
\quad \cont{Y_i(w)}{S_i(z)}=\frac{1}{w-z}\,,\label{contSY}
\\
&\cont{S_i(z)}{Y_j(w)}=\cont{Y_j(w)}{S_i(z)}=1\quad (i\neq j).
\label{contSY2}
\end{align}
In all cases considered in this paper, the series  $\varphi_{V,W}(w/z)$ can be written as 
\begin{align*}
\log\varphi_{V,W}(w/z)=\sum_{n>0}\frac{1}{n}\Bigl(\frac{w}{z}\Bigr)^n\times
\bigl(f_{V,W}\Bigl|_{q_{i}\to q_{i}^n,\forall i}\bigr)\,,
\end{align*}
where  $f_{V,W}$ is a rational function of the parameters $q_{i}$ entering the Cartan matrix. 
We denote the function $f_{V,W}$  by $\mathcal{C}(V,W)$, and use it  as a mnemonic for $\varphi_{V,W}(w/z)$. 
In this notation we have
\begin{align}
&\mathcal{C}(S_i,S_j)=-\frac{1}{q-q^{-1}}(\sigma_{ij}-\sigma_{ij}^{-1})\,,\notag\\
&\mathcal{C}(S_i,Y_j)=\mathcal{C}(Y_j,S_i)=\,\delta_{ij}\,,\notag
\\
&\mathcal{C}(A_i,A_j)=(q-q^{-1})(\sigma_{ij}-\sigma_{ij}^{-1})\,,\label{AA}
\\
&\mathcal{C}(A_i,Y_j)=-\mathcal{C}(Y_j,A_i)
=(q-q^{-1})\delta_{i,j}\,,\label{AY}
\\
&\mathcal{C}(Y_i,Y_j)=-(q-q^{-1})(C^{-1})_{i,j}\,.\notag
\end{align}
In the last line we use the entries of the inverse $C^{-1}$ of the Cartan matrix.

\subsection{Bosonization of $qq$-characters}\label{subsec:bosonization}

With every monomial $m\in\mathcal Y$, $m=\prod_{i\in I}\prod_{a\in\C}Y^{n_{i,a}}_{i,a}$, a finite product
in the variables $\{Y_{i,a}\}$, $n_{i,a}\in \Z$, we associate a vertex operator 
\begin{align*}
V_m(z)=:\prod_{i\in I}\prod_{a\in\C}Y_{i}(a z)^{n_{i,a}}:\,.
\end{align*}
Let $m$ be a generic monomial of degree zero. 
Due to \eqref{contSY}, \eqref{contSY2},  
the contractions with screening currents $\cont{S_i(w)}{V_{m}(z)}$, $\cont{V_{m}(z)}{S_i(w)}$
depend only on the restriction $\rho_{\{i\}}(m)$. 
Moreover they converge to the same rational function whose poles are all simple. 
It follows that the commutator is a finite sum 
\begin{align*}
[S_i(w),V_{m}(z)]=\sum_{a}c_{i,{m},a}w^{-1}\delta\Bigl(\frac{az}{w}\Bigr) \,:S_i(az)V_{m}(z): \,,
\end{align*}
where $c_{i,{m},a}\in\C$ are some coefficients and as before $\delta(z)=\sum_{n\in \Z}z^n$ stands for the delta function. 
Introduce the screening operator as a formal integral
\begin{align}
S_i=\int S_i(w)dw\,.\label{screening-charge}
\end{align}
Given a $qq$-character $\chi=\sum_s m_s$, 
we shall say that it formally commutes with the screening operator $S_i$
if $\sum_s \sum_a c_{i,m_s,a }:S_i(az)V_{m_s}(z):=0$. We write this as $[ S_i,\chi]=0$. 
\medskip

\noindent{\it Example.}\quad 
Consider an elementary block of length $k+1$ in a single color, 
\begin{align*}
\chi=\sum_{j=0}^k m_k\,,\quad m_j=\mathbf{1}^{a_1,\dots,a_{k}}_{a,\ldots,\widehat{q^{2j}a},\ldots,q^{2k}a}
=m_0\prod_{s=1}^j A^{-1}_{1,q^{2s-1}a}
\,,
\end{align*}
where the hat signifies the missing factor.
By \eqref{Ai}, $V_{m_j}(z)$ can be written in terms of screening currents as 
$V_{m_j}(z)=:V_{m_0}(z)S_1(q^{2j}az)S_1(az)^{-1}:$.
Computing the residues we find 
\begin{align*}
&[S_1(w),V_{m_j}(z)]=\sum_{0\le i\le k\atop i\neq j}(q^{2i}a-q^{2j}a) c_{i}\,
w^{-1}\delta\Bigl(\frac{q^{2i}az}{w}\Bigr)
:V_{m_0}(z)\frac{S_1(q^{2i}az)S_1(q^{2j}az)}{S_1(az)}:\,,
\end{align*}
where
$c_{i}=\prod_{r=1}^k(q^{2i}a-a_r)/\prod_{0\le s\le k, s\neq i}(q^{2i}a-q^{2s}a)$. 
We define 
\begin{align*}
T_\chi(z)=\sum_{j=0}^{k}c_j\,V_{m_j}(z)\,.    
\end{align*}
In the product $S_1(w)T_\chi(z)$,  
the residue at $w=q^{2i}az$ coming from $S_1(w)V_{m_j}(z)$   
cancels with the one at $w=q^{2j}az$ coming from $S_1(w)V_{m_i}(z)$, for all pairs $i\neq j$.       
This means that $T_\chi(z)$ formally commutes with the screening operator 
$S_1$. 
\qed
\medskip

This example generalizes as follows. 

\begin{thm}\label{prop:qq-coeff}
Let $\chi=\sum_s m_s$ be a finite simple $qq$-character considered in Section 2.2. Then there exist coefficients $c_m\in \C^{\times}$
such that the corresponding current 
\begin{align}
T_\chi(z)=\sum_s c_{m_s} V_{m_s}(z)\,\label{Tchi}
\end{align}
formally commutes with all screening operators $S_i=\int S_i(w)dw$, $i\in I$. 
The  $c_{m_s}$'s are unique up to an overall scalar multiple.
\end{thm}
\begin{proof}
Let $\{m_s\}_{s=0,1,\ldots,l}$ be a sequence of monomials in $\chi$ such that 
$m_s{\xrightarrow {i_s,a_s}}m_{s-1}$ with some $i_s\in I$, $a_s\in \C^{\times}$,  $s=1,\ldots,l$.
Consider the linear equations for unknowns $\{d_s\}_{s=0,\ldots,l}$:
\begin{align}
d_{s-1}\Res_{w=a_s qz}\cont{S_{i_s}(w)}{V_{m_{s-1}}(z)}dw
+d_s\Res_{w=a_s q^{-1}z}\cont{S_{i_s}(w)}{V_{m_s}(z)}dw=0\,,
\quad s=1,\ldots,l\,.
\label{neighbor}
\end{align} 
We show below that the ratio $d_l/d_0$ is determined by $m_0$ and $m_l$ alone, 
independently of the choice of the sequence $\{m_s\}$ connecting them. 
Since the graph of $\chi$ is connected, 
Theorem will then follow by setting $c_{m_l}/c_{m_0}=d_l/d_0$.

Let us take a closer look at the equation \eqref{neighbor}. 
Introducing $f_i(w)=\cont{S_i(w)}{V_{m_0}(z)}$ ($i\in I$), we have 
\begin{align*}
&\cont{S_{i_s}(w)}{V_{m_{s-1}}(z)}=f_{i_s}(w)\prod_{t=1}^{s-1}\frac{w-\sigma_{i_s,i_t}a_tz}{w-\sigma^{-1}_{i_s,i_t}a_tz}\,,\\
&\cont{S_{i_s}(w)}{V_{m_s}(z)}=\cont{S_{i_s}(w)}{V_{m_{s-1}}(z)}\frac{w-qa_sz}{w-q^{-1}a_s z}\,,
\end{align*}
and hence 
\begin{align}
&\Res_{w=a_s qz}\cont{S_{i_s}(w)}{V_{m_{s-1}}(z)}dw=
\Bigl\{(qw-qa_sz)f_{i_s}(qw)\prod_{t=1}^{s-1}
\frac{qw-\sigma_{i_s,i_t}a_tz}{qw-\sigma^{-1}_{i_s,i_t}a_tz}\Bigr\}\Bigl|_{w=a_sz}\,,
\label{Res+}\\
&\Res_{w=a_s q^{-1}z}\cont{S_{i_s}(w)}{V_{m_s}(z)}dw=
\Bigl\{(q^{-1}w-qa_sz)f_{i_s}(q^{-1}w)\prod_{t=1}^{s-1}
\frac{q^{-1}w-\sigma_{i_s,i_t}a_tz}{q^{-1}w-\sigma^{-1}_{i_s,i_t}a_tz}\Bigr\}\Bigl|_{w=a_sz}\,.
\label{Res-}
\end{align}
Here we have to be careful when some of the factors vanish at $w=a_sz$. 
For all pairs $s\neq t$, define
\begin{align*}
N^{\pm}_{s,t}=
\begin{cases}
1 & \text{if $\sigma_{i_s,i_t}=q^{\pm1}a_s/a_t$},\\
0 & \text{if $\sigma_{i_s,i_t}\neq q^{\pm1}a_s/a_t$}.\\
\end{cases}
\end{align*}
Let further $l^\pm_s$ be the order of zeroes of $(q^{\pm1}w-qa_sz)f_{i_s}(q^{\pm1}w)$ at $w=a_sz$. 
The relation $m_{s-1}\overset{i_s,a_s}{\longrightarrow} m_s$ ensures that \eqref{Res+} and \eqref{Res-} are both well-defined and non-zero.
Hence we must have 
\begin{align}
l^{\pm}_s+\sum_{t=1}^{s-1}(N^\pm_{s,t}-N^\mp_{t,s})=0\,,
\quad s=1,\ldots,l
\,. 
\label{lr_Nrs}
\end{align}
Under this condition, the ratios $d_s/d_{s-1}$ are well defined.
We rewrite further the right hand side of \eqref{Res+} as 
\begin{align*}
\Bigl\{(qw_s-qa_sz)f_{i_s}(qw_s)\prod_{t=1}^{s-1}
\frac{qw_s-\sigma_{i_s,i_t}w_t}{qw_s-\sigma^{-1}_{i_s,i_t}w_t}\Bigr\}\Bigl|_{w_1=a_1z}\cdots\Bigl|_{w_{s-1}=a_{s-1}z}\Bigl|_{w_s=a_sz}
\,.
\end{align*}
Doing the same for  \eqref{Res-}, and multiplying  $d_s/d_{s-1}$ through $s=1,\ldots,l$, we arrive at 
\begin{align}
&\frac{d_l}{d_0}=(-1)^l
\Bigl\{\prod_{s=1}^lF_s(w_s)
\prod_{s\neq t}\frac{qw_s-\sigma_{i_s,i_t}w_t}{q\sigma_{i_s,i_t}w_s-w_t}\Bigr\}\Bigl|_{w_1=a_1z,\ldots,w_l=a_lz}\,,
\label{standard}\\
&F_s(w)=\frac{(qw-qa_sz)f_{i_s}(qw)}{(q^{-1}w-qa_sz)f_{i_s}(q^{-1}w)}\,,\nn
\end{align}
where the specialization $w_s=a_sz$ is performed in the order $s=1,2,\ldots,l$. 

Now let $\{m_s'\}_{s=0,1,\ldots,{l'}}$ be another sequence of monomials in $\chi$ 
such that $m'_0=m_0$, $m'_{l'}=m_l$ and $m'_s=m'_{s-1}A^{-1}_{i'_s,a'_s}$, $s=1,\ldots,l'$.  
Since the affine roots are algebraically independent, we must have that $l'=l$ and 
$i'_s=i_{\lambda(s)}$, $a'_s=a_{\lambda(s)}$ for some permutation $\lambda\in \Sb_l$.
Define $\{d'_s\}_{s=0,\ldots,l}$, $l^{'\pm}_s$, $N^{'\pm}_s$ similarly as above, using $\{m'_s\}$. 
Then we have $l^{'\pm}_s=l^{\pm}_{\lambda(s)}$, $N^{'\pm}_s=N^{\pm}_{\lambda(s)}$, 
and 
$d'_\ell/d'_0$ is given by the same expression \eqref{standard} except that the specialization is performed in the order
$w_{\lambda(1)}=a_{\lambda(1)}z,\ldots, w_{\lambda(l)}=a_{\lambda(l)}z$.

We consider the ratio of  $d_l/d_0$ to $d'_l/d'_0$.
For the two ways of specialization, the factors $F_s(w_s)$ give the same contribution and hence 
cancel out. The factors $w_s-q^{\pm1}\sigma_{i_s,i_t}w_t$ also cancel except in the cases 
$q^{\pm1}\sigma_{i_s,i_t}=a_s/a_t$ and $t<s$, $\lambda^{-1}(t)>\lambda^{-1}(s)$
or  $t>s$, $\lambda^{-1}(t)<\lambda^{-1}(s)$. 
With the abbreviation $s'=\lambda^{-1}(s)$, $t'=\lambda^{-1}(t)$ we find
\begin{align}
\frac{d_l/d_0}{d_l'/d_0'}=\prod_{t<s\atop t'>s'}
\Bigl(-\frac{a_s}{a_t}\Bigr)^{N^-_{s,t}-N^+_{t,s}-N^+_{s,t}+N^-_{t,s}}=\varepsilon\prod_{s}a_s^{\nu_s}\,,
\,.\label{d-ratio}
\end{align}
with $\varepsilon=\pm1$ and $\nu_s\in\Z$. The power $\nu_s$ is given by
\begin{align*}
\nu_s=
\sum_{t:t<s \atop t'>s'}(N^-_{s,t}-N^+_{t,s})-\sum_{t:t<s\atop t'>s'}(N^+_{s,t}-N^-_{t,s})
+\sum_{t:t>s\atop t'<s'}(N^+_{s,t}-N^-_{t,s})-\sum_{t:t>s\atop t'<s'}(N^-_{s,t}-N^+_{t,s})\,.
\end{align*}
Due to the equality
\begin{align}
\sum_{t: t<s}(N^{\mp}_{s,t}-N^{\pm}_{t,s})
=-l^{\mp}_s   
=\sum_{t: t'<s'}(N^{\mp}_{s,t}-N^{\pm}_{t,s})
\,,\quad 
s=1,\ldots,l\,, 
\label{Nst}
\end{align}
following from \eqref{lr_Nrs} and its analog for $\{m'_s\}$, we obtain $\nu_s=0$.
Summing \eqref{Nst} 
over $s$ we obtain also
\begin{align*}
\sum_{t<s\atop t'>s'}(N^{-}_{s,t}-N^{+}_{t,s})=
\sum_{t>s \atop t'<s'}(N^{-}_{s,t}-N^{+}_{t,s})=
\sum_{t<s\atop t'>s'}(N^{-}_{t,s}-N^{+}_{s,t})\,,
\end{align*}
which shows that 
$\varepsilon=1$.
We thus conclude that $d'_l/d'_0=d_l/d_0$. 
\end{proof}

We shall say that $T_{\chi}(z)$ is the $qq$-current associated with $\chi$. 

\medskip

\noindent{\it Remark.}
While Theorem \ref{prop:qq-coeff} claims only exiestence of coefficients $c_m$, formula \eqref{standard} provides a way to compute them. In particular, all the coefficients naturally appear in a factorized form. \qed

\medskip

\noindent{\it Remark.}\quad 
In conformal field theory, the usual screening currents are Virasoro primary fields of conformal weight one,
and their integrals commute with the Virasoro current.  For a general primary field  $S(w)$
of conformal weight $\Delta$, the integral $\int w^{\Delta}S(w) dw/w$ commutes with the grading operator $L_0$. 
When the conformal limit has a clear meaning, it is more natural to redefine
the screening operator \eqref{screening-charge} in this way. Such a change amounts to shifting the zero mode
$s_{i,0}$ by a constant. It affects only a power of $q$'s in the coefficients $c_m$ of $qq$-currents.
\qed

\subsection{Vector representation of D$(2,1;\alpha)$ and $\widehat{\textrm D}(2,1;\alpha)$}
In \cite{FJMV}, the $qq$-currents of vector $qq$-characters have been given  
for a class of deformed $W$-algebras including $\gl_{n,n}$,  $\gl_{n+1,n}$, 
and $\mathfrak{osp}_{n,n}$. 
In this section, 
we use parameters $k_i$ with $k_0=0$ and $k_1+k_2+k_3=-4$, such that 
$q_i=q^{k_i+1}$, $i=0,1,2,3$, $q_0q_1q_2q_3=1$. As before, $p_i=q_0^2q_i^2$.

According to the general rule, we have zero modes $\{s_{i,0}\}_{i=0}^4$ and $\{Q_{s_i}\}_{i=0}^4$, where 
$s_{4,0}=y_{0,0}$ and $Q_{s_4}=Q_{y_0}$.  Their commutators are given by the extended matrix
\begin{align*}
\Bigl([s_{i,0},Q_{s_j}]\Bigr)_{0\le i,j\le 4}=
\begin{pmatrix}
1 & k_1+1 & k_2+1 & k_3+1 & -1 \\
k_1+1 & 1 &k_3+1 & k_2+1 &  0 \\
k_2+1 & k_3+1 & 1& k_1+1 &  0 \\
k_3+1 & k_2+1 & k_1+1 &1 &  0 \\
-1 & 0 & 0 & 0 & 0\\
\end{pmatrix}\,.
\end{align*}
The remaining zero modes $\{y_{j,0}, Q_{y_j}\}$ ($j=1,2,3$) are given by
\begin{align*}
&y_{1,0}=y_{0,0}-\frac{1}{2(k_3+2)}(s_{1,0}+s_{2,0})-\frac{1}{2(k_2+2)}(s_{1,0}+s_{3,0})\,,\\
&Q_{y_1}=Q_{y_0}-\frac{1}{2(k_3+2)}(Q_{s_1}+Q_{s_2})-\frac{1}{2(k_2+2)}(Q_{s_1}+Q_{s_3})\,,
\end{align*}
and by cyclically permuting $1,2,3$. 

In what follows we shall assume that $|p_1|<1$. We use the standard symbols 
\begin{align*}
(z_1,\ldots,z_m;p)_k=\prod_{s=1}^m\prod_{j=0}^{k-1}(1-z_s p^j)\,,\quad
\Theta_p(z)=(z,p/z,p;p)_\infty\,.
\end{align*}
Following the remark at the end of the previous Section, 
we modify the screening operators as follows:
\begin{align}
S_i=\int w^{-2\delta_{i,0}}S_i(w)\, dw\quad (i=0,1,2,3).\label{D21-screening-charge}
\end{align}

The $qq$-currents associated with the vector $qq$-characters ${\chi}^{312}$, $\widehat{\chi}^{312}$,
see Figures \ref{D vector pic}, \ref{D hat vector pic},
are formal infinite sums of vertex operators $V^{312}_{a,b}(z)=V_{V^{312}_{a,b}}(z)$,
\begin{align}
T^{312}(z)&=
\sum_{a=0}^\infty c^{312}_{a,0} V^{312}_{a,0}(z)+\sum_{a=1}^\infty c^{312}_{a,1} V^{312}_{a,1}(z)\,,\label{T312formal}
\\
\widehat{T}^{312}(z)&=\sum_{a,b\in\Z\atop a\ge b}c^{312}_{a,b}V^{312}_{a,b}(z)\,.
\label{T312formal2}
\end{align}
Explicitly the coefficients $c^{312}_{a,b}$ described by Theorem \ref{prop:qq-coeff} are given by
\begin{align}
&c^{312}_{2k,2l}=\frac{(q_2^{-2}p_1^{k-l},p_1^{k-l+1};p_1)_\infty}{(q_0^{-2}q_2^{-2}p_1^{k-l},q_0^{-2}p_1^{k-l+1};p_1)_\infty}q_0^{4(k-l)}\,,
\label{c312-1}\\
&c^{312}_{2k,2l+1}=-\frac{(q_2^{-2}p_1^{k-l},p_1^{k-l};p_1)_\infty}{(q_0^{-2}q_2^{-2}p_1^{k-l},q_0^{-2}p_1^{k-l};p_1)_\infty}q_0^{4(k-l)}\,,
\label{c312-2}\\
&c^{312}_{2k+1,2l}=-\frac{(q_2^{-2}p_1^{k-l+1},p_1^{k-l+1};p_1)_\infty}{(q_0^{-2}q_2^{-2}p_1^{k-l+1},q_0^{-2}p_1^{k-l+1};p_1)_\infty}q_0^{4(k-l)+2}\,,
\label{c312-3}\\
&
c^{312}_{2k+1,2l+1}=\frac{(q_2^{-2}p_1^{k-l},p_1^{k-l+1};p_1)_\infty}
{(q_0^{-2}q_2^{-2}p_1^{k-l},q_0^{-2}p_1^{k-l+1};p_1)_\infty}q_0^{4(k-l)+2}\,.
\label{c312-4}
\end{align}
Note that $c^{312}_{a,b}=0$ unless $a\ge b$.

Formulas \eqref{c312-1}--\eqref{c312-4}
can be obtained by solving the recurrence relations sketched in the proof of Theorem \ref{prop:qq-coeff}. 
We give below a direct way to derive them. 

Recall that the monomials $V^{312}_{a,b}$ are composed of elementary pieces $R^\pm_1$, $T^\pm_1$, 
see \eqref{R}, \eqref{T}. 
Let 
\begin{align*}
\rho^\pm_1(z)=V_{R^\pm_{1}}(z)\,, \quad 
\tau^\pm_1(z)=V_{T^\pm_{1}}(z)
\end{align*}
be the corresponding vertex operators.  
Then the affine root currents are written as
\begin{align}
&A_0(z)=:\frac{\tau^+_1(p_1^{1/2}z)}{\tau_1^-(p_1^{-1/2}z)}:\,,\quad
A_1(z)=:\frac{\tau^-_1(z)}{\tau_1^+(z)}:\,,\label{A-tau-rho1}\\
&A_2(z)=:\frac{\rho^+_1(p_1^{1/2}z)}{\rho_1^-(p_1^{-1/2}z)}:\,,\quad
A_3(z)=:\frac{\rho^-_1(z)}{\rho_1^+(z)}:\,.\label{A-tau-rho2}
\end{align}
We shall use the contractions 
\begin{align}
&\cont{\tau^{\epsilon_1}_1(z)}{\rho^{\epsilon_2}_1(w)}=z^{-\frac{2}{k_1+2}}q_0^{\epsilon_1}
\frac{(q_0 q_2^{-1}w/z,p_1q_0q_2w/z;p_1)_\infty}
{(q_0^{-1}q_2^{-1}w/z,p_1q_0^{-1}q_2w/z;p_1)_\infty}g_{\epsilon_1\epsilon_2}(w/z)\,,\label{tau-rho}\\
&g_{\pm,\pm}(z)=1\,,\quad g_{\pm,\mp}(z)=\frac{1-(q_0q_2)^{\mp 1}z}{1-(q_0^{-1}q_2)^{\mp 1}z}\,,\nn
\end{align}
and the contractions with screening currents given in the table below. 
\bigskip

\begin{table}[H]  
\begin{center}
\begingroup
\renewcommand{\arraystretch}{1.5}
\begin{tabular}{|c|c|c|c|c|}
\hline
& $S_0(w)$ &$S_{1}(w)$ &  $S_{2}(w)$ &  $S_{3}(w)$ \\
\hline
$\rho^\pm_{1}(z)$ & $q_3^{\pm 1}z-w$ &$q_2^{\pm 1}z-w$ &$(q_1^{\mp 1}z-w)^{-1}$ &$(q^{\mp 1}z-w)^{-1}$ \\
\hline
$\tau^\pm_{1}(z)$ &$(q_1^{\mp 1}z-w)^{-1}$ &$(q^{\mp 1}z-w)^{-1}$ & $q_3^{\pm 1}z-w$ &$q_2^{\pm 1}z-w$ \\
\hline
\end{tabular} 
\endgroup
\end{center}
\caption{Contractions $\cont{X(z)}{Y(w)}$, where $X(z)=\rho^\pm_1(z),\tau^\pm_1(z)$, $Y(w)=S_i(w)$.
\label{tab:art}}
\end{table}
\bigskip

The above formulas allow us to calculate the commutator between current $S_1(w)$ and the product $\tau_1^{\epsilon_1}(z)\rho_1^{\epsilon_2}(z')$:
\begin{align*}
[S_1(w), \tau_1^{\epsilon_1}(z)\rho_1^{\epsilon_2}(z')]=
w^{-1}\delta\Bigl(q_0^{-\epsilon_1}\frac{z}{w}\Bigr)\bigl(q_0^{-\epsilon_1}z-q_2^{\epsilon_2}z'\bigr)
\cont{\tau_1^{\epsilon_1}(z)}{\rho_1^{\epsilon_2}(z')}\times
:S_1(q_0^{-\epsilon_1}z)\tau_1^{\epsilon_1}(z)\rho_1^{\epsilon_2}(z'):\,.
\end{align*}
We have $:S_1(q_0^{-1}z)\tau_1^+(z):=:S_1(q_0z)\tau_1^-(z):$ from \eqref{A-tau-rho1}.  
Noting the relation 
\begin{align*}
(q_0^{-1}z-q_2^{\pm1}z')\cont{\tau^+_1(z)}{\rho_1^{\pm}(z')}=(q_0z-q_2^{\pm1}z')\cont{\tau^-_1(z)}{\rho_1^{\pm}(z')}\,,
\end{align*}
which follows from \eqref{tau-rho}, we find the commutativity with the screening operator 
\eqref{D21-screening-charge}:
\begin{align*}
[S_1, (\tau_1^+(z)-\tau_1^-(z))\rho_1^\pm(z')]=0\,.
\end{align*}
Similarly we obtain
\begin{align*}
&[S_2, \tau_1^\pm(z)(\rho_1^+(p_1^{1/2}z')-\rho_1^-(p_1^{-1/2}z'))]=0\,,\\
&[S_3, \tau_1^\pm(z)(\rho_1^+(z')-\rho_1^-(z'))]=0\,,\\
&[S_0, (q_0^4\tau_1^+(p_1^{1/2}z)-\tau_1^-(p_1^{-1/2}z))\rho_1^\pm(z')]=0\,.
\end{align*}
It is now obvious that the formal sum
\begin{align*}
\sum_{k\in \Z}(\tau_1^+(z)-\tau_1^-(z))(\rho_1^+(p_1^{-k}z')-\rho_1^-(p_1^{-k}z'))
\end{align*}
commutes with $S_1,S_2,S_3$, and 
\begin{align*}
\sum_{k,l\in \Z}q_0^{-4l}(\tau_1^+(p_1^{-l}z)-\tau_1^-(p_1^{-l}z))(\rho_1^+(p_1^{-k}z')-\rho_1^-(p_1^{-k}z'))
\end{align*} 
commutes with $S_0,S_1,S_2,S_3$. 
Upon taking the residue at $z'=q_0q_2 z$ using \eqref{tau-rho},
the sum over $k\in\Z$ becomes one-sided:
\begin{align}
\sum_{k\ge0}\mathop{\mathrm{res}}_{w=p_1^{-k}q_0q_2z}(\tau_1^+(z)-\tau_1^-(z))(\rho_1^+(w)-\rho_1^-(w))\,.
\label{Tformal}
\end{align}

Rewriting the summand into a normal-ordered form, we arrive at \eqref{T312formal}, \eqref{T312formal2} with the coefficients given by 
\eqref{c312-1}--\eqref{c312-4}.

\subsection{Regularization}\label{subsec:reg}
Infinite sums of vertex operators such as \eqref{T312formal}, \eqref{T312formal2} are only symbolic expressions, 
and do not converge to operators on the Fock space. 
In order to give them a meaning, some regularization is necessary.

As an illustration, let us consider a simpler example of $\hat{\mathfrak{gl}}_{n,n}$ vector $qq$-character 
in Figure \ref{mm hat vector}. 
We choose $\int w^{-(n-1)\delta_{i,0}}S_i(w)dw$ ($0\le i\le 2n-1$) as the screening operators. 
The corresponding $qq$-current is a formal sum
\begin{align*}
T(z)=\sum_{i\in \Z}\sum_{k=0}^{n-1}
\bigl((q_1-q_1^{-1})q^{2k}V_{2k}(q^{-2ni}z)+(q_2-q_2^{-1})q^{2k+1}V_{2k+1}(q^{-2ni}z)\bigr)\,,
\end{align*}
which may be viewed as a Jackson integral. We regularize it by the contour integral
\begin{align*}
T_{\mathrm{reg}}=\int \sum_{k=0}^{n-1}\bigl((q_1-q_1^{-1})q^{2k}V_{2k}(w)+(q_2-q_2^{-1})q^{2k+1}V_{2k+1}(w)\bigr)
\frac{dw}{2\pi i w}\,.
\end{align*}
This formula is nothing but the first member of the integrals of motion
associated with the $W$ algebra of type $\hat \gl_{n,n}$ \cite{FJMV}.

Let us return to D$(2,1;\alpha)$. 
The contraction \eqref{tau-rho} has two series of simple poles on the $w$-plane:
\begin{align}
&w=q_0q_2p_1^{-k}z\quad (k\ge0),\label{first} 
\\
&w=q_0q_2^{-1}p_1^{-k+1}z\quad (k\ge0).\label{second}
\end{align} 
For simplicity of presentation let us assume that $|q_0q_2^{\pm1}|>1>|q_0q_2^{\pm1}p_1|$. 
We wish to interpret \eqref{Tformal} as the result of computing residues of a contour integral.
For that purpose, consider an integral of the form
\begin{align}
T_{\mathrm{reg}}^{312}(z)=
\int_{|w|=|z|}\bigl(\tau^+(z)-\tau_1^-(z)\bigr)\bigl(\rho^+_1(w)-\rho_1^-(w)\bigr)F(w/z)\frac{dw}{2\pi i w}\,.
\label{T312-reg}
\end{align}
We require the kernel function $F(w/z)$ to have two properties: it is a quasi-constant, i.e. $F(p_1w/z)=F(w/z)$, 
and the only poles of the integrand in the region $|w|>|z|$ are simple poles \eqref{first}. 
Inspection of the contractions \eqref{tau-rho} then leads us to the expression
\begin{align*}
F(w/z)=
\Bigl(\frac{w}{z}\Bigr)^{\mu}\frac{\Theta_{p_1}(q_0^{-1}q_2w/z)\Theta_{p_1}(p_1^\mu q_0^3q_2^{-1}w/z)}
{\Theta_{p_1}(q_0q_2w/z)\Theta_{p_1}(q_0q_2^{-1}w/z)}
\end{align*}
where $\mu$ is an arbitrary parameter. 
Formula \eqref{T312-reg} gives a well-defined operator on suitable sectors of 
the Fock space where the integrand comprises integral powers in $w$.
Collecting residues in $|w|>|z|$ and  ignoring the contribution from $w=\infty$, we recover 
 \eqref{Tformal} 
upto an irrelevant overall multiplicative constant. 

As opposed to the sum over $k$, no truncation takes place for the sum over $l\in\Z$. 
We interpret it simply replacing the sum by the integral 
\begin{align}
\widehat{T}_{\mathrm{reg}}^{312}=\int_{|z|=1} 
z^{\frac{2}{k_1+2}}T_{\mathrm{reg}}^{312}(z)\frac{dz}{z}\,.
\label{T312_reg}
\end{align}

Let us compare this formula with the deformed (non-local) integrals of motion associated with
the quantum toroidal $\mathfrak{gl}_2$ algebra \cite{FKSW}, \cite{FJM2}. The first member reads
\begin{align}
\mathbb{G}_{1,1}(\vartheta)&=\int\!\!\int_{|z|=|w|=1}
z^{\frac{2}{k_1+2}}\bigl(\tau_1^+(z)-\tau_1^-(z)\bigr)\bigl(\rho_1^+(w)-\rho_1^-(w)\bigr) 
\label{non-local}
\\
&\times
\Bigl(\frac{w}{z}\Bigr)^{1+\frac{h_{1,0}}{2(k_1+2)}}
\frac{\vartheta(q_0^{1+h_{1,0}/2}w/z)}{\Theta_{p_1}(q_0q_2w/z)\Theta_{p_1}(q_0q_2^{-1}w/z)}\,
\frac{dz}{z}\frac{dw}{w}\,,\nn
\end{align}
where 
\footnote{In \cite{FJM1}, the parameter $\bar{p}_1$ in \cite{FJM2} was set to $1$. For the commutativity with screening operators $S_i$, 
this has to be chosen rather as $q_0^2$.}
\begin{align*}
h_{1,0}=\frac{2(k_2+2)}{k_2-k_3}(s_{1,0}+s_{2,0})+\frac{2(k_3+2)}{k_2-k_3}(s_{1,0}+s_{3,0})\,,
\end{align*}
and
$\vartheta(z)$ is a holomorphic function on $\C^{\times}$ satisfying the quasi-periodicity
$\vartheta(p_1z)=p_1^{-1}z^{-2}\vartheta(z)$.
The space of such functions is two-dimensional. 
Using this freedom one can make the following choice:
\begin{align*}
\vartheta^{\pm}(q_0^{1+h_{1,0}/2}w/z)=
\Theta_{p_1}(q_0^{-1}q_2^{\pm 1}w/z)
\Theta_{p_1}(p_1q_0^{h_{1,0}}q_0^3q_2^{\mp 1} w/z)\,.
\end{align*}
With the identification $\mu=1+h_{1,0}/(2(k_1+2))$, formula \eqref{non-local} match \eqref{T312_reg} and its analog:
\begin{align}
\mathbb{G}_{1,1}(\vartheta^{+})=\widehat{T}^{312}_{\mathrm{reg}}\,,
\quad
\mathbb{G}_{1,1}(\vartheta^{-})=\widehat{T}^{213}_{\mathrm{reg}}\,,
\label{Th-reg}
\end{align}
where we ignore constant multiples.

Alternatively, one can take residues on the poles inside the circle $|w|<|z|$.
To this end we use the commutation relation (in the sense of analytic continuation of matrix elements)
\begin{align*}
\tau_1^{\epsilon_1}(z)\rho_1^{\epsilon_2}(w)=
\rho_1^{\epsilon_2}(w)\tau_1^{\epsilon_1}(z)\times
q_0^{-4}\Bigl(\frac{w}{z}\Bigr)^{\frac{2}{k_1+2}}
\frac{\Theta_{p_1}(q_0q_2^{-1}w/z)\Theta_{p_1}(q_0q_2w/z)}
{\Theta_{p_1}(q_0^{-1}q_2^{-1}w/z)\Theta_{p_1}(q_0^{-1}q_2w/z)}
\end{align*}
to rewrite \eqref{non-local} as 
\begin{align*}
\mathbb{G}_{1,1}(\vartheta)&=q_0^{-4}\int\!\!\int_{|z|=|w|=1}
w^{\frac{2}{k_1+2}}\bigl(\rho_1^+(w)-\rho_1^-(w)\bigr) \bigl(\tau_1^+(z)-\tau_1^-(z)\bigr)
\\
&\times
\Bigl(\frac{w}{z}\Bigr)^{1+\frac{h_{1,0}}{2(k_1+2)}}
\frac{\vartheta(q_0^{1+h_{1,0}/2}w/z)}{\Theta_{p_1}(q_0^{-1}q_2w/z)\Theta_{p_1}(q_0^{-1}q_2^{-1}w/z)}\,
\frac{dz}{z}\frac{dw}{w}\,.
\end{align*}
Choosing $\vartheta$ to be 
\begin{align*}
\widetilde{\vartheta}^{\pm}(q_0^{1+h_{1,0}/2}w/z)=
\Theta_{p_1}(q_0q_2^{\pm 1}w/z)
\Theta_{p_1}(p_1q_0^{1+h_{1,0}}q_2^{\mp 1} w/z)\,,
\end{align*}
computing residues at $w=p_1^k q_0^{-1}q_2^{\mp1}$, $k\ge0$, and ignoring the contribution from $w=0$, 
we find 
\begin{align}
\mathbb{G}_{1,1}(\widetilde{\vartheta}^{+})=\widehat{T}^{130}_{\mathrm{reg}}\,,
\quad
\mathbb{G}_{1,1}(\widetilde{\vartheta}^{-})=\widehat{T}^{031}_{\mathrm{reg}}\,.
\label{Th-reg1}
\end{align}
Together with the symmetry in $1,2,3$, \eqref{Th-reg}--\eqref{Th-reg1} give the regularization of 
the 12 vector $qq$-characters of type $\hat{\textrm D}(2,1;\alpha)$ in Section \ref{sec D hat}.

\subsection{Adjoint representation and its fusion}

For finite $qq$-characters, there is no issue of convergence. 
The current $T^{18}(z)=T_{\chi^{18}}(z)$
associated with the adjoint $qq$-character of D$(2,1;\alpha)$ depicted in Figure \ref{18 pic} reads
\begin{align*}
T^{18}(z)=\sum_{j=1}^{18} c^{18}_jV_{v_j}(z)\,,
\end{align*}
where
\begin{align*}
&c^{18}_1=q^{-3}\,,&& c^{18}_2=-q^{-2}\frac{[k_1+1]}{[k_1+2]}\,,&& c^{18}_5=q^{-1}\,, && c^{18}_8=-\frac{[k_1+1][k_2+1][k_3]}{[k_1+2][k_2+2][k_3+1]}\,,
\\
&c^{18}_{18}=q^3\,,&&c^{18}_{15}=-q^2\frac{[k_1+1]}{[k_1+2]}\,,&&c^{18}_{12}=q\,, &&
c^{18}_{11}=-\frac{[k_1+2][k_2+2][k_3+2]}{[k_1+1][k_2+1][k_3+1]}\,,
\end{align*}
with the notation $[x]=(q^x-q^{-x})/(q-q^{-1})$.
The rest of the coefficients $c^{18}_j$ are given by simultaneously permuting colors $1,2,3$ and 
$k_1,k_2,k_3$. 

The current $T^{18}(z)$ commutes with fermionic screening operators $S_i$, $i=1,2,3$.
It can be shown that $T^{18}(z)$ commutes also with the bosonic screening operators $\rho_i$, $i=1,2,3$,
in \cite{FJM1}.\footnote{In accordance with the definition \eqref{D21-screening-charge} the bosonic screening operators in \cite{FJM1}
should be redefined as $\rho_i=\int \rho_i(w)dw$, $\tau_i=\int w^{\frac{2}{k_i+2}}\tau_i(w)dw$. 
}
 Hence $T^{18}(z)$  belongs to what is termed the deformed $W$-algebra $\mathcal{W}{\textrm D}(2,1;\alpha)$, which is a deformation of the coset theory 
$(\mathfrak{sl}_2)_{k_1}\times(\mathfrak{sl}_2)_{k_2}/(\mathfrak{sl}_2)_{k_1+k_2}$.

One can generate further currents in $\mathcal{W}{\textrm D}(2,1;\alpha)$ by the fusion construction. 
For example, consider the product $T^{18}(z)T^{18}(w)$.
It has simple poles at $w=(q_0q_i)^{\pm2}z$, $i=1,2,3$, and $w=q_0^{\pm2}z$.
Their resiudes give the $qq$-currents $T^{66,1}(z)=T_{\chi^{66,1}}(z)$
and  $T^{130}(z)=T_{\chi^{130}}(z)$ associated with the $66$ and $130$ $qq$-characters 
in Section \ref{130 sec} and Figure \ref{66 pic}, respectively:
\begin{align*}
&T^{66,1}(z)=const. \mathop{\mathrm{res}}_{w=q^2_0q_1^2z}T^{18}(z)T^{18}(w)\,dw
\,,\\
&T^{130}(z)=const. \mathop{\mathrm{res}}_{w=q_0^2z}T^{18}(z)T^{18}(w)\,dw\,.
\end{align*}
It would be interesting to know if  $T^{18}(z)$ generates the entire $\mathcal{W}{\textrm D}(2,1;\alpha)$, 
and whether one can extract the spin four current from it in the conformal limit. 
These are the questions left for further investigation. 

\bigskip


{\bf Acknowledgments.\ }
The research of BF is supported by 
the Russian Science Foundation grant project 16-11-10316. 
MJ is partially supported by 
JSPS KAKENHI Grant Number JP19K03549. 
EM is partially supported by grants from the Simons Foundation  
\#353831 and \#709444.

 \end{document}